\documentclass[12pt]{article}
\usepackage{amsmath,amsthm,amssymb}
\usepackage{graphicx,epstopdf}
\usepackage{euscript,latexsym}

\newtheorem{theorem}{Theorem}[section]

\newtheorem{lemma}[theorem]{Lemma}

\newtheorem{proposition}[theorem]{Proposition}
\theoremstyle{definition}
\newtheorem{definition}[theorem]{Definition}
\newtheorem{example}[theorem]{Example}

\newtheorem{remark}[theorem]{Remark}

\numberwithin{equation}{section}                        

\newcommand{\mylabel}[1]{\label{#1}
            \ifx\undefined\stillediting
            \else \fbox{$#1$}\fi }
\newcommand{\BE}{\begin{equation}}
\newcommand{\BEQ}[1]{\BE\mathlabel{#1}} 
\newcommand{\EEQ}{\end{equation}}
\newcommand{\rfb}[1]{\mbox{\rm
   (\ref{#1})}\ifx\undefined\stillediting\else:\fbox{$#1$}\fi}

\newfont{\roma}{cmr10 scaled 1200}

\newcommand{\nline}  {{\mathbb N}}
\newcommand{\rline}  {{\mathbb R}}

\newcommand{\zline}  {{\mathbb Z}}

\newcommand{\Mscr} {{\cal M}}
\newcommand{\Nscr} {{\cal N}}
\newcommand{\mm}     {{\hbox{\hskip 0.5pt}}}
\newcommand{\m}      {{\hbox{\hskip 1pt}}}
\newcommand{\nm}     {{\hbox{\hskip -3pt}}}
\newcommand{\nmm}    {{\hbox{\hskip -1pt}}}
\newcommand{\bluff}  {{\hbox{\raise 26pt \hbox{\mm}}}}
\newcommand{\sbluff} {{\hbox{\raise  9pt \hbox{\mm}}}}
\renewcommand{\L}    {{\Lambda}}
\renewcommand{\l}    {{\lambda}}
\renewcommand{\o}    {{\omega}}
\newcommand{\Om}     {{\Omega}}
\newcommand{\e}      {{\varepsilon}}
\newcommand{\vp}     {{\varphi}}
\newcommand{\half}   {{\frac{1}{2}}}

\newcommand{\xx}     {{\rm\bf x}}
\newcommand{\dd}     {{\rm d\hbox{\hskip 0.5pt}}}

\newcommand{\FORALL} {{\hbox{$\hskip 11mm \forall \;$}}}
\newcommand{\rarrow} {\mathop{\rightarrow}}

\newcommand{\secp}{{\hbox{\hskip -7mm.\hskip 4mm}}}

\let\oldlabel=\label

\renewcommand{\label}[1]{\leavevmode\smash{\raise 10pt\llap
             {\fbox{\scriptsize#1}}}\oldlabel{#1}}
\newcommand{\mathlabel}[1]{\smash{\raise 9pt\llap
             {\scriptsize(#1)}}\label{#1}}

\renewcommand{\label}[1]{\oldlabel{#1}}
\renewcommand{\mathlabel}[1]{\label{#1}}

\newcommand{\bbm}[1]{\left[\begin{matrix} #1 \end{matrix}\right]}

\begin{document}

\renewcommand{\thefootnote}{\fnsymbol{footnote}}
\renewcommand{\thefootnote}{\fnsymbol{footnote}}
\newcommand{\footremember}[2]{%
   \footnote{#2}
    \newcounter{#1}
    \setcounter{#1}{\value{footnote}}%
}
\newcommand{\footrecall}[1]{%
    \footnotemark[\value{#1}]%
}
\makeatletter
\def\blfootnote{\gdef\@thefnmark{}\@footnotetext}
\makeatother

\begin{center}
{\Large \bf Almost global asymptotic stability of a \\[1ex]
grid-connected synchronous generator}\\[2ex]
Vivek Natarajan and George Weiss
\blfootnote{This work was partially supported by grant no. 800/14
of the Israel Science Foundation.}
\blfootnote{V. Natarajan (n.vivek.n@gmail.com) and G. Weiss
(gweiss@eng.tau.ac.il) are with the School of Electrical
Engineering, Tel Aviv University, Ramat Aviv, Israel, 69978,
Ph:+97236405164.}
\blfootnote{Preliminary versions of this paper have been presented at
the IEEE-CDC 2014, see \cite{NaWe:14} and at the IEEEI 2014, see
\cite{NaWe:14b}.}
\end{center}

{\leftskip 10mm {\rightskip 10mm \small \noindent {\bf Abstract.} We
study the global asymptotic behavior of a grid-connected constant
field current synchronous generator (SG). The grid is regarded as an
``infinite bus'', i.e. a three-phase AC voltage source. The generator
does not include any controller other than the frequency droop
loop. This means that the mechanical torque applied to this generator
is an affine function of its angular velocity. The negative slope of
this function is the frequency droop constant. We derive sufficient
conditions on the SG parameters under which there exist exactly two
periodic state trajectories for the SG, one stable and another unstable,
and for almost all initial states, the state trajectory of the SG
converges to the stable periodic trajectory (all the angles are measured
modulo $2\pi$). Along both periodic state trajectories, the angular
velocity of the SG is equal to the grid frequency. Our sufficient
conditions are easy to check computationally. An important tool in our
analysis is an integro-differential equation called the {\em exact
swing equation}, which resembles a forced pendulum equation and
is equivalent to our fourth order model of the grid-connected
SG. Apart from our objective of providing an analytical proof for a
global asymptotic behavior observed in a classical dynamical system, a
key motivation for this work is the development of synchronverters
which are inverters that mimic the behavior of SGs. Understanding the
global dynamics of SGs can guide the choice of synchronverter
parameters and operation. As an application we find a set of stable
nominal parameters for a 500\m kW synchronverter. \par}}

{\leftskip 10mm {\rightskip 10mm \small \noindent {\bf Key words.}
synchronous machine, infinite bus, almost global asymptotic stability,
forced pendulum equation, synchronverter, virtual inductor. \par}}

{\leftskip 10mm {\rightskip 10mm \small \noindent {\bf AMS
classification.} 34D23, 93D20, 94C99. \par}}

\section{\secp Introduction} \label{sec1}     

\ \ \ Synchronous generators (SGs), once synchronized to the power
grid, tend to remain synchronized even without any control unless very
strong disturbances destroy the synchronism - this is a feature that
enabled the development of the AC electricity grid at the end of the
XIX century. We investigate this feature by considering one
synchronous generator and analyzing its ability to synchronize when it
is connected to a much more powerful grid, so that this one generator
has practically no influence on the grid. Thus we model the grid as an
``infinite bus'', i.e. a three-phase AC voltage source. Following
standard practice, the prime mover (the engine that gives the
mechanical torque to the generator) is assumed to provide a torque of
the form $T_m-D_{p,\m{\rm droop}}\o$. Here $T_m>0$ is a mechanical
torque constant, $D_{p,\m{\rm droop}}>0$ is the frequency droop
constant (this is used to stabilize the utility grid) and $\o$ is the
angular velocity of the rotor. The question we address is: {\em under
what conditions will the state trajectory of a grid-connected SG,
driven by a prime mover as above, having a constant field current
(rotor current) and starting from an arbitrary initial state, converge
to a state of synchronous rotation?} (Synchronous rotation means a
constant difference between the grid angle and the SG rotor angle.)

The above question can be reformulated as a question of almost global
asymptotic stability of a SG model in a transformed coordinate system.
The importance of the stability of a grid-connected generator has been
recognized for a long time and this or closely related problems have
been studied, for instance, in \cite{DeCo:1969,GOBS:03,GrSt:94,Kun:94,
Park:1929,SaPa:97,YuWo:1967,ZhOh:09}. A full model of the SG consists
of the electrical equations governing the fluxes in the stator, rotor
and damper windings, along with the mechanical swing equation
governing the rotor dynamics. As far as we know, all the available
stability studies are based on some sort of simplification/reduction
of the full model obtained by: (i) reducing the full model to a lower
order (usually second or third order) non-linear system by
approximating the stator and the damper flux dynamics by static
equations and sometimes assuming constant rotor current, or (ii)
linearizing the full or the reduced order model around some
equilibrium point. Most of the studies that use reduced order models
focus on the local stability properties of the generator. A notable
exception in this regard is \cite{HaLeRa:87}, which considers various
reduced order SG models and derives sufficient conditions for every
state trajectory of the model to converge to an equilibrium point. The
paper \cite{CaTa:14} considers (among other things) a synchronous
machine connected to a three-phase AC voltage source having a constant
phase difference with respect to the machine angle. Such a dependent
voltage source is encountered in ``brushless DC motors''. The rotor
current is assumed to be constant and there are no damper windings,
and in this respect, their setup resembles ours. They prove the global
asymptotic stability of this system. This is an interesting problem,
but different from the stability of a SG connected to an infinite
bus. The paper \cite{FZOSS} proves (among other things) the global
asymptotic stability of a full (8th order) SG model when it is
connected to a linear resistive load (not a grid), using the formalism
of port-Hamiltonian systems.

In the present work, we study the global asymptotic stability
properties of a grid-connected generator without approximating the
stator flux dynamics (analyzing reduced order models that approximate
the stator flux dynamics can lead to incorrect conclusions about the
full model, see Remark \ref{red_model}). But we do restrict our attention
to the case where the rotor current is constant and the damper windings
are absent. We derive sufficient conditions on the SG parameters under
which there exist exactly two periodic state trajectories, one stable
and another unstable, and for almost all initial states, the state
trajectory of the SG converges to the stable periodic trajectory (all
the angles are measured modulo $2\pi$), see Theorem \ref{stab_cond2}.
Along both the periodic trajectories, the rotor angular velocity is
equal to the grid frequency. To derive the sufficient conditions, a
fourth order nonlinear time-invariant model for the grid-connected SG
is constructed in a transformed coordinate system using the Park
transformation in Section \ref{sec3}. In this coordinate system the
two periodic state trajectories of the SG are mapped into two distinct
points which are the unique stable and unstable equilibrium points of
the fourth order model. If for almost every initial state, the state
trajectory of the SG model converges to the stable equilibrium point,
then we call the model {\em almost globally asymptotically stable}. In
Section \ref{sec4} we derive an integro-differential equation called
the {\em exact swing equation} (ESE), which resembles a forced
pendulum equation and is equivalent to the fourth order SG model from
Section \ref{sec3}. Every trajectory of the fourth order model
converges to one of its equilibrium points if and only if every
trajectory of the ESE converges to one of two possible limit
points. We derive some new estimates for the asymptotic response of a
forced pendulum equation driven by a time-varying bounded forcing in Section
\ref{sec5}. Applying these estimates to the ESE, we define a nonlinear
map $\Nscr:(0,\Gamma]\to[0,\infty)$ in Section \ref{sec6} which (along
with $\Gamma$) depends on the SG parameters. We prove that if
$\Nscr(x) <x$ for all $x\in (0,\Gamma]$, then the SG is almost
globally asymptotically stable. For any given set of SG parameters, it
is easy to plot $\Nscr$ to verify if the above sufficient stability
condition is satisfied.

The inherent stability of networks of synchronous generators coupled
with various types of loads and power sources (such as inverters) is
currently an area of high interest and intense research, see for
instance \cite{CaTa:14,DoBu:12,DoBu:14,FZOSS,SaPa:14}. This is partly
due to the proliferation of power sources that are not synchronous
generators, which threatens the stability of the power grid. One
approach to addressing this threat has been the introduction of
synchronverters, see \cite{BeHe:07,Bro:15,DoChLi:15,DrVi:08,
Zh_etal:14, ZhWe:09,ZhWe:11}. A synchronverter consists of an inverter
(i.e. a DC to three-phase AC switched power converter) together with a
passive filter (inductors and capacitors) that behave towards the
power grid like a SG. A synchronverter has a rotor with inertia, a
field coil with inductance and three stator coils with inductance and
resistance, like a SG. But the field coils and the rotor in a
synchronverter are virtual, i.e. they are implemented in software,
while the stator coils are realized using the filter inductors. The
dynamical equations governing the SG and the synchronverter are the
same. Thus the synchronverter can be controlled like a SG, employing
droop control loops and other controllers. This makes the power grid
with inverters implemented as synchronverters easier to control using
well established algorithms developed for SGs.

One motivation for our study comes from the development of
synchronverters. In \cite{Zh_etal:14} an initial synchronization
algorithm was proposed that can be run (typically for some seconds)
before connecting the synchronverter to the grid. The purpose of this
algorithm is to ensure that the voltages generated by the inverter are
practically equal to the grid voltages. During this initial
synchronization stage, the filter inductors are not used. Instead, the
control algorithm creates virtual stator coils between the synchronous
internal voltage and the grid, which carry virtual currents, and the
initial synchronization is carried out using these virtual currents
instead of real currents. Thus, even very high virtual currents that
may arise as a transient phenomenon, do not cause any damage. A
natural question is: will this initial synchronization stage always
succeed? If we simplify this question by assuming a constant field
current and a constant grid frequency, then this question reduces to
the one addressed in this paper. We remark that it is possible to
construct an initial synchronization algorithm, using the results in
this work, that is guaranteed to succeed (the details of such an
algorithm are not included in this paper).

Our conclusions are relevant not only for the initial synchronization
stage, but also for finding a good choice of parameters for the
synchronverter. Indeed, our study shows that it is beneficial to have
stator coils with large inductance in a synchronverter. We shall
indicate in Section \ref{sec7} how to realize the effect of a
large inductor in the control algorithm of the synchronverter, without
actually using a large and expensive filter inductor in the hardware.
As an application, we find a set of stable nominal parameters for a
500kW synchronverter in Example \ref{500kW}.

The motivation for formulating the question of stability of a
grid-connected SG in a global setting (i.e. for arbitrary initial
states) comes from intensive simulations which indicate that for a range
of parameters the SG could be almost globally asymptotically stable.
We wanted to develop a rigorous analytical proof for this numerical
observation about a classical dynamical system, which turned out to be
very challenging. Our sufficient conditions for almost global asymptotic
stability seem to be conservative: according to simulations, there are
grid-connected SGs that do not satisfy our conditions, but nevertheless
appear to be almost globally asymptotically stable. Also, it is easy to
find such systems that have a locally stable equilibrium point but are
not almost globally asymptotically stable. It is more difficult, but
still possible, to find such systems whose equilibrium points are all
unstable. Examples of systems described above are in Section \ref{sec7}.

\section{\secp Model of a SG connected to an infinite bus}
\label{sec2} 

\ \ \ Detailed mathematical models for synchronous machines can be
found in \cite{Fitzgerald:03,GrSt:94,KoNa:04,Kun:94,Walker:94}. In
this section we will briefly derive the equations for a grid connected
synchronous generator, as required in this work, using the notation
and sign conventions in \cite{MaWe:15,ZhWe:11}. We consider a SG
with round (non-salient pole) rotor and, for the sake of simplicity,
assume that the generator has one pair of field poles. The generator
is ``perfectly built'', meaning that in each stator winding, the flux
caused by the rotor is a sinusoidal function of the rotor angle
$\theta$ (with shifts of $\pm 2\pi/3$ between the phases of course).
The rotor current $i_f>0$ is assumed to be constant (or equivalently,
the rotor is a permanent magnet). The stator windings are connected
in star, with no neutral connection, and there are no damper windings.

Figure 1 shows the structure of the SG being considered. The stator
windings have self-inductance $L>0$, mutual inductance $-M<0$ and
resistance $R_s>0$. (The typical value for $M$ is $L/2$.) We define
$L_s=L+M$. A current in a stator winding is considered positive if
it flows outwards (see Figure 1). The vectors $e=[\m e_a\ \ e_b\ \
e_c]^\top$, $v=[\m v_a\ \ v_b\ \ v_c]^\top$ and $i=[\m i_a\ \ i_b\ \
i_c]^\top$ are the electromotive force (also called the synchronous
internal voltage), stator terminal voltage and stator current,
respectively. The voltage at the (unconnected) center of the star is
denoted by $v_s$. Let $v^n=[\m v_s\ \ v_s\ \ v_s]^\top$. Then, using
$i_a+i_b+i_c=0$ (there is no neutral line), we have \vspace{-2.5mm}
\BEQ{voltam_Savtanal}
   L_s \dot i+R_s i \m=\m e-v+v^n \m. \vspace{-1mm}
\end{equation}
Note that if the synchronous generator is connected to the infinite
bus via an impedance that consists of a resistor and an inductor in
series, then these can be regarded as being parts of $R_s$ and $L_s$,
respectively.

$$\includegraphics[scale=0.25]{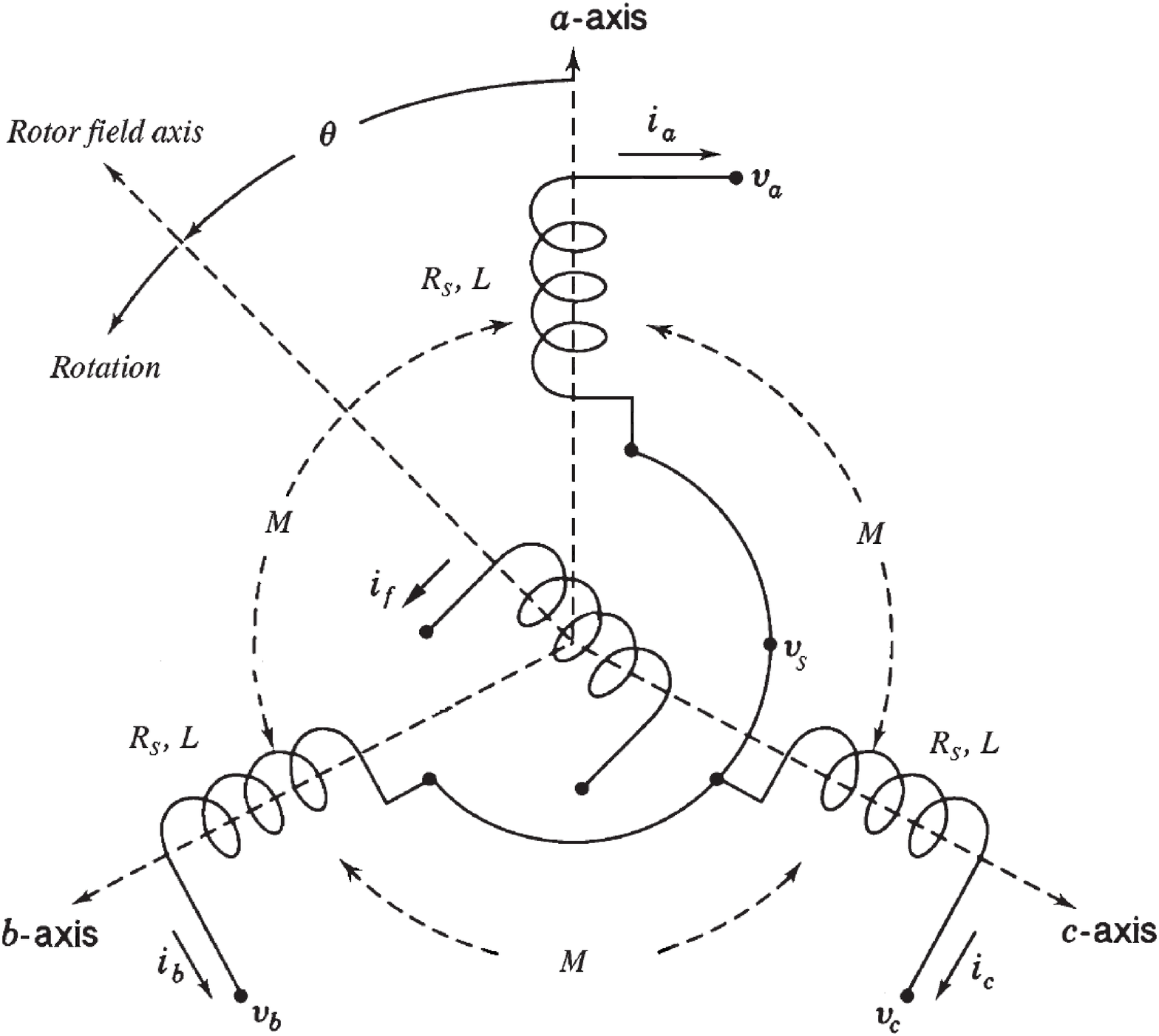} $$
\centerline{ \parbox{5.3in}{\vspace{3mm}
   Figure 1. Structure of an idealized three-phase round-rotor SG,
   modified from \cite[Fig. 3.4]{GrSt:94}. The rotor angle is
   $\theta$ and the field current is $i_f$.}}
\medskip

Denote the rotor angle by $\theta$ and the angular velocity by $\o$.
The power invariant version of the Park transformation is the unitary
matrix \vspace{-2mm}
$$ U(\theta) \m=\m \sqrt{\frac{2}{3}} \bbm{\cos\theta & \cos(\theta
   -\frac{2\pi}{3}) & \cos(\theta+\frac{2\pi}{3}) \\ -\sin\theta &
   -\sin(\theta-\frac{2\pi}{3}) & -\sin(\theta+\frac{2\pi}{3}) \\ 1/
   \sqrt 2 & 1/\sqrt 2 & 1/\sqrt 2} \m.$$
With the notation $e_{dq}=U(\theta)e$, $v_{dq}=U(\theta)v$, $v^n_{dq}
=U(\theta)v^n$ and $i_{dq}=U(\theta)i$, \rfb{voltam_Savtanal} can be
written as \vspace{-1mm}
\BEQ{dq_eq}
 L_s U(\theta)\dot i+R_s i_{dq} \m=\m e_{dq}-v_{dq}+v^n_{dq} \m.
\end{equation}
Let $e_{dq}=[\m e_d\ \ e_q\ \ e_0]^\top$, $v_{dq}=[\m v_d\ \
v_q\ \ v_0]^\top$ and $i_{dq}=[\m i_d\ \ i_q\ \ i_0]^\top$.
It is easy to check that if $x_{dq}=[\m x_d\ \ x_q\ \ x_0]^\top=
U(\theta)x$, then regardless of the physical meaning of $x$
\vspace{-2mm}
$$ \frac{\dd}{\dd t}\bbm{x_d\\ x_q\\ x_0} \m=\m U(\theta)\dot x + \o
   \bbm{x_q\\ -x_d\\ 0} \m. \vspace{-2mm}$$
This, the easily verifiable expression $v^n_{dq}=[\m 0\ \ 0 \ \
\sqrt{3}v_s]^\top$ and  \rfb{dq_eq} yield \vspace{-2mm}
\BEQ{Malaysia_plane}
   L_s \dot i_d \m=\m -R_s i_d + \o L_s i_q + e_d-v_d \m,\qquad
   L_s \dot i_q \m=\m -\o L_s i_d - R_s i_q + e_q-v_q \m.
\end{equation}
Note that $i_0=0$ since $i_a+i_b+i_c=0$ and hence $e_0=v_0-\sqrt{3}
v_s$. Since the rotor current $i_f$ is constant, it can be shown
that \vspace{-2mm}
\BEQ{sync_burned}
   e \m=\m M_f i_f\o \bbm{\sin\theta\\ \sin(\theta-\frac{2\pi}{3})\\
   \sin(\theta+\frac{2\pi}{3})} \m, \vspace{-2mm}
\end{equation}
where $M_f>0$ is the peak mutual inductance between the rotor winding
and any one stator winding (see \cite[equation (4)]{ZhWe:11}). This,
by a short computation, gives \vspace{-3mm}
\BEQ{rockets_from_Gaza}
   e_d \m=\m 0 \m,\qquad e_q \m=\m -m \o i_f \m, \vspace{-4mm}
\end{equation}
where $m=\sqrt{\frac{3}{2}}M_f$. The rotational dynamics of the
generator is governed by the equation \vspace{-2mm}
\BEQ{hut_in_Galil_reserved}
 J\dot\o \m=\m T_m - T_e - D_p\o \m,
\end{equation}
where $J>0$ is the moment of inertia of all the parts rotating with
the rotor, $T_m>0$ is a mechanical torque constant (see the explanations
further below), $T_e$ is the electromagnetic torque developed by the
generator (which normally opposes the movement) and $D_p>0$ is a
damping factor. $T_e$ can be found from energy considerations, see
for instance \cite[equation (7)]{ZhWe:11}: \vspace{-2mm}
$$ T_e \m=\m -m i_f i_q \m. \vspace{-2mm} $$
The constant $D_p$ is a sum of $D_{p,\m{\rm fric}}>0$ which accounts
for the viscous friction acting on the rotor and $D_{p,\m{\rm droop}}
>0$ which is created by a feedback, called the {\em frequency droop},
from $\o$ to the mechanical torque of the prime mover (as explained
in the cited references). The frequency droop increases the active
power in response to a drop of the grid frequency. Normally,
$D_{p,\m{\rm droop}}$ is much larger than $D_{p,\m{\rm fric}}$. The
actual active mechanical torque $T_a$ coming from the prime mover is
$T_m-D_{p,\m{\rm droop}}\o$. Substituting the expression for $T_e$
into \rfb{hut_in_Galil_reserved}, we obtain \vspace{-2mm}
\BEQ{Arye_Levinson_Karat}
   J\dot\o \m=\m m i_f i_q - D_p\o + T_m \m. \vspace{-2mm}
\end{equation}

The stator terminals are connected to the grid. Denote the grid
voltage magnitude and angle by $V$ and $\theta_g$, respectively.
By this we mean that the components of $v$ are
\vspace{-2mm}
$$ v_a \m=\m \sqrt{\frac{2}{3}}V \sin\theta_g \m,\qquad v_b \m=\m
   \sqrt{\frac{2}{3}}V \sin(\theta_g-\frac{2\pi}{3})\m, \qquad
   v_c \m=\m \sqrt{\frac{2}{3}}V \sin(\theta_g+\frac{2\pi}{3})\m.$$
Define the angle difference $\delta$, called {\em the power angle},
as \m $\delta=\theta-\theta_g$\m. Applying the Park transformation
to $v$, we get \vspace{-2mm}
$$ v_d \m=\m -V\sin\delta \m,\qquad v_q \m=\m -V\cos\delta \m.$$
Substituting this and \rfb{rockets_from_Gaza} into
\rfb{Malaysia_plane} gives
$$ L_s \dot i_d \m=\m -R_s i_d + \o L_s i_q + V\sin\delta \m,\qquad
   L_s \dot i_q \m=\m -\o L_s i_d-R_s i_q -m\o i_f+V\cos\delta\m.$$
Denoting $\o_g=\dot\theta_g$ (the grid frequency), it is clear from
the definition of $\delta$ that \vspace{-1mm}
\BEQ{smaller_K}
  \dot\delta \m=\m \o-\o_g \m. \vspace{-1mm}
\end{equation}
The last three equations together with \rfb{Arye_Levinson_Karat} can
be written in matrix form:
\BEQ{eq:SG} \m\hspace{-3mm}
  \bbm{L_s\dot{i_d}\\ L_s\dot{i_q}\\ J\dot\o\\ \dot\delta} \nm=\nm
  \bbm{-R_s & \o L_s & 0 & 0\\ -\o L_s & -R_s & -m i_f & 0\\ 0 & m
  i_f & -D_p & 0\\ 0 & 0 & 1 & 0} \bbm{i_d\\ i_q\\ \o\\ \delta} +
  \bbm{V\sin\delta\\ V\cos\delta\\ T_m \\ -\o_g} \m.
\end{equation}
The above fourth order nonlinear dynamical system, with state
variables $i_d,i_q,\o$ and $\delta$ is our model for a grid
connected synchronous generator. In a synchronous generator we may
control $i_f$ indirectly via the rotor voltage (this adds $i_f$ as
one more state variable to the system) and we may control also $D_p$
and $T_m$ (though not instantly). In a synchronverter we may control
$i_f$, $D_p$, $T_m$ and even $J$ instantly, but in this study they
are considered to be positive constants.

\section{\secp Equilibrium points of the SG model} \label{sec3}

\ \ \ The right side of the SG model \rfb{eq:SG} is a locally Lipschitz
function on its state space $\rline^4$. For any $(i_{d0},i_{q0},\o_0,
\delta_0)\in\rline^4$, it follows from standard wellposedness results
(see for instance \cite[Ch.~3]{Kha:02}) that there exists a unique
solution $(i_d,i_q,\o,\delta)$ for \rfb{eq:SG} defined on a maximal time
interval $[0,T_{\max})$, with $T_{\max}>0$, such that $(i_d(0),i_q(0),
\o(0),\delta(0))=(i_{d0},i_{q0},\o_{0},\delta_{0})$. We will show, via
contradiction, that $T_{\max}=\infty$. To this end, suppose that
$T_{\max}$ is finite. For each $t\in[0,T_{\max})$ define $W(t)= (L_s
i_d^2(t)+ L_s i_q^2(t)+J\o^2(t))/2$. Then
$$  \dot W(t) \m=\m -R_s(i_d^2(t)+i_q^2(t))-D_p\o^2(t)+Vi_d(t)\sin
    \delta(t)+Vi_q(t)\cos\delta(t)+T_m\o(t) $$
for all $t\in[0,T_{\max})$. Define $C=V^2/(2R_s)+T_m^2/(4D_p)$.
Clearly
$$  \dot W(t) \m\leq\m -R_s\left(|i_d(t)|-\frac{V}{2R_s}\right)^2
    -R_s\left(|i_q(t)|-\frac{V}{2R_s}\right)^2 - D_p \left(|\o(t)|-
    \frac{T_m}{2D_p}\right)^2 + C $$
which shows that if either $|i_d(t)|$, $|i_q(t)|$ or $|\o(t)|$ is
sufficiently large, then $\dot W(t)<0$. In other words, if $W(t)$ is
sufficiently large, then $\dot W(t)<0$. Therefore $W$ (and hence also
$i_d$, $i_q$ and $\o$) are bounded on $[0,T_{\max})$. Since $T_{\max}$
is finite, it follows from \rfb{smaller_K} that $\delta$ must also be
bounded on $[0, T_{\max})$. Hence $(i_d,i_q,\o,\delta)$ are bounded
functions on $[0,T_{\max})$, which contradicts \cite[Corollary II.3]
{JaWe:09}. Therefore $T_{\max}=\infty$. So for all initial conditions
there exists a unique global (in time) solution for \rfb{eq:SG}.

Denote $p=R_s/L_s$. Let the angle $\phi\in(0,\pi/2)$ be determined by
the equations
\BEQ{phi_def}
   \sin\phi \m=\m \frac{\o_g}{\sqrt{p^2+\o_g^2}} \m, \qquad
   \cos\phi \m=\m \frac{p}{\sqrt{p^2+\o_g^2}} \m.
\end{equation}
Any equilibrium point $(i_d^e,i_q^e,\o^e,\delta^e)$ of \rfb{eq:SG}
must satisfy
\BEQ{fsquared1}
  \o^e \m=\m \o_g \m, \qquad i_q^e \m=\m \frac{D_p\o_g-T_m}{mi_f},
  \qquad i_d^e \m=\m \frac{\o_g (D_p\o_g-T_m)}{m i_f p} +
  \frac{V\sin\delta^e}{R_s}  \m,
\end{equation}
\BEQ{fsquared2}
  \cos(\delta^e+\phi) \m=\m \frac{(D_p\o_g-T_m)}{mi_f} \frac{L_s\sqrt
  {p^2+\o_g^2}}{V} + \frac{mi_f\o_g p}{V\sqrt{p^2+\o_g^2}} \m.
\end{equation}

Denote the right side of \rfb{fsquared2} by $\L$. Depending on $|\L|$,
\rfb{fsquared2} has either zero, one or two solutions, modulo $2\pi$.
For typical sets of SG parameters $|\L|<1$ and \rfb{fsquared2} has two
solutions $\delta^{e,1}=\l-\phi$ and $\delta^{e,2}=-\l-\phi$.  Here
$\l\in(0,\pi)$ is such that $\cos\l=\L$. Corresponding to these two
solutions, two equilibrium points $(i_d^{e,1},i_q^e,\o_g,\delta^{e,
1})$ and $(i_d^{e,2},i_q^e,\o_g,\delta^{e,2})$ for \rfb{eq:SG} can be
determined using \rfb{fsquared1}. If $(i_d^e,i_q^e, \o_g,\delta^e)$ is
an equilibrium point for \rfb{eq:SG}, then so is $(i_d^e,i_q^e,\o_g,
\delta^e+2k\pi)$ for any integer $k$. Therefore, when $|\L|<1$ there
are in fact two sequences of equilibrium points for \rfb{eq:SG} in
$\rline^4$. In general depending on $|\L|$, like the pendulum equation
with constant forcing, \rfb{eq:SG} has either zero, one or two
sequences of equilibrium points and in any such sequence the last
component $\delta$ differs by an integer multiple of $2\pi$.

An equilibrium point of \rfb{eq:SG} is called {\em locally
exponentially stable} (in short: stable) if the linearization
of the system around this point is exponentially stable.

The linearization of \rfb{eq:SG} around an equilibrium point
$(i_d^e,i_q^e,\o_g,\delta^e)$ is
\BEQ{linearized_SG} \m\hspace{-3mm}
  \bbm{\dot{x_1}\\ \dot{x_2}\\ \dot x_3\\ \dot x_4} \nm=\nm
  \bbm{-p & \o_g & i_q^e & (V\cos\delta^e)/L_s \\ -\o_g & -p &
  -i_d^e-m i_f/L_s & -(V\sin\delta^e)/L_s\\ 0 & m i_f/J & -D_p/J & 0\\
  0 & 0 & 1 & 0} \bbm{x_1\\ x_2\\ x_3\\ x_4} \m,
\end{equation}
where $x_1=i_d-i_d^e$, $x_2=i_q-i_q^e$, $x_3=\o-\o_g$ and $x_4=\delta-
\delta^e$. The characteristic polynomial of the matrix in
\rfb{linearized_SG} is $s^4+a_3 s^3+a_2 s^2+a_1 s+a_0$, where
\vspace{-1mm}
\BEQ{seedless_grapes}
   a_0 \m=\m \frac{mi_f V\sqrt{p^2\sbluff+\o_g^2}}{J L_s} \sin
   (\delta^e + \phi) \m.
\end{equation}
The equilibrium point $(i_d^e,i_q^e,\o_g,\delta^e)$ is stable if all
the roots of the above characteristic polynomial are in the open left
half complex plane. For this to occur it is necessary (but not
sufficient) that $a_0,a_1,a_2,a_3>0$. Note that if
$(i_d^e,i_q^e,\o_g,\delta^e)$ is a stable (or unstable) equilibrium
point, then so is $(i_d^e,i_q^e,\o_g,\delta^e+2k\pi)$ for any
$k\in\zline$. When $|\L|<1$, the sign of $a_0$ when $\delta^e=
\delta^{e,1}$ is opposite the sign of $a_0$ when $\delta^e=\delta^
{e,2}$, so that \rfb{eq:SG} has at least one sequence of unstable
equilibrium points. It is also possible that \rfb{eq:SG} has two
sequences of unstable equilibrium points and no stable equilibrium
point (see end of Section \ref{sec7} for an example). Apart from
equilibrium points, simulations show that when $|\L|<1$, \rfb{eq:SG}
can have attracting periodic orbits (see Section \ref{sec7}). Hence
the global phase portrait of \rfb{eq:SG} can be quite complicated.

\begin{remark} \label{night_run}
Recall the currents $i=[\m i_a\ \ i_b\ \ i_c]^\top$ and $i_{dq}=[\m
i_d\ \ i_q\ \ 0]^\top$ and the Park transformation $U(\theta)$ from
Section \ref{sec2}. Since $i=U(\theta)^\top i_{dq}$ and
$\theta=\delta+\theta_g$, it follows that each equilibrium point of
\rfb{eq:SG} corresponds to a periodic state trajectory of the
grid-connected SG, if we use the state variables $(i_a,i_b,\o,\theta)$
(with $i_c=-i_a-i_b$). This periodic trajectory is stable (unstable)
if the equilibrium point of \rfb{eq:SG} is stable (unstable). From the
earlier discussion we get that when $|\L|<1$, if we measure all the
angles modulo $2\pi$, then the grid-connected SG with state variables
$(i_a,i_b,\o,\theta)$ has two unique periodic state trajectories and
at least one of them is unstable.
\end{remark}

\begin{remark} \label{red_model}
Often in the literature on the control of power systems, the stator
currents $i_d$ and $i_q$ are viewed as fast variables and (using
singular perturbation theory) algebraic expressions are derived for
them. If we follow this approach then, by substituting an algebraic
expression for $i_q$ in the differential equation
\rfb{Arye_Levinson_Karat}, we get a second order nonlinear differential
equation in $\delta$ as a reduced order approximation for the SG
model \rfb{eq:SG}. When $|\L|<1$, unlike the SG model, this
nonlinear equation always has one sequence of stable equilibrium
points and one sequence of unstable equilibrium points. So the SG
model and its second order approximation can exhibit fundamentally
different local and global dynamics for some SG parameters.
This suggests that any controller designed using a reduced order
model that approximates $i_d$ and $i_q$ must be validated for its
performance on the full model.
\end{remark}

\begin{definition} \label{aGAS_defn}
The SG model \rfb{eq:SG} is {\em almost globally asymptotically
stable} if all its state trajectories, except those starting from a
set of measure zero and converging to an unstable equilibrium point,
converge to a stable equilibrium point.
\end{definition}

Note that this definition allows multiple stable and unstable
equilibrium points, but it does not allow limit cycles or unbounded
state trajectories.

Extensive simulations suggest that for a range of SG parameters
\rfb{eq:SG} is almost globally asymptotically stable (aGAS). Our
objective is to develop a practical test for verifying if for
a given set of SG parameters \rfb{eq:SG} is aGAS. In this regard, our
main result is Theorem \ref{stab_cond2} (also see Remark \ref{CDC_thm}).
Clearly if \rfb{eq:SG} is aGAS, then irrespective of initial conditions
the SG rotor eventually synchronizes with the grid.

\begin{definition} \label{hyper}
An equilibrium point $(i_d^e,i_q^e,\o_g,\delta^e)$ of \rfb{eq:SG} is
{\em hyperbolic} if all the eigenvalues of the matrix on the right
side of \rfb{linearized_SG} have non-zero real part.
\end{definition}

For typical SG parameters, all the equilibrium points of \rfb{eq:SG}
are hyperbolic.

\begin{lemma} \label{stable_manifold}
If all the equilibrium points of \rfb{eq:SG} are hyperbolic and
every trajectory of \rfb{eq:SG} converges to some equilibrium point,
then \rfb{eq:SG} is aGAS.
\end{lemma}

\begin{proof}
By assumption \rfb{eq:SG} has equilibrium points and so $|\L|\leq1$
($|\L|$ is defined below \rfb{fsquared2}). If $|\L|=1$, then for each
equilibrium point $(i_d^e,i_q^e,\o_g,\delta^e)$ of \rfb{eq:SG} we have
$a_0=0$ ($a_0$ is introduced below \rfb{linearized_SG}) meaning that
the equilibrium point is not hyperbolic, contradicting the assumption
in the lemma. Thus we can conclude that $|\L|<1$. From our earlier
discussion, we get that \rfb{eq:SG} has a sequence of unstable
equilibrium points. Let $z^e=(i_d^e,i_q^e,\o_g,\delta^e)$ be an
unstable equilibrium point. It then follows from the stable manifold
theorem \cite[Theorem 1.7.2]{Szl:84} that the set of initial
conditions for which the trajectory of \rfb{eq:SG} converges to $z^e$
is the image of a $C^1$ injective map from $\rline^k\to\rline^4$, with
$k<4$. Using Sard's theorem \cite[Theorem 4.1]{Sar:42} we conclude
that this set, called the stable manifold of $z^e$, has Lebesgue
measure zero. Let $\Mscr$ be the union of the stable manifolds of all
the unstable equilibrium points of \rfb{eq:SG}. Since the set of
unstable equilibrium points is countable, $\Mscr$ has measure zero.
Since every trajectory of \rfb{eq:SG} converges to an equilibrium
point, it follows that \rfb{eq:SG} must have a sequence of stable
equilibrium points and all trajectories of \rfb{eq:SG} except those
starting from $\Mscr$ converge to these stable equilibrium points.
\end{proof}

If $(i_d,i_q,\o,\delta)$ is the solution of \rfb{eq:SG} for the
initial state $(i_d(0),i_q(0),\o(0),\delta(0))$, then $(i_d,i_q,\o,
\delta+2\pi)$ is the solution for the initial state $(i_d(0),i_q(0),
\o(0),\delta(0)+2\pi)$. Thus, in the terminology of \cite[Definition
2.4.1]{LePoSm:96}, \rfb{eq:SG} is a pendulum-like system.

\section{\secp An exact swing equation for the SG} \label{sec4}

\ \ \ In this section, starting from \rfb{eq:SG} we will derive an
integro-differential equation governing the power angle $\delta$, that
resembles the nonlinear pendulum equation with forcing. It is a
version of the classical swing equation (see \cite{Kun:94,ZhOh:09})
obtained by using the precise expressions for the mechanical torque
and the electrical torque.

Recall $p=R_s/L_s$. The first two equations in \rfb{eq:SG} can then
be written as
\BEQ{eq:currents}
  \bbm{\dot{i_d} \\ \dot{i_q}} \m=\m \bbm{-p & \omega\\
  -\o & -p} \bbm{i_d \\ i_q}-\bbm{0 \\ \frac{mi_f\o}{L_s}} +
  \frac{V}{L_s}\bbm{\sin\delta \\ \cos\delta} \m.
\end{equation}
We regard $\o$ and $\delta$ as continuous exogenous signals in
\rfb{eq:currents}. Therefore \rfb{eq:currents} is a linear
time-varying system with state matrix
$$  A(t) \m=\m  \bbm{-p & \o(t)\\ -\o(t) & -p} \m.  $$
Clearly $A(t_1)A(t_2) \m=\m A(t_2)A(t_1)$ for all $t_1, t_2 \geq0$.
So an explicit expression for the state transition matrix $\Phi(t,
\tau)$ generated by $A$ can be computed to be
\begin{align*}
  \Phi(t,\tau) &\m=\m e^{\int_\tau^t A(\sigma)\dd \sigma} \m=\m
    e^{\bbm{-p(t-\tau) & \int_\tau^t\omega(\sigma)\dd\sigma
    \\ -\int_\tau^t\omega(\sigma)\dd\sigma & -p(t-\tau)}} \\[4pt]
  &=\m e^{-p(t-\tau)} \bbm{\cos\left(\int_\tau^t\omega(\sigma)
    \dd\sigma\right) & \sin\left(\int_\tau^t\omega(\sigma)\dd\sigma
    \right)\\ -\sin\left(\int_\tau^t\omega(\sigma)\dd\sigma\right) &
    \cos\left(\int_\tau^t\omega(\sigma)\dd\sigma\right)}
    \FORALL t,\tau\geq0 \m.
\end{align*}
For any initial state $[i_d(0)\ \ i_q(0)]^\top$ and some functions
$\delta$ and $\omega$, the unique solution of \rfb{eq:currents} is
given by the expression
\BEQ{curr_soln}
  \bbm{i_d(t)\\i_q(t)} \m=\m \Phi(t,0)\bbm{i_d(0)\\i_q(0)}+\int_0^t
  \Phi(t,\tau) \left(\bbm{0\\-\frac{mi_f\o(\tau)}{L_s}} + \frac{V}
  {L_s}\bbm{\sin\left(\delta(\tau)\right) \\ \cos\left(\delta(\tau)
  \right)}\right) \dd\tau
\end{equation}
for each $t\geq0$. The first term under the integral in \rfb{curr_soln},
sans the constant $\frac{-mi_f}{L_s}$, can be written as
$$  \hspace{-25mm}\int_0^t\Phi(t,\tau)\bbm{0 \\ \o(\tau)}\dd\tau \m=\m
    \int_0^t e^{-p(t-\tau)}\bbm{\sin\left(\int_{\tau}^t\o(\sigma)\dd
    \sigma\right)\o(\tau) \\ \cos\left(\int_{\tau}^t\o(\sigma)\dd\sigma
    \right)\o(\tau)} \dd\tau$$
\BEQ{rain}
   \hspace{1mm}=\m e^{-p t}\bbm{e^{p\tau}\cos\left(\int_{\tau}^t\o
   (\sigma)\dd\sigma\right) \\ -e^{p\tau}\sin\left(\int_{\tau}^t
   \o(\sigma)\dd\sigma\right)}_{\tau=0}^{\tau=t} + p\int_0^t e^{-p(t-
   \tau)} \bbm{-\cos\left(\int_{\tau}^t \o(\sigma)\dd\sigma\right) \\
   \sin\left(\int_{\tau}^t\o(\sigma)\dd\sigma\right)} \dd\tau \m.
   \vspace{2mm}
\end{equation}
Using the expression $\delta(\tau)=\delta(0)+\int_0^\tau\o(\sigma)\dd
\sigma-\o_g\tau$ for all $\tau\geq0$, the second term under the integral
in \rfb{curr_soln}, sans the constant $\frac{V}{L_s}$, can be written as
\vspace{-3mm}
$$  \hspace{-15mm}\int_0^t\Phi(t,\tau)\bbm{\sin\left(\delta(\tau)\right)
    \\ \cos\left(\delta(\tau)\right)}\dd\tau \m=\m \int_0^t
    e^{-p(t-\tau)} \bbm{\sin\left(\int_\tau^t\o(\sigma)\dd\sigma+\delta
    (\tau)\right)\\ \cos\left(\int_\tau^t\o(\sigma)\dd\sigma+\delta
    (\tau)\right)}\dd\tau $$
$$  \hspace{-35mm}=\m \int_0^t e^{-p(t-\tau)} \bbm{\sin\left(\int_0^t
    \o(\sigma)\dd\sigma+\delta(0)-\o_g\tau\right) \\ \cos\left(\int_0^t
    \o(\sigma)\dd\sigma+\delta(0)-\o_g\tau\right)}\dd\tau $$
$$  \hspace{-32mm}=\m \frac{p e^{-p t}}{\left(p^2+\o_g^2 \right)}\bbm
    {e^{p\tau}\sin\left(\int_0^t\o(\sigma)\dd\sigma+\delta(0)-\o_g\tau
    \right)\\ e^{p\tau}\cos\left(\int_0^t\o(\sigma)\dd\sigma+\delta(0)
    -\o_g\tau\right)}_{\tau=0}^{\tau=t}$$
$$  +\m\frac{\o_g e^{-p t}}{\left(p^2+\o_g^2\right)}\bbm{e^{p\tau}\cos
    \left(\int_0^t\o(\sigma)\dd\sigma+\delta(0)-\o_g\tau\right)\\ -e^{p
    \tau}\sin\left(\int_0^t\o(\sigma)\dd\sigma+\delta(0)-\o_g\tau\right)
    }_{\tau=0}^{\tau=t} \m. \vspace{1mm}$$
Using the angle $\phi$ introduced in \rfb{phi_def}, the above equation
can be written as \vspace{-2mm}
$$  \int_0^t\Phi(t,\tau)\bbm{\sin\left(\delta(\tau)\right) \\ \cos
    \left(\delta(\tau)\right)}\dd\tau = \frac{e^{-p t}}{\sqrt{p^2+
    \o_g^2}} \bbm{e^{p\tau}\sin\left(\int_0^t\o(\sigma)\dd\sigma+\delta
    (0)\nm-\o_g\tau+\phi\right)\\ e^{p\tau}\cos\left(\int_0^t\o(\sigma)
    \dd\sigma + \delta(0)\nm-\o_g\tau+\phi\right)}_{\tau=0}^{\tau=t}.$$
Putting together \rfb{curr_soln}, \rfb{rain} and the last equation,
and using the notation \vspace{-2mm}
\BEQ{iv_defn}
   i_v \m=\m \frac{V}{L_s \sqrt{p^2+\o_g^2}} \m,\vspace{-2mm}
\end{equation}
we obtain that for all $t\geq0$ \vspace{-2mm}
$$ \hspace{-35mm} i_q(t) \m=\m i_v \cos\left( \int_0^t\o(\sigma)\dd
   \sigma + \delta(0)-\o_g t+\phi\right) $$
$$ -\frac{mi_f p}{L_s}\int_0^t e^{-p (t-\tau)} \sin\left( \int
_{\tau}^t \o(\sigma)\dd\sigma\right)\dd\tau+e^{-p t} f(t) \m,$$
where \vspace{-4mm}
$$ f(t) \m=\m -\sin\left(\int_0^t\o(\sigma)\dd\sigma\right) i_d(0)
   + \cos\left(\int_0^t \o(\sigma)\dd\sigma\right)i_q(0) \qquad\m$$
\BEQ{f_defn}
  \m\hspace{30mm} - \frac{mi_f}{L_s} \sin\left(\int_0^t \o(\sigma)\dd
  \sigma\right) - i_v \cos\left(\int_0^t\o(\sigma) \dd\sigma +
  \delta(0) + \phi\right) .
\end{equation}
Clearly $f$ is a continuous function of time that depends on $\o$, but
nevertheless can be bounded with a constant independent of $\o$.
Substituting for $i_q(t)$ in the equations for $\o$ and $\delta$ in
\rfb{eq:SG} we obtain the following integro-differential equation for
$\delta(t)$: \vspace{-2mm}
$$ \hspace{-16mm} J\ddot\delta(t) + D_p\dot\delta(t) - m i_f i_v \cos
   \left(\delta(t)+\phi\right) \m=\m T_m-D_p\o_g$$
$$ \m\qquad - \frac{m^2 i_f^2 p}{L_s}\int_0^t e^{-p (t-\tau)}
   \sin\left(\int_{\tau}^t\o(\sigma)\dd\sigma\right)\dd\tau
   + m i_f e^{-p t} f(t) \m.$$
If we introduce the new variable $\eta$ by \vspace{-2mm}
\BEQ{psi_defn}
  \eta(t) \m=\m \frac{3\pi}{2} + \delta(t) + \phi \vspace{-1.5mm}
\end{equation}
so that $\dot\eta(t)=\o(t)-\o_g$, then the above equation becomes
\vspace{-1mm}
$$ \hspace{-12mm} J\ddot\eta(t) + D_p\dot\eta(t) + m i_f i_v \sin
   \eta(t) \m=\m T_m-D_p\o_g + mi_f e^{-p t} f(t)\vspace{-2mm}$$
\BEQ{SG_Pend}
   \m\quad - \frac{m^2 i_f^2 p}{L_s} \int_0^t e^{-p (t-\tau)}
   \sin\left[\eta(t)-\eta(\tau)+\o_g(t-\tau)\right] \dd\tau \m.
\end{equation}
We will refer to \rfb{SG_Pend} as the {\em exact swing equation}
(ESE). For all initial conditions $(\eta(0),\dot\eta(0))$ and
every function $f$ given by \rfb{f_defn} for some $i_d(0)$ and
$i_q(0)$, there exists a unique global solution $(\eta,\dot\eta)$
for ESE. Indeed $(\eta,\dot\eta)=(3\pi/2+\delta+\phi,\dot\delta)$,
where $\delta$ is such that $(i_d,i_q,\o,\delta)$ is the unique
solution of \rfb{eq:SG} for the initial condition $(i_d(0),i_q(0),
\dot\eta(0)+\o_g,\eta(0)-\phi-3\pi/2)$. Clearly there is a 1-1
correspondence between the solutions of \rfb{eq:SG} and the solutions
of \rfb{SG_Pend} when $f$ is given by \rfb{f_defn}.

The integral in \rfb{SG_Pend} may be regarded as the output of a
first order low-pass filter (with corner frequency $p$) driven by a
bounded input, so that it is bounded. If we regard the right side
of \rfb{SG_Pend} as a bounded exogenous function, then \rfb{SG_Pend}
is a forced pendulum equation. In the next section, we derive
certain bounds to quantify the asymptotic response of forced
pendulum equations. These bounds are applied to \rfb{SG_Pend} in
Section \ref{sec6} to establish the main result of this paper.

\section{\secp Asymptotic response of a forced pendulum} \label{sec5}
\vspace{-1mm} 

\ \ \ Consider the forced pendulum equation \vspace{-1mm}
\BEQ{pendulum}
  \ddot\psi(t)+\alpha\dot\psi(t)+\sin\psi(t) \m=\m \beta+\gamma(t)
  \FORALL t\geq0 \m, \vspace{-1mm}
\end{equation}
where $\alpha>0$ and $\beta\in\rline$ are constants and $\gamma\in
L^\infty([0,\infty);\rline)$ is a continuous function of the time $t$
satisfying $\|\gamma\|_{L^\infty}<d$ for some $d\in\rline$. We assume
that $|\beta|+d<1$. Define the angles $\psi_1,\psi_2\in(-\pi/2,
\pi/2)$ so that \vspace{-1mm}
\BEQ{t1-t2}
  \sin\psi_1 \m=\m \beta+d \m, \qquad \sin\psi_2 \m=\m \beta-d \m.
  \vspace{-1mm}
\end{equation}
For any initial state $(\psi(0),\dot\psi(0))$, there is a unique
solution $\psi$ to \rfb{pendulum} on a maximal time interval $[0,
t_{\max})$, according to standard results on ordinary differential
equations (ODEs), see for instance \cite[Ch.~3]{Kha:02}. Since $|
\beta+\gamma(t)-\sin\psi|<|\beta|+d+1$ for all $t\geq0$ and $\alpha>
0$, we get from \rfb{pendulum} (by looking at the linear ODE $\dot z
+\alpha z=u$, with $z=\dot\psi$) that \vspace{-2mm}
\BEQ{JF_election}
  \sup_{t\in[0,\m t_{\max})}|\dot\psi(t)| \m<\m |\dot\psi(0)| +
  \frac{|\beta|+d+1}{\alpha} \m. \vspace{-2mm}
\end{equation}
Hence $|\dot\psi(t)|$ cannot blow up to infinity in a finite time,
and hence the same holds for $|\psi(t)|$. From \cite[Corollary II.3]
{JaWe:09} it follows that $t_{\max}=\infty$. Since $\gamma$ is a
continuous function of time, we get from \rfb{pendulum} that the
function $\psi$ is of class $C^2$.

The aim of this section is to show that if 
$\alpha$ is sufficiently large, then the solutions $\psi$ of
\rfb{pendulum} are eventually confined to a narrow interval, see
Theorem \ref{main_result_sec5}.

We will often regard the solution $(\psi,\dot\psi)$ of \rfb{pendulum}
as a curve in the phase plane. Recall that in the phase plane the
angle $\psi$ is on the $x$-axis and the angular velocity $\dot\psi$ is
on the $y$-axis. The curve corresponding to $(\psi,\dot\psi)$
satisfies the ODE \vspace{-2mm}
\BEQ{phase_eqn}
  \frac{\dd\dot\psi(t)}{\dd \psi(t)} \m=\m -\alpha + \frac{\beta
  +\gamma(t)-\sin\psi(t)}{\dot\psi(t)} \qquad {\rm whenever} \qquad
  \dot\psi(t)\neq0 \m.\vspace{-2mm}
\end{equation}
Suppose that the curve $(\psi,\dot\psi)$ passes through a point
$(\psi_0,\dot\psi_0)$ in the phase plane. We use the notation $\dot
\psi|_{\psi=\psi_0}$ to denote $\dot\psi_0$ provided there is no
ambiguity. We refer to Figures 2 and 3 for typical state trajectory
curves in the phase plane.

The following two lemmas establish a monotonicity in the behavior of
the solutions to \rfb{pendulum} with respect to the infinity norm of
the forcing term.

\begin{lemma} \label{traj_above}
Consider the pendulum equation \vspace{-1mm}
\BEQ{pend_max}
   \ddot\psi_p(t) +\alpha\dot\psi_p(t) +\sin\psi_p(t) \m=\m \beta+d
   \FORALL t\geq0 \m.
\end{equation}
Let \m $\psi$ and $\psi_p$ be the solutions of \rfb{pendulum} and
\rfb{pend_max}, respectively, for the initial conditions $\psi(0)=
\psi_p(0)=\psi_0$ and $\dot\psi(0)=\dot\psi_p(0)=\dot\psi_0$ and
recall that $\|\gamma\|_{L^\infty}<d$. Suppose that $\dot\psi(t)\geq
0$ for all $t\in[0,\tau]$ and $\psi(\tau)\neq\psi(0)$ for some $\tau>0
$. Then the curve $(\psi _p,\dot\psi_p)$ lies above the curve $(\psi,
\dot\psi)$ in the phase plane on the angle interval $(\psi(0),\psi
(\tau))$, i.e. for each $\varphi\in(\psi(0),\psi(\tau))$ we have $\dot
{\psi_p}|_{{\psi_p}= \varphi}>\dot\psi|_{\psi=\varphi}$.
\end{lemma}

\begin{proof}
The curve $(\psi_p,\dot\psi_p)$ satisfies the following ODE in the
variable $\psi_p(t)$: \vspace{-2mm}
\BEQ{phase_eqn_max}
  \frac{\dd\dot\psi_p(t)}{\dd \psi_p(t)} \m=\m -\alpha + \frac{\beta
  +d-\sin\psi_p(t)}{\dot\psi_p(t)} \qquad {\rm whenever} \qquad
  \dot\psi_p(t)\neq0 \m. \vspace{-2mm}
\end{equation}
First we claim that there exists $\sigma\in(0,\tau)$ such that $\psi_p
(\sigma)<\psi(\tau)$ and \vspace{-2mm}
$$\dot\psi_p(t) \m>\m 0 \FORALL t\in(0,\sigma] \m.\vspace{-2mm}$$
If $\dot\psi_0>0$ then this is obvious. If $\dot\psi_0=0$ then, since
$\dot\psi(t)\geq 0$ for all $t\in[0,\tau]$, $\ddot\psi(0)\geq 0$.
Using this and $\|\gamma\|_{L^\infty}<d$, \rfb{pendulum} and
\rfb{pend_max} give that $\ddot\psi_p(0)>0$, which together with
$\dot\psi_p(0)\geq 0$ implies the existence of $\sigma\in(0,\tau)$
with the desired properties.

Our second claim is that for each $\vp_1\in(\psi_0,\psi_p(\sigma))$,
there exists a $\varphi\in(\psi_0,\vp_1)$ such that \vspace{-2mm}
\BEQ{Dan_coming}
   \dot{\psi_p}|_{{\psi_p}=\vp} \m>\m \dot\psi|_{\psi=\vp} \m.
\end{equation}
Indeed, if this claim were false, then \vspace{-2mm}
\BEQ{saturday_eq}
  \dot{\psi_p}|_{{\psi_p}=\vp} \m\leq\m \dot\psi|_{\psi=\vp}
  \FORALL \vp\in(\psi_0,\vp_1) \vspace{-2mm}
\end{equation}
which using \rfb{phase_eqn} and \rfb{phase_eqn_max} gives that \m
$\frac{\dd\dot\psi_p(t)}{\dd\psi_p(t)}\big|_{\psi_p(t)=\varphi}>\frac
{\dd\dot\psi(t)}{\dd\psi(t)}\big|_{\psi(t)=\varphi}$ for all $\varphi
\in(\psi_0,\varphi_1)$. This contradicts \rfb{saturday_eq} since $\dot
\psi_p|_{\psi_p=\psi_0}=\dot\psi|_{\psi=\psi_0}=\dot\psi_0$.

So far we have shown that we can find points $\vp\in(\psi_0,\psi
(\tau))$ arbitrarily close to $\psi_0$ such that \rfb{Dan_coming}
holds. To complete the proof of this lemma, it is sufficient to
establish the following claim: if \rfb{Dan_coming} holds for some
$\vp\in(\psi_0,\psi(\tau))$, then \rfb{Dan_coming} holds for all
$\tilde\vp\in(\vp,\psi(\tau))$ (with $\tilde\vp$ in place of $\vp$).

To prove the above claim, suppose that it is not true for some
$\varphi$. Then define \vspace{-1mm}
\BEQ{contr}
   \varphi_{bad} \m=\m \inf \m \left\{\tilde\varphi\in(\varphi,
   \psi(\tau)) \mid\ \dot\psi_p|_{\psi_p=\tilde\varphi}=\dot\psi|
   _{\psi=\tilde\varphi} \right\} \m. \vspace{-2mm}
\end{equation}
Let $t_{bad}\in(0,\tau)$ be such that $\psi(t_{bad})=\varphi_{bad}$,
so that $\dot\psi_p(t_{bad})=\dot\psi(t_{bad})$. We will first show by
contradiction that $\dot\psi(t_{bad})>0$. To this end, suppose that
$\dot\psi(t_{bad})=0$. This implies that $\dot\psi(t_{bad})$ is
a local minimum for $\dot\psi$ and so $\ddot\psi(t_{bad})=0$. From
\rfb{pendulum} we get that $-\sin\varphi_{bad}+\beta+\gamma(t_{bad})
=0$, hence $-\sin\varphi_{bad}+\beta+d>0$. From \rfb{pend_max} we
get that $\ddot\psi_p(t_{bad})>0$, so that for $t<t_{bad}$ very
close to $t_{bad}$ and satisfying $\psi(t)\in(\vp,\psi(\tau))$, \m
$\dot\psi_p(t)<\dot\psi_p(t_{bad})=0$. But \rfb{contr} gives that
$\dot\psi(t)<\dot\psi_p(t)$ and so $\dot\psi(t)<0$, which contradicts
the assumption $\dot\psi\geq 0$ in the lemma. Thus $\dot\psi_p
(t_{bad})=\dot\psi(t_{bad})>0$. Now \rfb{phase_eqn} and
\rfb{phase_eqn_max} give that for some $\mu>0$ \vspace{-2mm}
\BEQ{diff_ineq}
   \frac{\dd\dot\psi_p(t)}{\dd\psi_p(t)}\bigg|_{\psi_p(t)=\tilde
   \vp} \m>\m \frac{\dd\dot\psi(t)}{\dd\psi(t)}\bigg|_{\psi(t)=
   \tilde\vp} \FORALL \tilde\vp\in(\varphi_{bad}-\mu,\vp_{bad})
\end{equation}
and $\varphi_{bad}-\mu>\varphi$. This is because the above inequality
holds when $\tilde\varphi=\varphi_{bad}$. By assumption $\dot\psi_p|
_{\psi_p=\varphi_{bad}-\mu}>\dot\psi|_{\psi=\varphi_{bad}-\mu}$
which, along with \rfb{diff_ineq}, gives the contradiction
$\dot\psi_p|_{\psi=\varphi_{bad}}>\dot\psi|_{\psi=\varphi_{bad}}$.
This proves the claim above \rfb{contr}.
\end{proof}

\begin{lemma} \label{traj_below}
Consider the pendulum equation
\BEQ{pend_min}
   \ddot\psi_n(t)+\alpha\dot\psi_n(t)+\sin\psi_n(t) \m=\m \beta-d
   \FORALL t\geq0\m.
\end{equation}
Let $\psi$ and $\psi_n$ be the solutions of \rfb{pendulum} and
\rfb{pend_min}, respectively, for the initial conditions $\psi(0)=
\psi_n(0)=\psi_0$ and $\dot\psi(0)=\dot\psi_n(0)=\dot\psi_0$. Suppose
that $\dot\psi(t)\leq 0$ for all $t\in[0,\tau]$ and $\psi(\tau)\neq
\psi(0)$ for some $\tau>0$. Then the curve $(\psi_n,\dot\psi_n)$ lies
below the curve $(\psi,\dot\psi)$ in the phase plane on the angle
interval $(\psi(\tau),\psi(0))$, i.e. for each $\vp\in(\psi(\tau),
\psi(0))$ we have $\dot{\psi_n}|_{{\psi_n}=\vp}<\dot\psi|_{\psi=\vp}$.
\end{lemma}

\begin{proof}
Apply the change of variables $\psi\mapsto-\psi$ and $\psi_n\mapsto-
\psi_p$ to \rfb{pendulum} and \rfb{pend_min}, respectively. Now
apply Lemma \ref{traj_above} to the resulting equations (instead of
\rfb{pendulum} and \rfb{pend_max}) after redefining $\beta$ and
$\gamma$ to be $-\beta$ and $-\gamma$, respectively.
\end{proof}

The following result on the nonexistence of non-constant periodic
solutions to the pendulum equation with a constant forcing
term has been established in \cite{Hay:53}.

\begin{theorem} \label{Hayes}
Consider the pendulum equation \vspace{-1mm}
\BEQ{pend_hayes}
  \ddot\psi_h(t)+\alpha\dot\psi_h(t)+\sin\psi_h(t) \m=\m \sin\l
  \FORALL t\geq0\m, \vspace{-1mm}
\end{equation}
where $\alpha>0$ and $\l\in(0,\pi/2)$. If \m $\alpha>2\sin(\l/2)$ and
$\psi_h$ is a solution of \rfb{pend_hayes} such that $\dot\psi_h$ is
non-negative and periodic, then \m $\psi_h$ is constant.
\end{theorem}

The locally asymptotically stable equilibrium points of \rfb
{pend_hayes} are located at $\psi_h=\l+2k\pi$ ($k\in\zline$) and $\dot
\psi_h=0$ while the unstable equilibria are at $\psi_h=\pi-\l+2k\pi$
and $\dot\psi_h=0$, regardless of the size of the damping factor
$\alpha>0$. Figure 2 shows the typical shape of the curves $(\psi_h,
\dot\psi_h)$ in the phase plane for a sufficiently large (but not too
large) $\alpha>0$, so that we are in the case considered in Theorem
\ref{Hayes}. In this case every state trajectory converges to an
equilibrium point. If $\alpha$ gets even larger, then of course we are
still in the case considered in Theorem \ref{Hayes}, but the curves do
not spiral around the stable equilibria. This distinction may be
visually remarkable, but is not important for our analysis, so we do
not discuss it further.

\m \vspace{-10mm}
$$\includegraphics[scale=0.18]{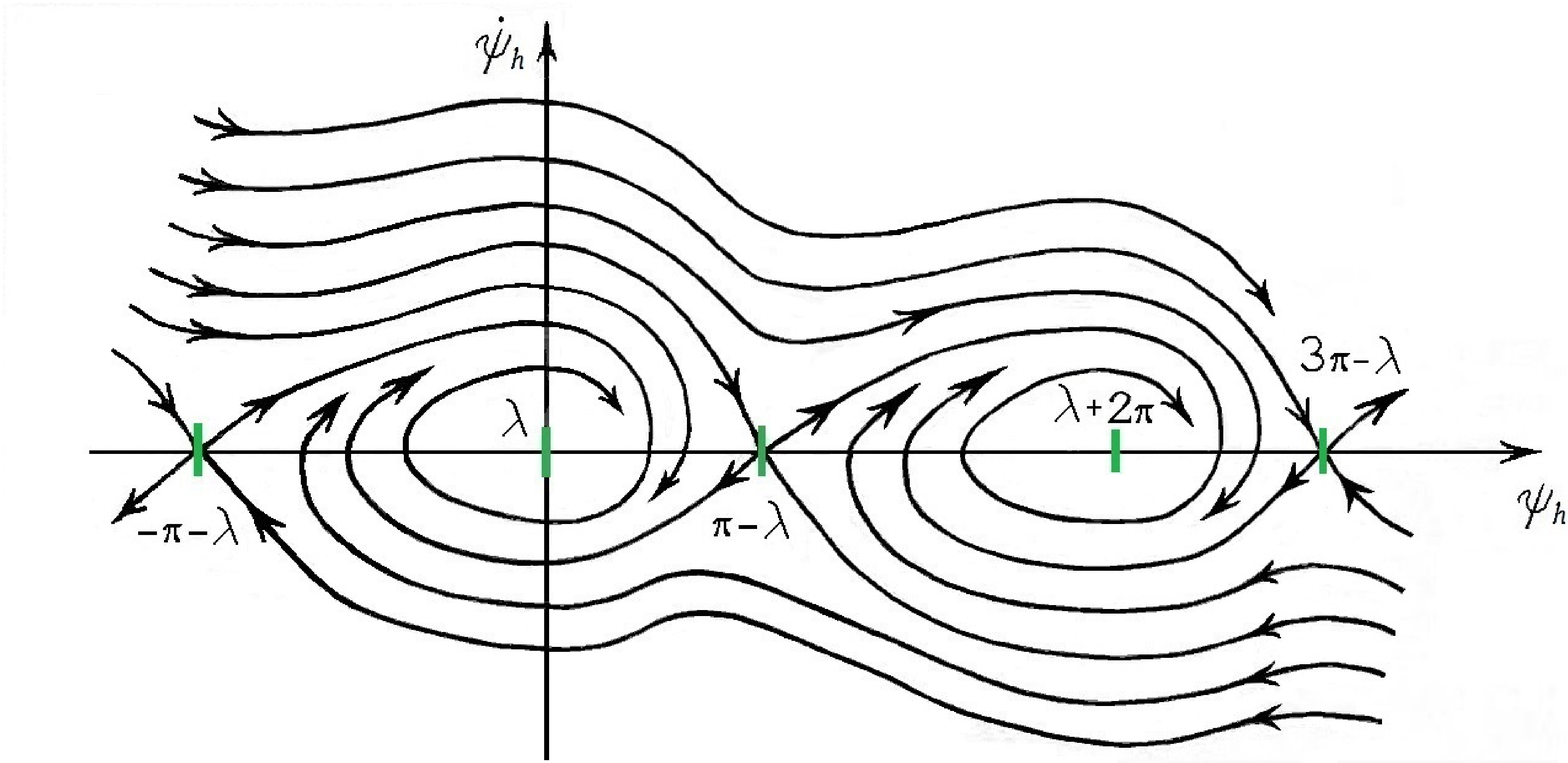}\vspace{-2mm}$$
\vspace{-2mm}\centerline{ \parbox{5.3in}{
Figure 2. Phase plane curves for the damped pendulum from
\rfb{pend_hayes} for moderately large damping factor $\alpha>2\sin
(\l/2)$ (not to scale).}}
\medskip

When $\alpha>0$ is small, then the curves $(\psi_h,\dot\psi_h)$ look
fundamentally different: some state trajectories still converge to one
of the same equilibrium points. Other state trajectories approach a
curve that is a stable limit cycle if we measure the angle $\psi_h$
modulo $2\pi$, see Figure 3 and the note after the proof of
Proposition \ref{similar_hayes}. The critical value of $\alpha$ that
separates between these two types of behavior (which depends on $\l$)
is estimated in Theorem \ref{Hayes} due to W. Hayes in 1953
\cite{Hay:53}. We are not aware of any better estimate available now
(other than by simulation experiments). Our interest is in the forced
pendulum \rfb{pendulum}, and for us \rfb{pend_hayes} is only a tool
for comparison. Much material about systems related to
\rfb{pend_hayes} can be found in \cite[Ch.~3]{LePoSm:96}.

\m\vspace{-8mm}
$$\includegraphics[scale=0.18]{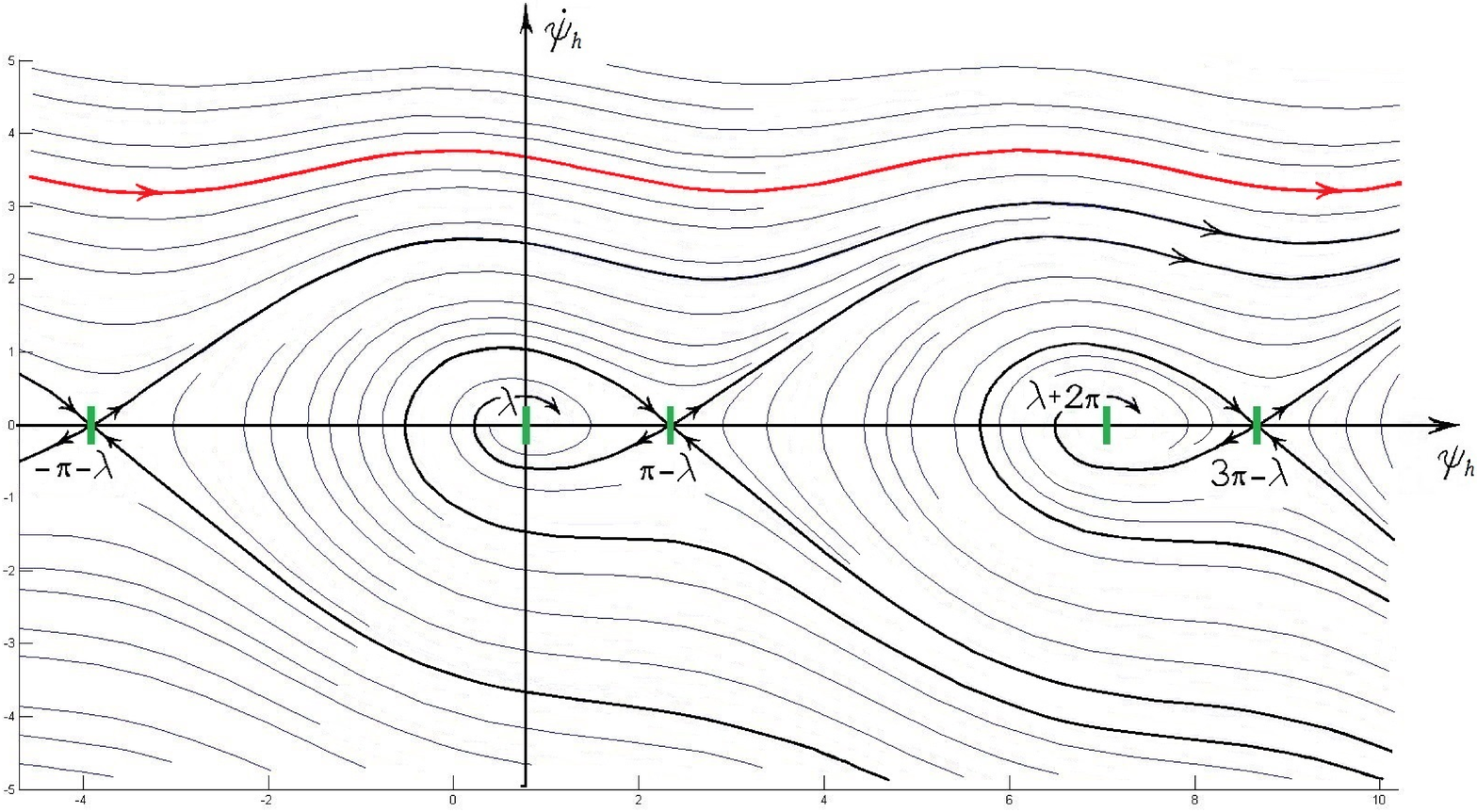}
  \vspace{-3mm}$$ \vspace{-1mm}\centerline{ \parbox{5.3in}{
Figure 3. Phase plane curves for the damped pendulum from
\rfb{pend_hayes} for small $\alpha$ ($\alpha=0.2,\ \sin\l=0.7$). The
limit cycle is shown as a red curve.}}

\medskip
For the pendulum system \rfb{pendulum}, we define the energy
function $E$ as \vspace{-1mm}
\BEQ{energy_fnc}
  E(t) \m=\m \half\dot\psi(t)^2 + (1-\cos\psi(t)) \FORALL t\geq0 \m.
  \vspace{-1mm}
\end{equation}
The time derivative of $E$ along the trajectories of \rfb{pendulum}
is given by \vspace{-1mm}
\BEQ{energy_der}
  \dot E(t) \m=\m -\alpha\dot\psi(t)^2 + (\beta+\gamma(t))\dot
  \psi(t) \m.
\end{equation}
Therefore for any $t_2>t_1\geq 0$ \vspace{-2mm}
$$ E(t_2)-E(t_1) \m=\m \int_{t_1}^{t_2} [-\alpha\dot\psi(s)+\beta+
   \gamma(s)] \dot\psi(s) \dd s \m. \vspace{-1mm}$$
If $\dot\psi(t)\neq 0$ for all $t\in[t_1,t_2]$, then using the
change of variables $s\mapsto\psi(s)$ we get \vspace{-1mm}
\BEQ{energy_phaseplane}
   E(t_2)-E(t_1) \m=\m \int_{\psi(t_1)}^{\psi(t_2)} \left[ -\alpha
   \dot\psi(\psi^{-1}(\vp))+\beta+\gamma(\psi^{-1}(\vp)) \right]
   \dd\vp \m.\vspace{-2mm}
\end{equation}

Using Theorem \ref{Hayes}, the next proposition shows that if
$\alpha$ is sufficiently large, then each solution $(\psi,\dot
\psi)$ of \rfb{pendulum} must either converge to a limit point
or its velocity must change sign at least once after any
given time $t\geq 0$.

\begin{proposition} \label{similar_hayes}
Recall the angles $\psi_1$ and $\psi_2$ from \rfb{t1-t2}. Assume
that \vspace{-2mm}
\BEQ{assumption}
  \alpha \m>\m 2\sin\frac{|\psi_1|}{2} \m, \qquad \alpha \m>\m 2
  \sin\frac{|\psi_2|}{2} \m.\vspace{-2mm}
\end{equation}
Then there exists no solution $\psi$ of \rfb{pendulum} such that
$\psi$ is unbounded and $\dot\psi(t)$ is either non-negative or
non-positive for all $t\geq 0$.
\end{proposition}

\begin{proof}
In the first part of this proof we assume that $\psi$ is a solution of
\rfb{pendulum} such that $\dot\psi(t)\geq 0$ \m for all $t\geq 0$. Our
first claim is that if $\psi$ is unbounded, then \vspace{-2mm}
$$\beta+d \m>\m 0 \m,\ \mbox{ hence }\ \psi_1>0 \m.\vspace{-1mm}$$
Indeed, if not, then the right-hand side of \rfb{energy_der} is \m
$\leq-\mu\dot\psi(t)$ for some $\mu>0$, forcing $E$ to become
eventually negative, which is impossible.

Our second claim is that for any $\tau>0$, \vspace{-1mm}
\BEQ{Simy_still_ill}
   \mbox{if \ $\sin\psi(\tau)>\beta+\|\gamma\|_{L^\infty}$, \
   then \ $\ddot\psi(\tau)<0$ \m and \m $\dot\psi(\tau)>0$\m.}
\end{equation}
If $\sin\psi(\tau)>\beta+\|\gamma\|_{L^\infty}$ then clearly $\sin\psi
(\tau)>\beta+\gamma(\tau)$, hence from \rfb{pendulum} we get $\ddot
\psi(\tau)<0$. It follows that $\dot\psi(\tau)>0$ since otherwise (if
it is zero) then for $t>\tau$ close to $\tau$ we would have $\dot\psi
(t)<0$, contradicting our assumption that $\dot\psi\geq 0$.

In the sequel, we assume that $\psi$ is unbounded (which will lead to
a contradiction). Let $t_0>0$ be such that $\psi(t_0)=2m\pi+\psi_1$,
$m\in\zline$. Our third claim is that 
\BEQ{Pizza_Hut}
  \inf\left\{\dot\psi(t) \m\big|\ \psi(t) \in [2k\pi+\psi_1,(2k+1)\pi
  -\psi_1],\ k\in\zline, k\geq m\right\} \m\geq\m \e \m>\m 0 \m.
\end{equation}
Define $\psi_\gamma\in(-\pi/2,\pi/2)$ so that $\sin\psi_\gamma=\beta+
\|\gamma\|_{L^\infty}$. Let $t_1,t_2>0$ be such that $\psi(t_1)=
(2m+1)\pi-\psi_1$ and $\psi(t_2)=(2m+1)\pi-\psi_\gamma$, so that $t_0
<t_1<t_2$. Then since $\ddot\psi(t)<0$ for all $t\in[t_0,t_1]$ it
follows that for all $t\in[t_0,t_1]$ we have $\dot\psi(t)>\dot\psi
(t_1)>0$. To prove \rfb{Pizza_Hut}, we have to find a lower bound on
$\dot\psi(t_1)$ that is independent of $m$. If we regard $(\psi,\dot
\psi)$ as a curve in the phase plane, then from \rfb{phase_eqn} and
\rfb{Simy_still_ill} we see that for all $t\in[t_0,t_2)$, \m $\frac{
\dd\dot\psi(t)}{\dd \psi(t)}<-\alpha$\m. From here, by integration,
$\dot\psi(t_1)>\dot\psi(t_2)+\alpha(\psi_1-\psi_\gamma)$. Using again
\rfb{Simy_still_ill} we see that we can take $\e=\alpha(\psi_1-\psi_
\gamma)>0$ in \rfb{Pizza_Hut}.

Let $\psi_p$ be the solution of \rfb{pend_max} with $\psi_p(0)=\psi
(t_0)$ and $\dot\psi_p(0)=\dot\psi(t_0)$. Then $\dot\psi_p$ is bounded
(by the argument in \rfb{JF_election}) and $\psi_p$ is defined on $[0,
\infty)$. Our fourth claim is that $\dot\psi_p|_{\psi_p=\vp}>\dot\psi|
_{\psi=\vp}$ for all $\vp>\psi(t_0)$. Indeed, from Lemma \ref{traj_above}
it follows that the curve $(\psi_p,\dot\psi_p)$ lies above the curve
$(\psi,\dot\psi)$ on the angle interval $(\psi(t_0),\infty)$. Hence
$\dot\psi_p$ is a strictly positive function and \rfb{Pizza_Hut} holds
if we replace $\psi$ and $\dot\psi$ with $\psi_p$ and $\dot\psi_p$.

The fifth claim is that there is a solution $\psi_p^f$ of \rfb
{pend_max} such that $\dot\psi_p^f$ is a periodic and strictly positive
function of time. For this, first we find the function $f$ which
represents one period of $\psi_p^f$ in the phase plane. Consider the
sequence of strictly positive continuous functions $(f_k)_{k=m}^\infty$
defined on the angle interval $I=[\psi_1,\psi_1+2\pi]$ as follows: $f_k
(\varphi)=\dot\psi_p|_{\psi_p=\varphi+2k\pi}$ for each $\varphi\in I$,
where $\psi_p$ is as defined in the previous paragraph. Clearly for each
$k\geq m$, the curve defined by the graph of $f_k$ in the phase plane
is a segment of the curve $(\psi_p,\dot\psi_p)$ and so it follows
from \rfb{phase_eqn_max} that \vspace{-1mm}
\BEQ{fk_curve}
  \frac{\dd f_k(\vp)}{\dd \vp} \m=\m -\alpha + \frac{\beta+d-\sin\vp}
  {f_k(\vp)} \FORALL \vp\in[\psi_1,\psi_1+2\pi]\m. \vspace{-1mm}
\end{equation}
Since no two curves corresponding to two distinct solutions of \rfb
{pend_max} can intersect in the phase plane, for $k_1\neq k_2$ the curves
defined by $f_{k_1}$ and $f_{k_2}$ must either be the same or do not
intersect at all. This, along with the fact that $f_k(\psi_1+2\pi)=f_
{k+1}(\psi_1)$ for each $k\geq m$, implies that $f_{k+1}-f_k$ is either
a non-negative function for all $k$ or it is a non-positive function for
all $k$. Therefore the sequence $(f_k)_{k=m}^\infty$ converges to $f$
which is a non-negative continuous function defined on $I$ satisfying
$f(\psi_1)=f(\psi_1+2\pi)$ (here we have used the fact that the functions
$f_k$ are uniformly bounded). By using the version of \rfb{Pizza_Hut}
with $\psi_p$ in place of $\psi$, for each $\varphi\in[\psi_1,\pi-
\psi_1]$ we get that $f_k(\varphi)\geq\e$ for all $k\geq m$ and so
$f(\varphi)\geq\e$. This means that for all $\vp\in[\psi_1,\pi-\psi_1]$
and all $k\geq m$, the right side of \rfb{fk_curve} is bounded in
absolute value by $(\beta+d+1)/\e+\alpha$. Using this, we can take the
limit as $k\to\infty$ on both sides of \rfb{fk_curve} to conclude that
$f$ satisfies \rfb{fk_curve} on the interval $\varphi\in[\psi_1,\pi
-\psi_1]$.

To complete the proof of the above claim, let $\psi_p^f$ be the
solution of \rfb{pend_max} for the initial state $\psi_p^f(0)=\psi_1$,
$\dot\psi_p^f(0)=f(\psi_1)$. Since the curve $(\psi_p^f,\dot\psi_p^f)$
satisfies \rfb{phase_eqn_max} and $f$ satisfies \rfb{fk_curve} (which
is the same ODE as \rfb{phase_eqn_max}), it follows that $f(\varphi)=
\dot\psi_p^f|_{\psi_p^f=\varphi}$ for all $\varphi\in[\psi_1,\pi
-\psi_1]$. In particular \m $\dot\psi_p^f(t)\geq\e$ \m as long as
$\psi_p^f(t)\leq\pi-\psi_1$. We now show that $\dot\psi_p^f(t)>0$ as
long as $\psi_p^f(t)\in(\pi-\psi_1,\psi_1+2\pi)$. Indeed, if $\dot
\psi_p^f(t)<(\beta+d-\sin\varphi)/\alpha$ (which is a positive number)
and $\psi_p^f(t)\in(\pi-\psi_1,\psi_1+2\pi)$, then \rfb{pend_max} gives
that $\ddot\psi_p^f(t)>0$, so that $\dot\psi_p^f$ is increasing and
hence it cannot become $\leq 0$. Therefore \m $\lim_{t\to\infty}
\psi_p^f(t)\geq\psi_1+2\pi$. By the same argument as used earlier
for $\varphi\in[\psi_1,\pi-\psi_1]$, \m $f(\varphi)=\dot\psi_p^f
|_{\psi_p^f=\varphi}$ for all $\varphi\in[\psi_1,\psi_1+2\pi]$.
Therefore $\dot\psi_p^f|_{\psi_p^f=\psi_1}=\dot\psi_p^f|_{\psi_p^f
=\psi_1+2\pi}$, so that $\dot\psi_p^f$ is a periodic and strictly
positive function.

The fifth claim (that we proved) together with the first inequality in
\rfb{assumption} contradict Theorem \ref{Hayes}, because $\psi_p^f$ is
a solution of \rfb{pend_hayes} when $\l=\psi_1$ (here we have used the
first claim). Thus if $\dot\psi(t)\geq 0$ for all $t\geq 0$, then
$\psi$ must be bounded.

Next assume that $\psi$ is unbounded and $\dot\psi(t)\leq0$ for all
$t\geq0$. Then $-\psi$ is unbounded and $-\dot\psi(t)\geq0$ and
$-\psi$ is a solution of \rfb{pendulum} when $\beta+\gamma(t)$ on the
right side is replaced with $-\beta-\gamma(t)$. The above proof, using
$-\beta$ in place of $\beta$, and the second inequality in \rfb
{assumption} will again give rise to a contradiction implying that if
$\dot\psi(t)\leq0$ for all $t\geq0$, then $\psi$ must be bounded.
%
\end{proof}

Note that the above proof also contains (around the fifth claim) the
main ingredients of the proof of the following fact: If $\beta+d>0$,
then any unbounded solution $\psi_p$ of \rfb{pend_max} converges to a
solution $\psi_p^f$ such that $\dot\psi_p^f$ is positive and periodic
(both as a function of time and as a function of $\psi_p^f$) (shown
as the red curve in Figure 3). A similar statement holds for $\beta+d
<0$, in which case $\dot\psi_p^f$ is negative and periodic.

\begin{definition} \label{posneg}
A point $(\vp,0)$ in the phase plane is called a {\em positive
acceleration point} if $\beta+d-\sin\vp>0$, i.e. $\ddot\psi_p(0)>0$
according to \rfb{pend_max} when $\psi_p(0)=\vp$ and $\dot\psi_p(0)
=0$. A point $(\vp,0)$ in the phase plane is called a {\em negative
acceleration point} if $\beta-d-\sin\vp<0$, which has a similar
interpretation as $\ddot\psi_n(0)<0$ using \rfb{pend_min}.
\end{definition}

Using the notation \rfb{t1-t2}, the set of positive acceleration
points is \vspace{-1mm}
$$ \left\{ (\varphi,0)\m\big|\ \varphi\in((2k-1)\pi-\psi_1,2k\pi+
   \psi_1),\ k\in\zline \right\} \vspace{-1mm} $$
and the set of negative acceleration points is \vspace{-1mm}
$$ \left\{ (\varphi,0)\m\big|\ \varphi\in(2k\pi+\psi_2,(2k+1)\pi-
   \psi_2),\ k\in\zline \right\} \m.$$

\begin{lemma} \label{pos_acc}
Suppose that \m $\alpha>2\sin(|\psi_1|/2)$ and $\psi_p$ is the
solution of \rfb{pend_max} when \vspace{-1mm}
$$ (\psi_p(0),\dot\psi_p(0)) \m=\m (\vp^0,0) \m,\qquad (2m-1)\pi-
   \psi_1 \m<\m \vp^0 \m<\m 2m\pi+\psi_1 \m,\ \ m\in\zline
   \vspace{-1mm} $$
(so that $(\varphi^0,0)$ is a positive acceleration point). Denote
\vspace{-2mm}
$$ \tau \m=\m \sup\m\{T>0\m\big|\ \dot\psi_p(t)>0 \textrm{ for all }
   t\in(0,T)\} \m.\vspace{-2mm}$$

If \m\m $\tau<\infty$ then \m $(\psi_p(\tau),\dot\psi_p(\tau))=(\vp^1,
0)$, where \ $\varphi^1\in(2m\pi+\psi_1,(2m+1)\pi-|\psi_1|)$.

\vspace{-1mm}
If \m\m $\tau=\infty$ then \ $\lim_{t\to\infty}(\psi_p(t),\dot
\psi_p(t))=(\varphi^1,0)$, where $\varphi^1=2m\pi+\psi_1$.
\end{lemma}

\vspace{-2mm}
Note that in both cases listed above, $(\varphi^1,0)$ is a negative
acceleration point.

\begin{proof}
Choose $\l\in(|\psi_1|,\pi/2)$ such that $\alpha>2\sin(\l/2)$.
Suppose that \m $\lim_{t\to\tau}\psi_p(t)>(2m+1)\pi-\l$ (which will
lead to a contradiction). Let $\psi_h$ be a solution of \rfb
{pend_hayes} with $\psi_h(0)=\varphi^0$ and $\dot\psi_h(0)\geq 0$.
Since $\sin\l>\sin\psi_1$ it can be shown (like in the proof of Lemma
\ref{traj_above}) that the curve $(\psi_h,\dot\psi_h)$ is above the
curve $(\psi_p,\dot\psi_p)$ in the phase plane on the angle interval
$(\varphi^0,(2m+1)\pi-\l)$. Therefore $\dot\psi_h|_{\psi_h=(2m+1)\pi-
\l}>0$. This implies that there exists $\tau_1>0$ such that $\psi_h
(\tau_1)=2(m+1)\pi+\l$ \m and
$$ \dot\psi_h|_{\psi_h=\varphi} \m>\m 0 \FORALL \varphi\in
   ((2m+1)\pi-\l,2(m+1)\pi+\l) \m.$$
This is because if $\dot\psi_h|_{\psi_h=\varphi}<(\sin\l-\sin\varphi)
/\alpha$ (which is $>0$), then \rfb{pend_hayes} gives that $\ddot
\psi_h|_{\psi_h=\varphi}>0$. It now follows from the above discussion
that
\BEQ{switched_off}
   \dot\psi_h|_{\psi_h=\vp} \m>\m 0 \FORALL\vp\in(\vp^0,\vp^0+2\pi]\m.
\end{equation}

Let $\psi_h^1$ be the solution of \rfb{pend_hayes} with \m $(\psi_h^1
(0),\dot\psi_h^1(0))=(\varphi^0,\dot\psi_h|_{\psi_h=\varphi^0+2\pi})$.
Then repeating the above argument we obtain that \rfb{switched_off}
holds with $\psi_h^1$ in place of $\psi_h$. By concatenating $\psi_h$
with $\psi_h^1$, we obtain that the solution $\psi_h$ of
\rfb{pend_hayes} actually advances by at least two full circles from
its initial angle $\varphi^0$, and we have
$$ \dot\psi_h|_{\psi_h=\varphi} \m>\m 0 \FORALL \varphi\in (\varphi^0,
   \varphi^0+4\pi] \m.$$
Continuing by induction, we obtain that $\psi_h$ is unbounded and
$\dot\psi_h>0$ all the time. This contradicts Proposition \ref
{similar_hayes} in which we replace $\psi_1$ with $\tilde\psi_1>\l$
such that $\alpha>2\sin(\tilde\psi_1/2)$ still holds. Indeed, then
$\psi_h$ is a solution of \rfb{pendulum} when $\beta=0$ and $\gamma(t)
=\sin\l$ for all $t\geq0$. Hence the assumption at the start of our
proof is false, which means that \vspace{-1mm}
$$\lim_{t\to\tau} \psi_p(t) \m\leq\m (2m+1)\pi-\l \m.\vspace{-2mm}$$

Consider the case when $\tau<\infty$, so that $\dot\psi_p(\tau)=0$. We
claim that $\varphi^1=\psi_p(\tau)\geq 2m\pi+\psi_1$. Indeed, $\dot
\psi_p(t)$ cannot reach 0 for a time $t$ when \m $\varphi^0<\psi_p(t)<
2m\pi+\psi_1$, because \rfb{pend_max} would imply that $\ddot\psi_p(t)
>0$. Thus, $\vp^1\in[2m\pi+\psi_1,(2m+1)\pi-\l]$. Next we claim that
$\vp^1>2m\pi+\psi_1$. Indeed, if $\vp^1=2m\pi+\psi_1$, then $(\vp^1,0)$
is an equilibrium point of the system \rfb{pend_max} and $\xx=(\psi_p,
\dot\psi_p)-(\vp^1,0)$ satisfies an ODE of the form $\dot\xx=f(\xx)$,
where $f\in C^1$ and $f(0)=0$. It is well known that for such an ODE,
any trajectory starting from $\xx(0)\not=0$ cannot reach the point
$(0,0)$ in a finite time. Thus, we have $\vp^1>2m\pi+\psi_1$. Combining
this with the fact that $\l>|\psi_1|$, we get that \m $\vp^1\in(2m\pi+
\psi_1,(2m+1)\pi-|\psi_1|)$, as stated in the lemma.

Now consider the case when $\tau=\infty$. Since $\psi_p$ is increasing
and bounded, clearly $\dot\psi_p\in L^1[0,\infty)$. Since $\dot\psi_p$
is bounded (by the argument at \rfb{JF_election}), it follows from
\rfb{pend_max} that $\ddot\psi_p$ is also bounded, so that $\dot\psi
_p$ is uniformly continuous. Now applying Barb\u alat's lemma (see
\cite[Lemma 8.2]{Kha:02} or see \cite{FaWe:15,LoRy:04} for nice
presentations with a more general perspective), we get that $\lim_{t\to
\infty}\dot\psi_p(t)=0$. We have $\varphi^1=\lim_{t\to\infty}\psi_p(t)
\in[2m\pi+\psi_1,(2m+1)\pi-|\psi_1|)$, for similar reasons as in the
case $\tau<\infty$. By differentiating \rfb{pend_max}, we see that
$\dddot\psi_p$ is also bounded. Since the expressions $\int_0^t\ddot
\psi_p(\sigma)\dd\sigma=\dot\psi_p(t)$ are uniformly bounded (with
respect to $t$), we can apply again Barb\u alat's lemma, this time to
$\ddot\psi_p$, to show that $\lim_{t\to\infty}\ddot\psi_p(t)=0$.
Looking at \rfb{pend_max}, it follows that $\lim_{t\to\infty}\sin\psi
_p(t)=\sin\psi_1$, so that $\sin\vp^1=\sin\psi_1$. Looking at the
range of possible values of $\vp^1$, we conclude that it has indeed
the value stated in the lemma.
\end{proof}

With the notation of the last lemma, we call $(\varphi^1,0)$ the {\em
first negative acceleration point} for $\psi_p$. We remark that $\tau=
\infty$ for sufficiently large $\alpha$, regardless of $\varphi^0$.

The following lemma concerning solutions of \rfb{pend_min} is similar
to Lemma \ref{pos_acc}.

\begin{lemma} \label{neg_acc}
Suppose that \m $\alpha>2\sin(|\psi_2|/2)$ and $\psi_n$ is the
solution of \rfb{pend_min} when \vspace{-1mm}
$$ (\psi_n(0),\dot\psi_n(0)) \m=\m (\vp^1,0) \m,\qquad 2m\pi+\psi_2
   \m<\m \vp^1 \m<\m (2m+1)\pi-\psi_2 \m,\ \ m\in\zline
   \vspace{-1mm} $$
(so that $(\varphi^1,0)$ is a negative acceleration point). Denote
\vspace{-2mm}
$$ \tau \m=\m \sup\m\{T>0\m\big|\ \dot\psi_n(t)<0 \textrm{ for all }
   t\in(0,T)\} \m.\vspace{-2mm}$$

\vspace{-1mm}
If \m\m $\tau<\infty$ then \m $(\psi_n(\tau),\dot\psi_n(\tau))=(\vp
^2,0)$, where \ $\vp^2\in((2m-1)\pi+|\psi_2|,2m\pi+\psi_2)$.

\vspace{-1mm}
If \m\m $\tau=\infty$ then \ $\lim_{t\to\infty}(\psi_n(t),\dot
\psi_n(t))=(\varphi^2,0)$, where $\varphi^2=2m\pi+\psi_2$.
\end{lemma}

Note that in both cases listed above, $(\varphi^2,0)$ is a positive
acceleration point. \vspace{-7mm}

\begin{proof} \vspace{3mm}
Define $\tilde\psi_p=-\psi_n$ and $\tilde\beta=-\beta$, then $\tilde
\psi_p$ satisfies \rfb{pend_max} with $\tilde\beta$ in place of $\beta$.
Define $\tilde\psi_1=-\psi_2$, then the first expression in \rfb
{t1-t2} holds with $\tilde\psi_1$ and $\tilde\beta$ in place of
$\psi_1$ and $\beta$, and of course $\alpha>2\sin(|\tilde\psi_1|/2)$.
Define $\tilde\vp^0=-\vp^1$ and $\tilde m=-m$, then these
satisfy the assumption on initial conditions in Lemma \ref{pos_acc}.
Thus, we can apply Lemma \ref{pos_acc} with the tilde variables in
place of the original ones, and we get exactly the conclusions of the
lemma that we are now proving, with $\vp^2= -\tilde\vp^1$.
\end{proof} \vspace{-2mm}

With the notation of the last lemma, we call $(\varphi^2,0)$ the {\em
first positive acceleration point} for $\psi_n$. We remark that $\tau=
\infty$ for sufficiently large $\alpha$, regardless of $\varphi^1$.

Next we define a family of continuous curves in the phase plane
referred to as {\em spiral curves}. These curves have the structure
of an inward spiral.

\begin{definition} \label{traj_env}
Suppose that $\alpha$ satisfies \rfb{assumption}. Let $\varphi^0=(2m
-1)\pi-\psi_2$ for some integer $m$. Construct a sequence $(\varphi^k)
_{k=0}^\infty$ as follows: for each odd $k$, $(\varphi^k,0)$ is the
first negative acceleration point for the solution $\psi_p^k$ of \rfb
{pend_max} with initial conditions $\psi_p^k(0)=\varphi^{k-1}$, $\dot
\psi_p^k(0)=0$. For each even $k>0$, $(\varphi^k,0)$ is the first
positive acceleration point for the solution $\psi_n^k$ of \rfb
{pend_min} with initial conditions $\psi_n^k(0)=\varphi^{k-1}$, $\dot
\psi_n^k(0)=0$. For $k\in\nline$, denote the segment of the curve
$(\psi_p^k,\dot\psi_p^k)$ (or $(\psi_n^k,\dot\psi_n^k)$) between
$(\vp^{k-1},0)$ and $(\vp^k,0)$ by $\Gamma_k$. A {\em spiral curve}
$\Gamma$ starting from $(\varphi^0,0)$ is a continuous curve in the
phase plane obtained by concatenating all $\Gamma_k$ ($k\in\nline$).
\end{definition} \vspace{-2mm}

Lemmas \ref{pos_acc} and \ref{neg_acc} ensure that the points
$\varphi^k$ introduced in Definition \ref{traj_env} in fact exist for
all $k\in\nline$. The spiral curve can be interpreted as the phase
plane trajectory of the solution of \rfb{pendulum} with $\psi(0)=
\vp^0$, $\dot\psi(0)=0$ and \vspace{-2mm}
\BEQ{smoke_alarm}
   \gamma(t) \m=\m d\;{\rm sign}(\dot\psi(t)) \m, \vspace{-2mm}
\end{equation}
where the trajectory is continued even if it happens that a segment
$\Gamma_k$ takes an infinite amount of time. The above expression for
$\gamma(t)$ is like a static friction torque acting on a pendulum,
but with the wrong sign. We remark (but will not use) that for any
sufficiently large damping coefficient $\alpha$ the sequence
$(\varphi^k)_{k=1}^\infty$ is such that \vspace{-2mm}
\BEQ{Warwick_sunshine}
   \vp^1=\varphi^3=\vp^5=\ldots =2m\pi+\psi_1\ \mbox{ and }\
   \vp^2=\varphi^4=\vp^6=\ldots=2m\pi+\psi_2 \m.\vspace{-2mm}
\end{equation}
For smaller $\alpha$ only a part of the equalities in
\rfb{Warwick_sunshine} hold, possibly none. Figure 4 shows possible
shapes of $\Gamma_1$, $\Gamma_2$ and some limit curves $\Gamma_a$,
$\Gamma_b$ that will be introduced later, in the case when none of
the equalities in \rfb{Warwick_sunshine} holds.

\begin{lemma} \label{lm:env_traj}
Suppose that $\alpha$ satisfies \rfb{assumption}. Fix an integer $m$
and consider the spiral curve $\Gamma$ starting from $(\varphi^0,0)$
with $\varphi^0=(2m-1)\pi-\psi_2$. There exists a simple closed curve
$\Gamma_c$ in the phase plane to which $\Gamma$ converges, i.e. for
any $\e>0$ there exists an $N_\e\in\nline$ such that for every
$n\in\nline$ with $n\geq N_\e$, \vspace{-2mm}
\BEQ{Sainsbury}
   \textrm{\bf d}(\xx\m,\m\Gamma_c) \m<\m \e \FORALL \xx\in\Gamma_n
   \m.\vspace{-2mm}
\end{equation}
Here {\bf d} is the Euclidean distance in $\rline^2$. \vspace{-6mm}
\end{lemma}

\begin{proof}
Let $\bluff(\vp^k)_{k=0}^\infty$ be the sequence introduced in
Definition \ref{traj_env}. It follows directly from Lemmas
\ref{pos_acc} and \ref{neg_acc} that $\varphi^2>\varphi^0$. Using the
fact that two curves corresponding to two distinct solutions of \rfb
{pend_max} cannot intersect, we conclude that $\vp^3\leq\vp^1$. (We
remark that equality can only occur if $\vp^3=\vp^1=2m\pi+\psi_1$,
since distinct solutions may meet at a common limit point, which is
an equilibrium point of \rfb{pend_max}. In this case we have \rfb
{Warwick_sunshine} except possibly the last equality.) Using the fact
that no two curves corresponding to two distinct solutions of \rfb
{pend_min} can intersect, we conclude that $\varphi^4\geq\varphi^2$.
(We remark that if $\varphi^3<\varphi^1$, then $\varphi^4=\varphi^2$
can only occur if $\varphi^4=\varphi^2=2m\pi+\psi_2$, since distinct
solutions may meet at a common limit point, which is an equilibrium
point of \rfb{pend_min}. In this case we have \rfb{Warwick_sunshine}
except for the first and possibly the last equality from the first
string.) Continuing like this, we get that the sequence $(\varphi
^{2k+1})_{k=0}^\infty$ is nonincreasing and bounded from below by
$2m\pi+\psi_1$, while the sequence $(\varphi^{2k})_{k=0}^\infty$ is
nondecreasing and bounded from above by $2m\pi+\psi_2$. Let
\vspace{-2mm}
\BEQ{two_limit_points}
   \varphi_{low} \m=\m \lim_{k\to\infty} \varphi^{2k} \m,\qquad
   \varphi_{high} \m=\m \lim_{k\to\infty} \varphi^{2k+1} \m.
\end{equation}
Clearly these are positive and negative acceleration points,
respectively.

For the remainder of this proof, for any $\vp\in\rline$ we denote by
$\psi_p(\cdot,\varphi)$ the solution $\psi_p$ of \rfb{pend_max}
satisfying $\psi_p(0)=\varphi$ and $\dot\psi_p(0)=0$. Let
$(\tilde\varphi_{high},0)$ be the first negative acceleration point
for $\psi_p(\cdot,\varphi_{low})$. By Lemma \ref{pos_acc} we have
$\tilde\varphi_{high}\geq 2m\pi+\psi_1$. Since for any $k\in\nline$ we
have $\varphi_{low}\geq\varphi^{2k}$ and no two curves corresponding
to two distinct solutions of \rfb{pend_max} can intersect in the phase
plane, we have $\tilde\varphi_{high}\leq \varphi^{2k+1}$. Taking
limits, we obtain that $\tilde\varphi_{high}\leq\varphi_{high}$. Thus,
using also Lemma \ref{pos_acc}, \vspace{-1mm}
\BEQ{Putin}
   2m\pi+\psi_1 \m\leq\m \tilde\varphi_{high} \m\leq\m
   \varphi_{high} \m<\m (2m+1)\pi-|\psi_1| \m.  \vspace{-1mm}
\end{equation}

Denote the segment of the curve $(\psi_p(\cdot,\varphi_{low}),\dot
\psi_p(\cdot,\varphi_{low}))$ between $(\varphi_{low},0)$ and $(\tilde
\varphi_{high},0)$ by $\Gamma_a$. We claim that for any $\e>0$ there
exists $N_\e\in\nline$ such that \vspace{-1mm}
\BEQ{first_day_of_boys_in_UK_school}
   \mbox{ if }\ \ 2n \m\geq\m N_\e \m,\ \ \mbox{ then }\ \ \textrm
   {\bf d}(\xx\m,\m\Gamma_a) \m<\m \e \FORALL \xx\in\Gamma_{2n+1} \m.
\vspace{-1mm} \end{equation}
Note that this implies (by an easy argument that we omit) that $\tilde
\varphi_{high}=\varphi_{high}$.

To prove \rfb{first_day_of_boys_in_UK_school}, we have to consider two
cases:

{\bf Case 1:} \m $\tilde\varphi_{high}>2m\pi+\psi_1$ (this is the
easier case). According to Lemma \ref{pos_acc}, there exists a
smallest $\tau>0$ such that $(\psi_p(\tau,\varphi_{low}),\dot\psi_p
(\tau,\varphi_{low}))=(\tilde\varphi_{high},0)$. It is easy to see
that there exists $T>\tau$ such that \vspace{-0.5mm}
$$ 2m\pi+\psi_1 \m<\m \psi_p(t,\varphi_{low}) \m<\m \tilde\varphi_{high}
   \ \mbox{ and }\ \dot\psi_p(t,\varphi_{low}) \m<\m 0 \FORALL t\in
   (\tau,T] \m. \vspace{-0.5mm}$$
According to the standard result on the continuous dependence of
solutions of differential equations (satisfying a Lipschitz
condition) on their initial conditions, for any $\e>0$ there exists
an $N_\e\in\nline$ such that for all $n\in\nline$ with $2n\geq N_\e$,
\vspace{-1mm}
\BEQ{Oxfam_suitcase}
   |\psi_p(t,\vp_{low})-\psi_p(t,\vp^{2n})| + |\dot\psi_p(t,\vp_{low})
   -\dot\psi_p(t,\vp^{2n})| \m<\m \e\FORALL t\in[0,T] \m.
\end{equation}
For each such $n$, let $\tau_{2n}$ be the smallest positive number
such that $(\psi_p(\tau_{2n},\varphi^{2n}),0)$ is the first negative
acceleration point of \m $\psi_p(\cdot,\varphi^{2n})$. Then it is easy
to verify, using \rfb{Oxfam_suitcase}, that \m $\e<|\dot\psi_p(T,
\varphi_{low})|$ \m implies \ $\dot\psi_p(T,\varphi^{2n})<0$, so that
\vspace{-1mm}
$$ \mbox{if }\ \ \e \m<\m |\dot\psi_p(T,\varphi_{low})| \ \ \mbox{
   and  }\ \ 2n \m\geq\m N_\e \m,\ \ \mbox{ then }\ \ \tau_{2n}
   \m<\m T \m.\vspace{-1mm}$$
From here, by an easy argument using \rfb{Oxfam_suitcase} we obtain
that for $\e$ and $n$ as above, \m $\textrm{\bf d}(\xx\m,\m\Gamma_a)<
\e$ \ for all \m $\xx\in\Gamma_{2n+1}$. Clearly this implies
\rfb{first_day_of_boys_in_UK_school}.

{\bf Case 2:} \m $\tilde\varphi_{high}=2m\pi+\psi_1$, so that
$(\tilde\varphi_{high},0)$ is a locally asymptotically stable (in
particular, Lyapunov stable) equilibrium point of \rfb{pend_max}.
From the Lyapunov stability, for every $\e>0$ there exists $\delta
_\e\in(0,\e)$ such that the following holds: if, for some $\varphi
\in\rline$ and $T>0$, \vspace{-6mm}
\BEQ{Kenilworth}
   \m\ \ \ \ \ |\psi_p(T,\varphi)-\tilde\varphi_{high}| + |\dot
   \psi_p(T,\varphi)| \m<\m \delta_\e \m,
\end{equation}
then $|\psi_p(t,\varphi)-\tilde\varphi_{high}| + |\dot\psi_p(t,\varphi
)|<\e$ for all $t\geq T$. For some $\e>0$, let $T>0$ be such that
\rfb{Kenilworth} holds, with $\varphi_{low}$ in place of $\varphi$ and
$\delta_\e/2$ in place of $\delta_\e$. Using again the standard result
on the continuous dependence of solutions of ODEs on their initial
conditions, there exists an $N_\e\in\nline$ such that for all $n\in
\nline$ with $2n\geq N_\e$, \rfb{Oxfam_suitcase} holds with $\delta_\e
/2$ in place of $\e$. Using the Lyapunov stability, this implies that
\m $|\psi_p(t,\varphi^{2n})-\tilde\varphi_{high}|+|\dot\psi_p(t,
\varphi^{2n})|<\e$ \m for all $t\geq T$. This implies that \m $\textrm
{\bf d}(\xx\m,\m\Gamma_a)<\e$ \ for all $\xx$ in the phase plane curve
of $\psi_p(\cdot,\varphi^{2n})$ (for positive time), and in particular
for all \m $\xx\in\Gamma_{2n+1}$ Thus, we have proved
\rfb{first_day_of_boys_in_UK_school} also in the second case.

For the remainder of this proof, for any $\varphi\in\rline$ we denote
by $\psi_n(\cdot,\varphi)$ the solution $\psi_n$ of \rfb{pend_min}
satisfying $\psi_n(0)=\varphi$ and $\dot\psi_n(0)=0$. Denote the
segment of the curve $(\psi_n(\cdot,\varphi_{high}),\dot\psi_n(\cdot,
\varphi_{high}))$ between $(\varphi_{high},0)$ and its first positive
acceleration point $(\tilde\varphi_{low},0)$ by $\Gamma_b$. We claim
that for any $\e>0$ there exists $N_\e\in\nline$ such that
\vspace{-2mm}
\BEQ{second_day_of_boys_in_UK_school}
   \mbox{ if }\ \ 2n \m\geq\m N_\e \m,\ \ \mbox{ then }\ \ \textrm
   {\bf d}(\xx\m,\m\Gamma_b) \m<\m \e \FORALL \xx\in\Gamma_{2n} \m.
\vspace{-2mm}\end{equation}
This implies that $\tilde\vp_{low}=\vp_{low}$. The proof of these
facts is similar to the proof of \rfb{first_day_of_boys_in_UK_school},
by replacing everywhere $\psi_p$ with $-\psi_n$, \rfb{pend_max} with
\rfb{pend_min}, $m$ with $-m$, $\vp^{2n}$ with $-\vp^{2n+1}$,
$\vp_{low}$ with $-\vp_{high}$ and viceversa, $\tilde\vp_{high}$ with
$-\tilde\vp_{low}$ and $\Gamma_a$ with $-\Gamma_b$.

It follows from $\tilde\varphi_{high}=\varphi_{high}$ and $\tilde
\varphi_{low}=\varphi_{low}$ that the union of the curves $\Gamma_a$
and $\Gamma_b$ defined above is a simple closed curve in the phase
plane (see Figure 4), which we denote by $\Gamma_c$. It follows from
\rfb{first_day_of_boys_in_UK_school} and
\rfb{second_day_of_boys_in_UK_school} that \rfb{Sainsbury} holds.
\end{proof}

\begin{remark}
Putting together \rfb{Putin}, the obvious $\varphi^2\leq\vp_{low}$
and the lower estimate for $\vp^2$ from Lemma \ref{neg_acc}, we have
(with the notation of the last proof) \vspace{-2mm}
$$ (2m-1)\pi+|\psi_2|<\vp_{low}\leq2m\pi+\psi_2<2m\pi+\psi_1\leq
   \vp_{high}<(2m+1)\pi-|\psi_1| \m.\vspace{-2mm}$$
\end{remark}

Recall the curves $\Gamma_a$ and $\Gamma_b$ introduced in the last
proof. We now show that, under some conditions, the regions in the
phase plane enclosed by $\Gamma_a$ and the horizontal axis (and by
$\Gamma_b$ and the horizontal axis) are convex, as illustrated in
Figure 4. Using this, in Lemma \ref{velocity} we derive upper bounds
for the heights of $\Gamma_a$ and $\Gamma_b$. These bounds are then
used to derive an estimate for $\varphi_{high}-\varphi_{low}$.

\begin{lemma} \label{convexity}
Let $\alpha$ satisfy \rfb{assumption} and $m\in\zline$. Recall
$\varphi_{low},\m\varphi_{high},\m\Gamma_a,\m\Gamma_b$ introduced in
the last proof. Denote the closed subsets of the phase plane enclosed
by the curve $\Gamma_a$ and the horizontal axis by $\Delta_a$, and by
the curve $\Gamma_b$ and the horizontal axis by $\Delta_b$. If
$\varphi_{high}\neq 2m\pi+\psi_1$, then $\Delta_a$ is convex. If
$\varphi_{low}\neq 2m\pi+\psi_2$, then $\Delta_b$ is convex.
\vspace{-7mm} \m
\end{lemma}

\begin{proof} \vspace{6mm}
In this proof, we denote by $\psi_p$ the solution of \rfb{pend_max}
corresponding to the initial condition $(\psi_p(0),\dot\psi_p(0))=
(\varphi_{low},0)$ and let $\tau$ be the time that it takes $(\psi
_p,\dot\psi_p)$ to reach $(\varphi_{high},0)$ (while moving along
$\Gamma_a$). We assume that $\varphi_{high}\neq 2m\pi+\psi_1$. Hence
by Lemma \ref{pos_acc}, $\tau<\infty$ and $\dot\psi_p(t)>0$ for all
$t\in(0,\tau)$. We consider the function $\psi_p$ only on the
interval $[0,\tau]$. For $\varphi\in[\varphi_{low},\varphi_{high}]$
we denote $f(\varphi)=\dot\psi_p(t)|_{\psi_p(t)=\varphi}$, so that
$\Gamma_a$ is the graph of $f$. The slope of $f$ is, according
to \rfb{phase_eqn_max}, \vspace{-2mm}
\BEQ{phase_max}
  f'(\varphi) \m=\m \frac{\dd\dot\psi_p(t)}{\dd\psi_p(t)} \bigg|
  _{\psi_p(t)=\varphi} \m=\m -\alpha + \frac{\beta + d-\sin\varphi}
  {f(\varphi)} \FORALL \varphi\in(\varphi_{low},\varphi_{high}) \m.
\end{equation}

\m\vspace{-10mm}
$$\includegraphics[scale=0.18]{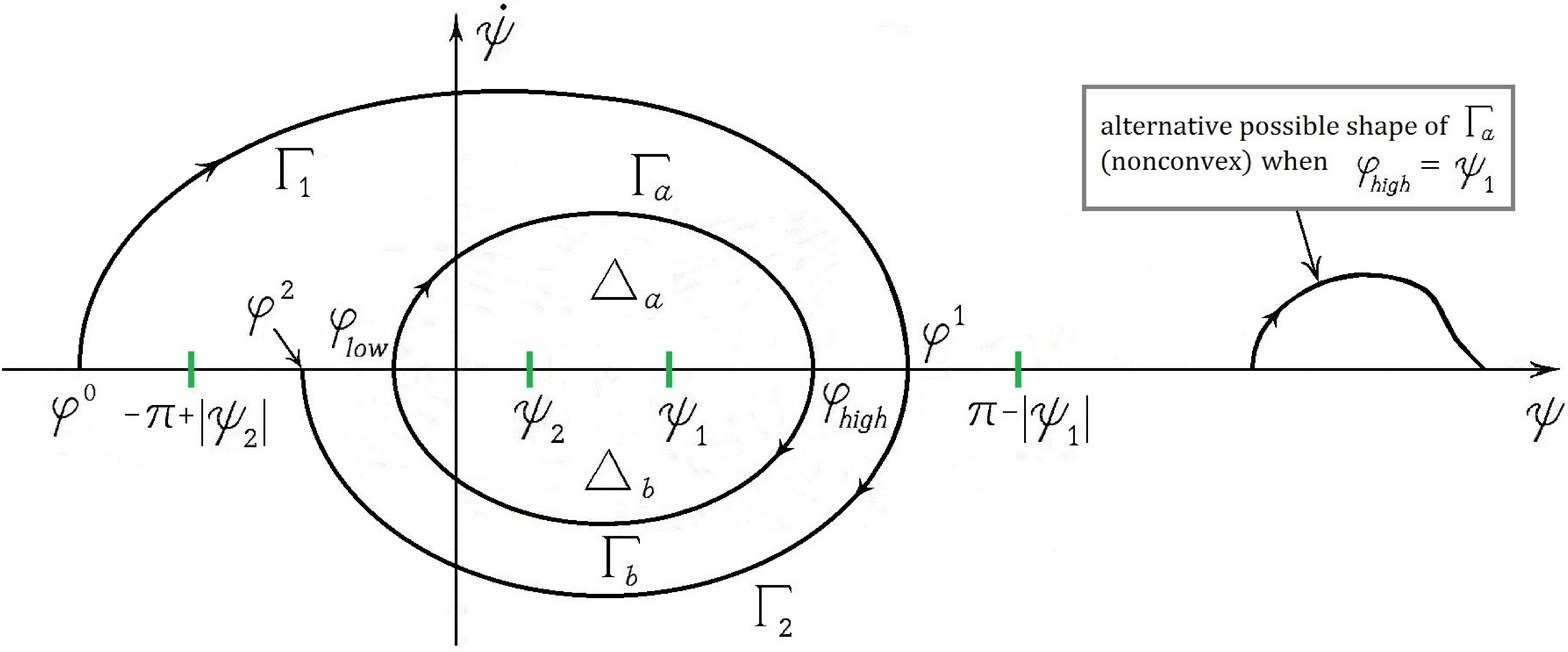}
\vspace{-1mm}$$ \vspace{-2mm}\centerline{ \parbox{5.3in}{
Figure 4. Possible shape of the curves $\Gamma_1$, $\Gamma_2$,
$\Gamma_a$ and $\Gamma_b$ in the phase plane, and of the sets
$\Delta_a$ and $\Delta_b$, when \rfb{assumption} holds and $m=0$.
We have shown the case when none of the equalities in
\rfb{Warwick_sunshine} holds and $\psi_2>0$.}}

\bigskip
We claim that $f''(\varphi)\leq 0$ for all $\varphi\in(\varphi_{low},
\varphi_{high})$. It can be shown by a somewhat tedious computation
that for every $\varphi\in(\varphi_{low},\varphi_{high})$,
\BEQ{Abbas_UN_bomb}
   f''(\varphi) \m=\m -\frac{f(\varphi) \cos\varphi + (\beta+d-\sin
   \varphi)f'(\varphi)}{f(\varphi)^2} \m,
\end{equation}
\BEQ{F_secondder}
   f'''(\varphi) \m=\m -\frac{3f'(\varphi)+\alpha}{f(\varphi)} \cdot
   f''(\varphi) + \frac{\sin\varphi}{f(\varphi)} \m.
\end{equation}

Suppose that our claim is false. Then $f''(\vp_0)>0$ for some $\vp_0
\in(\vp_{low},\vp_{high})$. Using \rfb{Abbas_UN_bomb} and the facts that
$f(\vp_{low})=0$, $\vp_{low}<2m\pi+\psi_2$ (see Lemma \ref{neg_acc})
and $\psi_2<\psi_1$, is easy to verify that for a sufficiently small
$\e>0$, \m $f''(\vp)<0$ for each $\vp\in(\vp_{low}, \vp_{low}+\e)$.
Hence there exists $\eta_{low}\in(\vp_{low},\vp_0)$ such that
$$f''(\eta_{low}) \m=\m 0 \m,\qquad f'''(\eta_{low}) \m\geq\m 0\m.$$
This, using \rfb{F_secondder}, implies that $\sin\eta_{low}\geq 0$.
Since $\vp_{low}\geq(2m-1)\pi+|\psi_2|$ (see Lemma \ref{neg_acc}), we
conclude that $\eta_{low}\in[2m\pi,\vp_0)$ and hence $\vp_0>2m\pi$.

According to \rfb{Putin} we have $\vp_{high}\in(2m\pi+\psi_1,(2m+1)\pi
-|\psi_1|)$. It follows from \rfb{Abbas_UN_bomb} (using $f(\vp_{high})
=0$) that for a sufficiently small $\e>0$ we have $f''(\vp)<0$ for all
$\vp\in(\vp_{high}-\e,\vp_{high})$. Therefore there exists $\eta_{high}
\in(\vp_0,\vp_{high})$ such that \vspace{-2mm}
$$f''(\eta_{high}) \m=\m 0\m,\qquad f'''(\eta_{high}) \m\leq\m 0\m.$$
Using \rfb{F_secondder}, this gives us that $\sin\eta_{high}\leq 0$.
Since, according to our earlier steps, $\eta_{high}\in(2m\pi,
(2m+1)\pi)$, this is a contradiction, proving our claim. By a well
known fact in analysis, our claim implies that $\Delta_a$ is a convex
set.

When $\varphi_{low}\neq2m\pi+\psi_2$, the convexity of $\Delta_b$
can be established similarly.
\end{proof}

\begin{lemma} \label{velocity}
Let $m,\m\alpha,\m\Gamma_a,\m\Gamma_b,\m\varphi_{low}$ and $\varphi
_{high}$ be as in Lemma {\rm\ref{convexity}}. Define \vspace{-2mm}
$$ v_a \m=\m \max_{(\psi_p,\dot\psi_p)\in\Gamma_a} \dot
   \psi_p \m, \qquad v_b \m=\m \max_{(\psi_n,\dot\psi_n)
   \in\Gamma_b} |\dot\psi_n| \m$$
Recall $\, d,\psi_1$ and $\psi_2$ from \rfb{t1-t2}. Then the following
relations hold: \vspace{-1mm}
\begin{align}
   &\varphi_{high} \m\neq\m 2m\pi+\psi_1 \ \implies v_a \m<\m
   \frac{4d}{\alpha} \m,\label{est2}\\
   &\varphi_{low} \m\neq\m 2m\pi+\psi_2 \ \ \implies v_b \m<\m
   \frac{4d}{\alpha} \m,\label{est3}\\
   &\varphi_{high} \m\neq\m 2m\pi+\psi_1 \m, \ \ \varphi_{low}
   \m\neq\m 2m\pi+\psi_2 \implies v_a+v_b \m<\m \frac{4d}{\alpha}
   \m.\label{est1}
\end{align}
\m\ \ Furthermore \vspace{-4mm}
\BEQ{angle_est1}
   \varphi_{high}-\varphi_{low} \m<\m \psi_1-\psi_2+\frac{4d}
   {\alpha^2} \m,
\end{equation}
\vspace{-3mm}
\BEQ{Friday_visit}
  \varphi_{high} \m<\m \psi_1+2m\pi+\frac{4d}{\alpha^2} \m,\qquad
  \varphi_{low}  \m>\m 2m\pi+\psi_2-\frac{4d}{\alpha^2} \m.
\end{equation}
\end{lemma}

\begin{proof}
First assume that $\varphi_{high}\neq 2m\pi+\psi_1$. Then $\Gamma_c$,
the closed curve that is the union of $\Gamma_a$ and $\Gamma_b$
(introduced in Lemma \ref{lm:env_traj}), can be regarded as a segment
of the curve corresponding to the solution $(\psi,\dot\psi)$ of
\rfb{pendulum} when $(\psi(0),\dot\psi(0))=(\varphi_{low},0)$ and
$\gamma$ is given by \rfb{smoke_alarm}. This solution takes a finite
time $\tau$ to reach $(\vp_{high},0)$ and then a possibly infinite
amount of time to return to $(\vp_{low},0)$. We denote by $\tau_c$ the
(possibly infinite) time that it takes for $(\psi,\dot\psi)$ to go
around the closed curve $\Gamma_c$.

Recall the function $f:[\m\varphi_{low},\varphi_{high}]\rarrow[0,
\infty)$ introduced before \rfb{phase_max}, so that $\Gamma_a$ is the
graph of $f$. Similarly, we introduce $g:[\varphi_{low},\varphi
_{high}]\rarrow(-\infty,0\m]$ so that $\Gamma_b$ is the graph of $g$.
Since $E(0)=\lim_{t\to\tau_c}E(t)$, it follows from
\rfb{energy_phaseplane} that \vspace{-1mm}
\BEQ{energy_gammac}
   \int_{\varphi_{low}}^{\varphi_{high}} \alpha [f(\varphi)-g
   (\varphi)] \m\dd\varphi \m=\m 2d (\varphi_{high}-\varphi_{low})\m.
   \vspace{-1mm}
\end{equation}

Let $\psi_+\in(\vp_{low},\vp_{high})$ be the angle at which $f$
reaches its maximum $v_a$. From Lemma \ref{convexity} we know that
$\Delta_a$ is convex. Hence, the triangle in the phase plane with
vertices $(\vp_{low},0)$, $(\vp_{high},0)$ and $(\psi_+,v_a)$ lies
inside $\Delta_a$. Therefore \vspace{-2mm}
\BEQ{damp_est1}
   \int_{\varphi_{low}}^{\varphi_{high}} f(\vp) \m\dd\vp \m\geq\m
   \frac{v_a}{2} (\varphi_{high}-\varphi_{low}) \m.\vspace{-1mm}
\end{equation}
This and \rfb{energy_gammac} imply that \m $2d(\vp_{high}-\vp_{low})
>\frac{\alpha v_a}{2}(\vp_{high}-\vp_{low})$, whence \rfb{est2}.

Now replace the assumption $\varphi_{high}\neq 2m\pi+\psi_1$ with
$\varphi_{low}\neq 2m\pi+\psi_2$. By repeating the above arguments
after \rfb{energy_gammac}, but using the function $g$ and the convex
set $\Delta_b$, we get that \vspace{-2mm}
\BEQ{damp_est2}
   \int_{\varphi_{low}}^{\varphi_{high}} -g(\vp) \m\dd\vp \m\geq\m
   \frac{v_b}{2} (\varphi_{high}-\varphi_{low})  \vspace{-1mm}
\end{equation}
which, together with \rfb{energy_gammac}, implies \rfb{est3}. Finally
if $\vp_{high}\neq 2m\pi+\psi_1$ and $\vp_{low}\neq 2m\pi+\psi_2$,
then both \rfb{damp_est1} and \rfb{damp_est2} hold, which together
with \rfb{energy_gammac} imply \rfb{est1}.

Next we will derive \rfb{angle_est1}. First assume that $\vp_{high}
\neq 2m\pi+\psi_1$. Let $\psi_p$ and $\tau$ be as at the beginning of
the proof of Lemma \ref{convexity} (so that $\psi_p:[0,\tau]\rarrow
[0,v_a]$, $\psi_p(0)=\vp_{low}$ and $\psi_p(\tau)=\vp_{high}$). Let
$\tau_1\in[0,\tau]$ be such that $\psi_p(\tau_1)=2m \pi+\psi_1$.
Using the energy from \rfb{energy_fnc} (and \rfb{energy_phaseplane}
with $d$ in place of $\gamma$) we get \vspace{-2mm}
$$ \frac{\dot\psi_p(\tau_1)^2}{2} \m=\m E(\tau_1)-E(\tau)+
   \cos\psi_1-\cos\varphi_{high} \hspace{42mm} \m\vspace{-6mm}$$
\begin{align}
  &\m\qquad=\m \int_{2m\pi+\psi_1}^{\vp_{high}} \alpha\dot\psi_p|
  _{\psi_p=\vp} \m\dd\vp + \int_{2m\pi+\psi_1}^{\vp_{high}}
  (\sin\varphi-\beta-d) \dd\varphi \nonumber\\[3pt]
  &\m\qquad\geq\m \frac{\alpha}{2}(\varphi_{high}-\psi_1-2m\pi)\dot\psi_p
  (\tau_1)+\int_{2m\pi+\psi_1}^{\varphi_{high}} (\sin\varphi-\beta
  -d)\dd\varphi \m. \label{ineq_for_est}
\end{align}

\vspace{-3mm}
To derive the last inequality, we have used the convexity of the
set $\Delta_a$. Since the integral term in \rfb{ineq_for_est}
is positive it follows that if $\vp_{high}\neq 2m\pi+\psi_1$ then
\vspace{-2mm}
\BEQ{BBQ_at_X1}
  \varphi_{high}-\psi_1-2m\pi \m<\m \frac{\dot\psi_p(\tau_1)}
  {\alpha} \m\leq\m \frac{v_a}{\alpha} \m.\vspace{-2mm}
\end{equation}

Next assume that $\varphi_{low}\neq 2m\pi+\psi_2$. Doing a similar
argument as we did to derive \rfb{BBQ_at_X1}, but now working on the
curve $\Gamma_b$ instead of $\Gamma_a$, and using the convexity of
$\Delta_b$, we get that \vspace{-3mm}
\BEQ{BBQ_at_X2}
   2m\pi+\psi_2-\varphi_{low} \m<\m \frac{v_b}{\alpha} \m.
\end{equation}
Finally, \rfb{angle_est1} follows by adding \rfb{BBQ_at_X1} and
\rfb{BBQ_at_X2}, using \rfb{est2}-\rfb{est1}. The inequalities
\rfb{Friday_visit} follow immediately from \rfb{BBQ_at_X1} and
\rfb{BBQ_at_X2}, using \rfb{est2}-\rfb{est3}.
\end{proof}

\begin{lemma} \label{why_trajenv}
We use the assumptions and the notation of Definition {\rm
\ref{traj_env}}. For each even $k\in\nline$, let $\Delta_k$ be the
closure of the set encircled by the curves $\Gamma_{k-1}$, $\Gamma_k$
and the line $L_{\m k}$ joining the points $(\vp^{k-2},0)$ and
$(\vp^k,0)$. Then for any solution $\psi$ of \rfb{pendulum}, if
$(\psi(0),\dot \psi(0)) \in\Delta_k$, then $(\psi(t),\dot\psi(t))\in
\Delta_k$ for all $t\geq 0$. \vspace{-2mm}
\end{lemma}

\begin{proof}
Fix $k$ and let $\psi$ be a solution of \rfb{pendulum} with $(\psi(0),
\dot\psi(0))\in\Delta_k$. It follows from Lemma \ref{traj_above} that
the curve $(\psi,\dot\psi)$ cannot go out of $\Delta_k$ by crossing
the curve $\Gamma_{k-1}$ and from Lemma \ref{traj_below} that it
cannot go out of $\Delta_k$ by crossing $\Gamma_k$. It is easy to
check that the curve $(\psi,\dot\psi)$ cannot escape through $L_{\m
k}$, because if it is on $L_{\m k}$, then the velocity $\dot\psi$
along the curve starts to increase, forcing the curve to stay within
$\Delta_k$.
\end{proof}

\begin{theorem} \label{main_result_sec5}
Using the notation from \rfb{t1-t2}, suppose that $\alpha$ satisfies
\rfb{assumption}. Then for every solution $\psi$ of \m \rfb{pendulum},
there exists a $T>0$ such that for any $t_1,t_2>T$, \vspace{-3mm}
\BEQ{convergence_est1}
  |\psi(t_1)-\psi(t_2)| \m<\m \psi_1-\psi_2+\frac{4d}{\alpha^2} \,.
\vspace{-1mm} \end{equation}
Moreover, for some integer $m$ and each $t>T$, one of the
following expressions hold: \vspace{-2mm}
\BEQ{final_estimate1}
  2m\pi+\psi_2-\frac{4d}{\alpha^2} \m<\m \psi(t) \m<\m 2m\pi+
  \psi_1+\frac{4d}{\alpha^2} \m,
\end{equation}
\BEQ{final_estimate2}
  (2m+1)\pi-\psi_1 \m<\m \psi(t) \m<\m (2m+1)\pi-\psi_2 \m.
\end{equation} \m\vspace{-8mm}
\end{theorem}

\begin{proof}
We call a solution $\psi$ of \rfb{pendulum} {\em oscillating}
\m if for each $T>0$ there exists $t_1,t_2>T$ such that $\dot\psi
(t_1)>0$ and $\dot\psi(t_2)<0$.

First consider the case when the solution $\psi$ is not oscillating
(hence it is eventually non-increasing or non-decreasing).  It then
follows from Proposition \ref{similar_hayes} that $\psi$ must remain
bounded and hence it converges to a finite limit, $\lim_{t\to\infty}
\psi(t)=\psi_\infty$, which trivially implies \rfb{convergence_est1}.
It follows from \rfb{pendulum} and \rfb{JF_election} that $\ddot\psi$
is a continuous bounded function of time. So we can apply Barb\u
alat's lemma (see \cite{FaWe:15,Kha:02,LoRy:04}) to $\dot\psi$ to
conclude that $\lim_{t\to\infty}\dot\psi(t)=0$. Using this and taking
upper and lower limits in \rfb{pendulum}, we get \vspace{-2mm}
$$ \limsup_{t\to\infty} \ddot\psi(t) \m=\m \beta + \limsup_{t\to
   \infty} \gamma(t)-\sin\psi_\infty \m\geq\m 0 \m,\vspace{-2mm}$$
$$ \liminf_{t\to\infty} \ddot\psi(t) \m=\m \beta + \liminf_{t\to
   \infty} \gamma(t)-\sin\psi_\infty \m\leq\m 0 \m,$$
whence \ $\beta-d<\beta+\liminf_{t\to\infty}\gamma(t)\leq\sin\psi
_\infty\leq\beta+\limsup_{t\to\infty}\gamma(t)<\m \beta+d$\m. In
short, \m $\sin\psi_1<\sin\psi_\infty<\sin\psi_2$\m. This implies that
either \rfb{final_estimate1} or \rfb{final_estimate2} hold.

Next suppose that $\psi$ is oscillating. We claim that $\psi$ will
eventually be captured in one of two types of bounded intervals.
Specifically, one of the following holds: \vspace{-1mm}

{\noindent \hangindent 12pt (i) For some $m\in\zline$ and all $t$
large enough, $\psi(t)\mm\in\mm((2m-1)\pi-\psi_2,(2m+1)\pi-\psi_1
\nmm)$ and moreover, \m $(\psi(t),\dot\psi(t))\in\Delta_2$.
\vspace{-1mm}

\noindent (ii) For some $m\in\zline$ and all $t$
large enough, $\psi(t)\nmm\in\nmm[(2m-1)\pi-\psi_1,(2m-1)\pi-\psi_2]$.
\vspace{-3mm}}

To prove the above claim, suppose that (ii) does not hold. Then for
any $T>0$ there must exist times $\tau_2>\tau_1>T$ such that $\dot\psi
(\tau_1)=0$, $\dot\psi(t)\not=0$ for $t\in(\tau_1,\tau_2]$ and
$\psi(\tau_2)\not\in[(2m-1)\pi-\psi_1,(2m-1)\pi-\psi_2],\ \forall\,m
\in\zline$. Without loss of generality, we may assume that $\dot\psi
(\tau_2)>0$, since for $\dot\psi(\tau_2)<0$ the argument is similar.
Then $\dot\psi(t)>0$ for all $t\in(\tau_1,\tau_2]$, whence $\ddot\psi
(\tau_1)\geq 0$ and $\psi(\tau_2)>\psi(\tau_1)$.

Since $\dot\psi(\tau_1)=0$ and $\ddot\psi(\tau_1)\geq 0$, from \rfb
{pendulum} we get $\sin\psi_1>\sin(\psi(\tau_1))$, so that $\psi(\tau
_1)\in((2m-1)\pi-\psi_1,2m\pi+\psi_1)$ for some $m\in\zline$. Therefore
$\psi(\tau_2)>(2m-1)\pi-\psi_2$ and $(\psi(\tau_1),0)$ is a positive
acceleration point. Let $\tau_3=\min\{\m t>\tau_1\,\big|\,\dot\psi(t)
=0\}$. Then $\ddot\psi(\tau_3)\leq 0$, which using \rfb{pendulum} implies
that $\sin(\psi(\tau_3))>\sin\psi_2$. Since $\psi(\tau_3)>\psi(\tau_2)$,
this implies that $\psi(\tau_3)>2m\pi+\psi_2$. Let $\psi_p$ be the
solution of \rfb{pend_max} when $(\psi_p(0),\dot\psi_p(0))=(\psi
(\tau_1),0)$. From Lemma \ref{traj_above} we have $\dot\psi_p|_{\psi
_p=\psi(t)}>\dot\psi(t)$ for each $t\in(\tau_1,\tau_3)$. This and the
fact that the first negative acceleration point $(\varphi,0)$ for
$(\psi_p,\dot\psi_p)$ is such that $\vp<(2m+1)\pi-|\psi_1|$ (see Lemma
\ref{pos_acc}) imply that $\psi(\tau_3)\in(2m\pi+\psi_2,(2m+1)\pi-
|\psi_1|)$. We now have to consider two cases:

{\bf Case (a)} \m If $\psi(\tau_3) \in(2m\pi+\psi_2,2m\pi+\psi_1]$
then, since the line joining the points $(2m\pi+\psi_2,0)$ and
$(2m\pi+\psi_1,0)$ in the phase plane is contained in $\Delta_2$,
we get that $(\psi(\tau_3),0) \in\Delta_2$. According to Lemma
\ref{why_trajenv}, $(\psi(t),\dot\psi(t))$ remains in $\Delta_2$
for all $t\geq\tau_3$, and moreover $(\psi(t),\dot\psi(t))$ cannot
reach the corner $(\vp^0,0)\in\Delta_2$, so that (i) holds.

{\bf Case (b)} \m If $\psi(\tau_3)\in(2m\pi+\psi_1,(2m+1)\pi-|\psi_1|)
$, then from \rfb{pendulum} we get that $\ddot\psi(\tau_3)<0$. In this
case, we have to do one more iteration: Denote $\tau_4=\min\{t>\tau_3
\,|\,\dot\psi(t)=0\}$, then $\dot\psi(t)<0$ for all $t\in(\tau_3,\tau
_4)$. By a reasoning similar to the one used before case (a), using
Lemmas \ref{traj_below} and \ref{neg_acc}, we get that $\psi(\tau_4)
\in((2m-1)\pi+|\psi_2|,2m\pi+\psi_1)$, so that $(\psi(\tau_4),0)\in
\Delta_2$. By the same argument as employed in case (a), this implies
that (i) holds. Thus, we have proved our claim.

First we consider the case when (ii) holds. Then actually $\psi(t)
\nmm\in\nmm ((2m-1)\pi-\psi_1,(2m-1)\pi-\psi_2)$ for some $m\in
\zline$ and all $t$ large enough. Indeed, for $t$ large enough, $\psi
(t)$ can no longer reach the endpoints of the interval in (ii),
because at the endpoints we would have $\dot\psi(t)=0$ and from
\rfb{pendulum} we see that the acceleration $\ddot\psi(t)$ would force
$\psi(t)$ to leave the interval. Now \rfb{convergence_est1} and
\rfb{final_estimate2} follow trivially.

Now we consider the case when (i) holds. We show by induction that
given any even $k\in\nline$, there exists a $\tau_k>0$ such that
$(\psi(t),\dot \psi(t))\in\Delta_k$ for all $t\geq\tau_k$. This, along
with Lemma \ref{lm:env_traj} and \ref{velocity} will imply that
\rfb{convergence_est1} and \rfb{final_estimate1} hold.

Assume that for some even $k\in\nline$ and some $\tau_k>0$, $(\psi(t),
\dot\psi(t))\in\Delta_k$ for all $t\geq\tau_k$. Let $T>\tau_k$ be such
that $\dot\psi(T)<0$. Define $\tau_{k+2}=\min\{t>T\,|\,\dot\psi(t)=0
\}$. Then $\ddot\psi(\tau_{k+2})\geq 0$ and since $(\psi(\tau_{k+2}),
\dot\psi(\tau_{k+2}))\in\Delta_k$, it is easy to see that $\psi(\tau
_{k+2})\in[\vp^k,2m\pi+\psi_1]$ (the upper bound follows from
\rfb{pendulum}). So $(\psi(\tau_{k+2}),\dot\psi(\tau_{k+2}))\in\Delta
_{k+2}$ and (from Lemma \ref{why_trajenv}) $(\psi(t),\dot\psi(t))\in
\Delta_{k+2}$ for all $t\geq\tau_{k+2}$.
\end{proof}

\begin{remark} \label{IC_infl}
The bounds for $|\psi(t_1)-\psi(t_2)|$ and $\psi(t)$ in
\rfb{convergence_est1}, \rfb{final_estimate1} and
\rfb{final_estimate2} depend only on $\alpha$, $\beta$ and $d$ and
they are independent of the initial state $(\psi(0),\dot\psi(0))$ of
\rfb{pendulum}. Therefore, even if $d$ satisfies the less restrictive
inequality $\limsup|\gamma(t)|<d$ (instead of $\|\gamma\|_{L^\infty}
<d$), Theorem \ref{main_result_sec5} continues to hold.
\end{remark}

\begin{remark}\label{lyap_fnc}
In the pendulum equation \rfb{pendulum}, $\gamma$ can be viewed as a
bounded disturbance. Often, bounds like \rfb{final_estimate1} that
characterize the asymptotic response of dynamical systems driven by
bounded disturbances are derived using Lyapunov functions. Using
Proposition \ref{similar_hayes} and Lemmas \ref{pos_acc} and
\ref{neg_acc} it can be shown that all the solutions of \rfb{pendulum}
(with the angles measured modulo $2\pi$) are eventually in a bounded
region $\Om$ of the phase plane. Then a Lyapunov function $V$ for
\rfb{pendulum} which is positive-definite on $\Om$ and for which $\dot
V$, evaluated along the solutions of \rfb{pendulum} with $\gamma=0$,
is negative-definite on $\Om$ can be constructed (see \cite[Example
4.4]{Kha:02} for a Lyapunov function that can be used when $\beta=0$).
Using $V$ a bound like \rfb{final_estimate1} can be derived for a
given $\alpha$, $\beta$ and $d$. The main problem with this approach
is that it is hard to express the bounds thus derived as simple
functions of $\alpha$, $\beta$ and $d$. This makes it difficult, not
only to state the main result of this paper concisely using them, but
also to verify the sufficient conditions in the main result.
\end{remark}

\section{\secp Stability of the SG connected to the bus} \label{sec6}

\ \ \ In this section we derive sufficient conditions for the SG
parameters under which the system \rfb{eq:SG} is almost globally
asymptotically stable. We obtain these conditions by applying the
asymptotic bounds derived in Section \ref{sec5} for the forced
pendulum equation to the exact swing equation (ESE) in \rfb{SG_Pend}.
To this end, we first write the ESE in the standard form for forced
pendulum equations shown in \rfb{pendulum}. Recall $i_v$ from
\rfb{iv_defn} and $p=R_s/L_s$. Define $V_r,\rho$ and $P_\infty$ (all
$>0$) as follows: \vspace{-1mm}
\BEQ{Vr_rho}
   V_r \m=\m \frac{mi_f}{L_s i_v} \m,\qquad \rho \m=\m \sqrt{\frac{J}
   {mi_f i_v}} \m,\qquad P_\infty \m=\m \frac{p\o_g}{\o_g^2+p^2} \m.
\end{equation}
It will be useful to note that
$$ p\rho\lim_{s\to\infty} \int_0^s e^{-p\rho(s-\tau)} \sin(\rho\o_g
   (s-\tau))\dd\tau \m=\m p\rho \int_0^\infty e^{-p\rho\sigma} \sin
   (\rho\o_g\sigma) \dd\sigma \m=\m P_\infty \m.$$
We also introduce the constants
\BEQ{alpha_beta}
   \alpha \m=\m \frac{D_p}{\sqrt{m i_f i_v J}} \m,\qquad \beta
   \m=\m \frac{T_m-D_p\o_g}{m i_f i_v} -V_r P_\infty \m.
\end{equation}
Consider the new time variable $s=t/\rho$ and new angle variable
$\psi(s)=\eta(\rho s)$. In terms of these variables, ESE has the
following representation, equivalent to \rfb{SG_Pend}: \vspace{-1mm}
\BEQ{SG_finalform}
   \psi''(s)+\alpha\psi'(s)+\sin\psi(s) \m=\m \beta+\gamma(s) \m,
\vspace{-1mm} \end{equation}
\BEQ{gamma_defn}
   \gamma(s) \m=\m \frac{e^{-p\rho s}f(\rho s)}{i_v} + V_r P_\infty
   - V_r P(s) \m,
\vspace{-1mm} \end{equation}
\BEQ{Pd_defn}
   P(s) \m=\m p\rho \int_0^s e^{-p\rho (s-\tau)}\sin\left[\psi(s)
   -\psi(\tau)+\rho\o_g(s-\tau)\right] \dd\tau \m,
\end{equation}
where $f$ \m is the bounded function defined in \rfb{f_defn} using the
function $\o$ and the initial conditions $i_d(0),i_q(0)$ and $\delta
(0)$. In turn, $\o$ depends on $\psi$ as $\o(t)=(1/\rho)\psi'(t/\rho)
+\o_g$ (according to \rfb{smaller_K} and \rfb{psi_defn}). The global
existence of a unique solution to the system of integro-differential
equations \rfb{SG_finalform}-\rfb{Pd_defn} together with \rfb{f_defn},
for any initial conditions $\psi(0),\psi'(0),i_d(0)$ and $i_q(0)$,
follows from the global existence of unique solutions to
\rfb{SG_Pend} (see also the discussion below \rfb{SG_Pend}).

Consider a solution $(i_d,i_q,\o,\delta)$ of \rfb{eq:SG} and the
corresponding solution $\psi(s)=(3\pi/2)+\delta(\rho s)+\phi$ of
\rfb{SG_finalform}, with $f$ as in \rfb{f_defn}, $\gamma$ as in \rfb
{gamma_defn} and $P$ as in \rfb{Pd_defn}. Clearly $f$ is bounded and
$|P(s)|<1$, and both $f$ and $P$ are continuous functions. Therefore
$\gamma$ is a bounded continuous function and $\psi$ is the
corresponding solution of \rfb{SG_finalform} regarded as a forced
pendulum equation. So the bounds in Theorem \ref{main_result_sec5},
developed for solutions of forced pendulum equations, can be used to
obtain asymptotic bounds for $\psi$ in terms of any $d>0$ that
satisfies $\limsup|\gamma(s)|<d$ (see Remark \ref{IC_infl}) and
$|\beta|+d<1$. Using these asymptotic bounds, Theorem \ref{stab_cond2}
shows that under some conditions on the SG parameters
$\limsup|\gamma(s)|=0$. This implies that $(\psi, \psi')$ converges to
a limit point, using which we can conclude that $(i_d,i_q,\o,\delta)$
converges to an equilibrium point of \rfb{eq:SG}. Since the stability
conditions in Theorem \ref{stab_cond2} are independent of the initial
state of \rfb{eq:SG} we get that, whenever these conditions hold,
every solution of \rfb{eq:SG} converges to an equilibrium point.

We briefly explain the idea behind the stability conditions in
Theorem \ref{stab_cond2}. Let \vspace{-2mm}
\BEQ{Gamma_defn}
   \Gamma \m=\m (1+P_\infty)V_r \m.
\end{equation}
It is easy to see from \rfb{gamma_defn} and \rfb{Pd_defn} that
$\limsup|\gamma(s)|\leq\Gamma$. From \rfb{Pd_defn} we get that
\vspace{-4mm}
\begin{align}
  P(s) &\m=\m p\rho \int_0^s e^{-p\rho(s-\tau)}\sin\left(\int_{\tau}^s
  \o_\rho(\sigma)\dd\sigma\right)\m \dd\tau \nonumber\\[3pt]
  \implies P(s)&\m=\m p\rho \int_0^s e^{-p\rho\tau} \sin\left(\int_0^
  \tau\o_\rho(s-\sigma) \dd\sigma\right) \dd\tau \m, \label{Pd_newdefn}
\end{align}
where $\o_\rho(s)=\rho\o(\rho s)$, so that $\o_\rho(\sigma)=\psi'
(\sigma)+\rho\o_g$. We define a function $\Nscr:(0,\Gamma]\to[0,\infty
)$ as follows. Fix $d\in(0,\Gamma]$. Suppose that $\limsup|\gamma(s)|
<d$. Using the asymptotic bounds in Section \ref{sec5} and Lemma
\ref{velocity_trap} below, we derive upper and lower bounds for $\psi'
(s)$ that are valid for large $s$. Using these bounds, the expression
$\o_\rho(s)=\psi'(s)+\rho\o_g$ and \rfb{Pd_newdefn}, we derive an
upper bound $P_u^d$ and a lower bound $P_l^d$ for $P(s)$, which is
again valid for large $s$ (this step uses Lemma \ref{bound}). Finally
we define \m $\Nscr(d)=V_r\max\{P_u^d-P_\infty,P_\infty-P_l^d\}$. It
is clear from \rfb{gamma_defn} that $\limsup|\gamma(s)|\leq\Nscr(d)$.
Our stability condition is $\Nscr(d)<d$ for all $d\in(0,\Gamma]$,
from which we can conclude that $\limsup|\gamma(s)|=0$ (see Theorem
\ref{stab_cond2} for details).

\begin{lemma} \label{velocity_trap}
Suppose $\beta\in\rline$, $d>0$ and $\alpha>0$. If $|\beta|+d<1$ and
\rfb{assumption} holds, where $\psi_1,\m\psi_2$ are as in \rfb{t1-t2},
then define $\phi_1,\m\phi_2\in[-\pi/2,\pi/2]$ by \vspace{-1mm}
$$ \phi_1 \m=\m \min\left\{\frac{\pi}{2},\psi_1+\frac{4d}{\alpha^2}
   \right\} \m,\qquad \phi_2 \m=\m \max\left\{-\frac{\pi}{2},\psi_2
   -\frac{4d}{\alpha^2}\right\}$$
and let $S_n=-\sin\phi_1$ and $S_p=-\sin\phi_2$. In all other cases
let $S_n=-1$ and $S_p=1$. Define \vspace{-2mm}
\BEQ{velocity_bound}
  \o_n \m=\m \frac{S_n+\beta-d}{\alpha} \m,\qquad
  \o_p \m=\m \frac{S_p+\beta+d}{\alpha} \m.\vspace{1mm}
\end{equation}
Assume that $\gamma:[0,\infty)\rarrow\rline$ is continuous, $\limsup
|\gamma(s)|<d$ and $\psi$ is a corresponding solution of
\rfb{SG_finalform}. Then for some $T>0$ we have \vspace{-1mm}
\BEQ{velocity_limit}
  \o_n \m<\m \psi'(s) \m<\m \o_p \FORALL s\geq T \m.
  \vspace{2mm}
\end{equation}
\end{lemma}

\begin{proof}
We claim that there exists $T_0>0$ such that
\BEQ{forcing_bound}
  S_n+\beta- d \m<\m -\sin\psi(s)+\beta+\gamma(s) \m<\m S_p+\beta
  + d \FORALL s\geq T_0 \m.
\end{equation}
Let $\tau>0$ be such that $|\gamma(s)|<d$ for all $s>\tau$. When
$|\beta|+d\geq1$ or when \rfb{assumption} is false, then it is easy
to see that \rfb{forcing_bound} holds with $T_0=\tau$. When $|\beta|
+d<1$ and \rfb{assumption} holds, then it follows by
applying Theorem \ref{main_result_sec5} and Remark \ref{IC_infl} to
\rfb{SG_finalform} that there exists $T_1>0$ such that for all $s
\geq T_1$ either \rfb{final_estimate1} or \rfb{final_estimate2} holds
(with $s$ in place of $t$). This implies that \rfb{forcing_bound}
holds with $T_0=\max\{\tau,T_1\}$. By regarding \rfb{SG_finalform}
as a stable first order linear dynamical system with state $\psi'$ and
external forcing $-\sin\psi(s)+\beta+\gamma(s)$, it can be easily
verified using \rfb{forcing_bound} that \rfb{velocity_limit} holds.
\end{proof}

The next lemma derives an upper bound and a lower bound for $P(s)$,
given an upper and lower bound for $\o_\rho$.

\begin{lemma} \label{bound}
Suppose that there exist constants $\o_{\max}>0$ and $\o_{\min}>0$
such that
\BEQ{omega_rel}
  \o_{\min}\leq\o_{\max}\leq2\o_{\min} \qquad and\qquad \o_{\min}
  \m\leq\m\o_{\rho}(\sigma)\m\leq\m\o_{\max}
\end{equation}
for all $\sigma\geq0$. Let $T_{\max}=2\pi/\o_{\max}$ and $T_{\min}=
2\pi/\o_{\min}$. Define the function $g$ on the interval $[0,T_{\max}]$
and the function $h$ on the interval $[0,T_{\min}]$ as follows\m:
\BEQ{g_defn}
  g(\tau) \m=\m \left\{
   \begin{array}{l l}
      \vspace{1mm} \sin(\o_{\max}\tau) \qquad& if \qquad 0\leq\tau<
      \frac{\pi} {2\o_{\max}}\\
      \vspace{1mm} 1 \qquad& if \qquad \frac{\pi}{2\o_{\max}}\leq
      \tau<\frac{\pi} {2\o_{\min}}\\
      \vspace{1mm} \sin(\o_{\min}\tau) \qquad& if \qquad\frac{\pi}
      {2\o_{\min}}\leq\tau<\frac{3\pi}{\o_{\min}+\o_{\max}}\\
      \vspace{1mm} \sin(\o_{\max}\tau)\qquad& if \qquad\frac{3\pi}
      {\o_{\min}+\o_{\max}}\leq\tau\leq\frac{2\pi}{\o_{\max}}
  \end{array} \right. ,
\end{equation}
\BEQ{h_defn}
  h(\tau) \m=\m \left\{
   \begin{array}{l l}
      \vspace{1mm} \sin(\o_{\min}\tau) \qquad& if \qquad 0\leq\tau<
      \frac{\pi}{\o_{\min}+\o_{\max}}\\
      \vspace{1mm} \sin(\o_{\max}\tau) \qquad& if \qquad \frac{\pi}
      {\o_{\min}+\o_{\max}} \leq\tau<\frac{3\pi} {2\o_{\max}}\\
      \vspace{1mm} -1 \qquad& if \qquad\frac{3\pi}{2\o_{\max}}\leq
      \tau<\frac{3\pi}{2\o_{\min}}\\
      \vspace{1mm} \sin(\o_{\min}\tau) \qquad& if \qquad \frac{3\pi}
      {2\o_{\min}}\leq\tau\leq\frac{2\pi}{\o_{\min}}
  \end{array} \right. .
\end{equation}
Define $T=T_{\max}$ if $\int_0^{T_{\min}} e^{-p\rho\tau}h(\tau)\dd
\tau<0$ and $T=T_{\min}$ otherwise. Then for any $s\geq 0$, $P(s)$
from \rfb{Pd_newdefn} satisfies the following bounds\m:
\BEQ{UB_eqn}
  P(s) \m\leq\m \frac{p\rho}{1-e^{-p\rho T_{\max}}} \int_0^{T_
  {\max}} e^{-p\rho\tau} g(\tau)\dd\tau + e^{-p\rho(s-T_{\min})} \m,
\end{equation}
\BEQ{LB_eqn}
  P(s) \m\geq\m \frac{p\rho}{1-e^{-p\rho T}} \int_0^{T_{\min}}
  e^{-p\rho\tau}h(\tau)\dd\tau - e^{-p\rho(s-T_{\min})} \m.
\end{equation}
\end{lemma}

\begin{proof}
Using the assumption $0<\o_{\min}\leq\o_{\max}\leq2\o_{\min}$, it is
easy to verify that
$$ \frac{\pi}{2\o_{\max}}\leq\frac{\pi}{2\o_{\min}}\leq\frac{3\pi}
   {\o_{\min}+\o_{\max}}\leq\frac{2\pi}{\o_{\max}}\m, \hspace{6mm}
   \frac{\pi}{\o_{\min}+\o_{\max}}<\frac{3\pi}{2\o_{\max}}\leq
   \frac{3\pi}{2\o_{\min}}<\frac{2\pi}{\o_{\min}} \m.$$
This means that $g$ and $h$ in \rfb{g_defn} and \rfb{h_defn} are
defined precisely on the intervals $[0,T_{\max}]$ and $[0,T_{\min}]$.
Given $s\geq0$, fix $\m0=\tau_0<\tau_1<\ldots\tau_n\leq s$ such that
$$ \int_0^{\tau_k}\o_\rho(s-\sigma)\dd\sigma \m=\m 2k\pi \FORALL k\in
   \{1,2,\ldots n\} \quad {\rm and} \quad  \int_{\tau_n}^s\o_\rho(s-
   \sigma)\dd\sigma<2\pi \m.$$
From \rfb{omega_rel} we get that for each $0\leq k\leq n$ and all
$\tau_k\leq\tau\leq s$,
\BEQ{mean_value}
  \m\qquad\ \o_{\min}(\tau-\tau_k)\leq\int_{\tau_k}^\tau\o_\rho(s-
  \sigma)\dd\sigma \m\leq\m \o_{\max}(\tau-\tau_k) \m.
\end{equation}

\m\vspace{-8mm} \vbox{
\begin{center}
\includegraphics[scale=0.73]{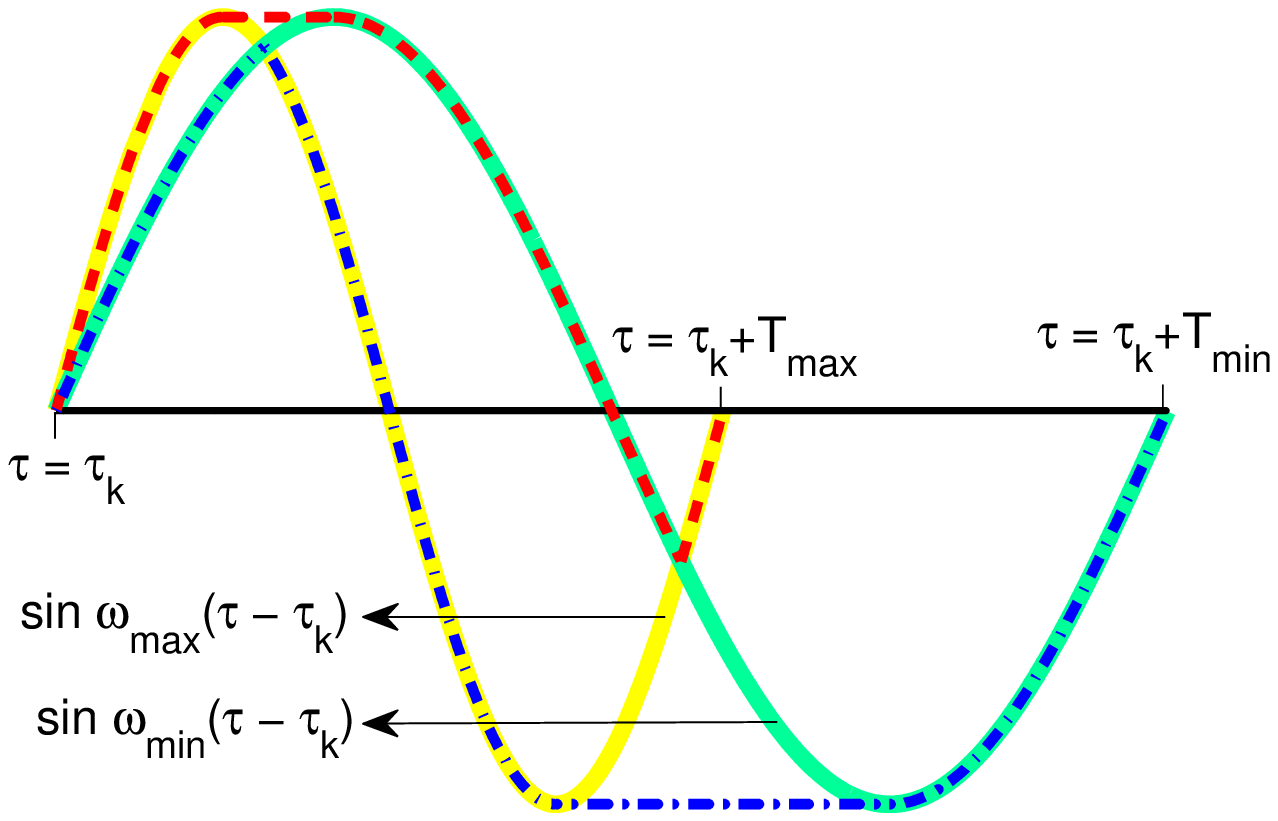} \vspace{-10mm}
\end{center}
\centerline{ \parbox{5.3in}{\vspace{2mm}
Figure 5. The function $\sin(\o_{\max}(\tau-\tau_k))$ on the interval
$[\tau_k,\tau_k+T_{\max}]$ is plotted in yellow, while $\sin(\o_{\min}
(\tau-\tau_k))$ on the interval $[\tau_k,\tau_k+T_{\min}]$ is plotted
in green. Here $T_{\max}=2\pi/\o_{\max}$ and $T_{\min}=2\pi/\o_{\min}
$. The dashed line in red is the function $g_k$ used in the proof
of Lemma \ref{bound} to obtain an upper bound for $P(s)$. The dash-dot
line in blue is the function $h_k$ used in the same proof to obtain a
lower bound for $P(s)$.\vspace{3mm}}}}

\noindent
For each $k<n$, by letting $\tau=\tau_{k+1}$ in \rfb{mean_value} it
follows from the second inequality that $\tau_{k+1}-\tau_k \geq T_
{\max}$ and from the first inequality that $\tau_{k+1}-\tau_k \leq
T_{\min}$. Therefore \vspace{-1mm}
\BEQ{new_step}
   kT_{\max} \m\leq\m \tau_k \m\leq\m kT_{\min} \FORALL
   k\in\{1,2,\ldots n\}\m.
\end{equation}
From \rfb{Pd_newdefn} we have \vspace{-3mm}
\begin{align}
  P(s) \m=\m & p\rho \sum_{k=1}^n \int_{\tau_{k-1}}^{\tau_k} e^{-p
  \rho\tau} \sin\left(\int_{\tau_{k-1}}^{\tau} \o_\rho(s-\sigma)\dd
  \sigma\right) \dd\tau \m, \nonumber\\[3pt] &\qquad + p\rho\int
  _{\tau_n}^s e^{-p\rho\tau}\sin\left(\int_{\tau_n}^{\tau} \o_\rho
  (s-\sigma)\dd\sigma\right)\dd\tau \m.\label{Pd_as_sum}
\end{align} \vspace{-4mm}

For each $k\in\{0,1,\ldots n\}$, define the function $g_k$ on the
interval $[\tau_k,\tau_k+T_{\max}]$ by $g_k(\tau)=g(\tau-\tau_k)$
(see Figure 5). Using $\o_{\min}\leq\o_{\max}\leq2\o_{\min}$ and
\rfb{mean_value} it can be verified that for each $k\in\{0,1,\ldots
n\}$ and all $\tau_k\leq\tau\leq\min\{\tau_k+T_{\max},s\}$,
\BEQ{below1_ub}
  \sin\left(\int_{\tau_k}^{\tau}\o_\rho(s-\sigma)\dd\sigma\right)\dd
  \tau \m\leq\m g_k(\tau)
\end{equation}
and when $\min\{\tau_k+T_{\max},s\}<\tau\leq\min\{\tau_{k+1},s\}$
(if $k=n$, then let $\tau_{k+1}=s$)
\BEQ{below2_ub}
   \sin\left(\int_{\tau_k}^{\tau}\o_\rho(s-\sigma)\dd\sigma\right)\dd
   \tau \m\leq\m 0 \m.
\end{equation}
Using \rfb{below1_ub} and \rfb{below2_ub}, we obtain from
\rfb{Pd_as_sum} that
\begin{align*}
   P(s) &\m\leq\m p\rho\sum_{k=1}^n\int_{\tau_{k-1}}^{\tau_{k-1}+
   T_{\max}} e^{-p\rho\tau}g_{k-1}(\tau)\dd\tau + p\rho\int_{\tau_n}
   ^{\min\{s,\tau_n+T_{\max}\}} e^{-p\rho\tau}g_n(\tau)\dd\tau \\
   &\m=\m p\rho \sum_{k=1}^n e^{-p\rho\tau_{k-1}} \int_0^{T_{\max}}
   e^{-p\rho\tau} g(\tau)\dd\tau + p\rho e^{-p\rho\tau_n}\int_0^{\min
   \{s-\tau_n,T_{\max}\}} e^{-p\rho\tau} g(\tau) \dd\tau \m.
\end{align*}
By letting $\tau=s$ and $k=n$ in \rfb{mean_value}, we get that
$s-T_{\min} <\tau_n$. This, and the inequality $|g(\tau)|\leq1$ for
all $\tau\in[0,T_ {\max}]$, means that the second term on the right
side of the above expression can be bounded in absolute value by
$e^{-p\rho(s-T_{\min})}$.  From the above expression, using the easily
verifiable fact $\int_0^ {T_{\max}}e^{-p\rho\tau} g(\tau)\dd\tau>0$
and the inequalities in \rfb{new_step}, the upper bound in
\rfb{UB_eqn} follows.

For each $k\in\{0,1,\ldots n\}$, define the function $h_k$ on the
interval $[\tau_k,\tau_k+T_{\min}]$ by $h_k(\tau)=h(\tau-\tau_k)$ (see
Figure 5). Using $\o_{\min}\leq\o_{\max}\leq2\o_{\min}$ and
\rfb{mean_value} it can be verified that for each $k\in\{0,1,\ldots
n\}$ and all $\tau_k \leq\tau\leq\min\{\tau_{k+1},s\}$ (if $k=n$,
then let $\tau_{k+1}=s$), \vspace{-1mm}
\BEQ{below1_lb}
  \sin\left(\int_{\tau_k}^{\tau}\o_\rho(s-\sigma)\dd\sigma\right)\dd
  \tau \m\geq\m h_k(\tau)
\end{equation}
and when $\min\{\tau_{k+1},s\}<\tau\leq\min\{\tau_k+T_{\min},s\}$,
$h_k(\tau)\leq0$. Using this and \rfb{below1_lb}, we obtain from
\rfb{Pd_as_sum} that
\begin{align*}
   P(s) &\m\geq\m p\rho\sum_{k=1}^n\int_{\tau_{k-1}}^{\tau_{k-1}+
   T_{\min}} e^{-p\rho\tau} h_{k-1}(\tau)\dd\tau + p\rho\int_{\tau_n}
   ^s e^{-p\rho\tau} h_n(\tau) \\
   &\m=\m p\rho \sum_{k=1}^n e^{-p\rho\tau_{k-1}} \int_0^{T_
  {\min}} e^{-p\rho\tau}h(\tau)\dd\tau + p\rho e^{-p\rho\tau_n}
  \int_0^{s-\tau_n} e^{-p\rho\tau}h(\tau) \dd\tau \m.
\end{align*}
Using $s-T_{\min}<\tau_n$ (shown earlier) and the inequality
$|h(\tau)| \leq 1$ for all $\tau\in[0,T_{\min}]$, it is easy to see
that the second term on the right side of the above expression can be
bounded in absolute value by $e^{-p\rho(s-T_{\min})}$. From the above
expression, using the inequalities in \rfb{new_step}, the lower bound
in \rfb{LB_eqn} follows.
\end{proof}


The next theorem is the main result of this paper. It presents
checkable conditions for the almost global asymptotic stability of the
SG model \rfb{eq:SG}. Recall the notations $V_r$ and $P_\infty$ from
\rfb{Vr_rho}, $\Gamma$ from \rfb{Gamma_defn} and $p=R_s/L_s$. As
discussed earlier, the conditions are specified in terms of a
nonlinear map $\Nscr:(0,\Gamma]\to[0,\infty)$. We show that if the
graph of this map is below the graph of $F(x)=x$, then the SG model in
\rfb{eq:SG} is almost globally asymptotically stable.

\begin{theorem} \label{stab_cond2}
Consider the SG model \rfb{eq:SG}. For each $d\in(0,\Gamma]$ let
$\o_{\max}^d=\o_p+\rho\o_g$ and \m $\o_{\min}^d=\o_n+\rho\o_g$, where
$\o_p$ and $\o_n$ are obtained from \rfb{velocity_bound} using
$\alpha$ and $\beta$ given in \rfb{alpha_beta} and $\rho$ is given
by \rfb{Vr_rho}. Then $\o_{\min}^d<\o_{\max}^d$ for all $d$. If
\m$\o_{\max}^d\leq2\o_{\min}^d$, then recall the functions $g$ and
$h$ from \rfb{g_defn} and \rfb{h_defn}, where we take $\o_{\max}=
\o_{\max}^d$ and $\o_{\min}=\o_{\min}^d$ so that $T_{\max}=2\pi/
\o_{\max}^d$ and $T_ {\min}=2\pi/\o_{\min}^d$, and define the
numbers \vspace{-1mm}
$$  P_u^d \m=\m \frac{p\rho}{1-e^{-p\rho T_{\max}}} \int_0^{T_
    {\max}} e^{-p\rho\tau} g(\tau)\dd\tau \m, \qquad P_l^d \m=
    \m \frac{p\rho}{1-e^{-p\rho T}} \int_0^{T_{\min}} e^{-p\rho
    \tau} h(\tau)\dd\tau \m, \vspace{-1mm}$$
where $T=T_{\max}$ if $\int_0^{T_{\min}} e^{-p\rho\tau}h(\tau)\dd
\tau<0$ and $T=T_{\min}$ otherwise. If \m $\o_{\max}^d>2\o_{\min}^d$, then
let $P_u^d=1$ and $P_l^d=0$. Define the map $\Nscr:(0,
\Gamma]\to\rline$ by \vspace{-1mm}
\BEQ{Defn_Phi}
   \Nscr(d) \m=\m V_r\max\{P_u^d-P_\infty, P_\infty-P_l^d\} \m.
   \vspace{-1mm}
\end{equation}
If $\Nscr(d)<d$ and $\o_{\max}^d\leq2\o_{\min}^d$ for each $d\in(0,
\Gamma]$, then every trajectory of the SG model \rfb{eq:SG} converges
to an equilibrium point. In addition, if all the equilibrium points of
\rfb{eq:SG} are hyperbolic, then \rfb{eq:SG} is almost globally
asymptotically stable.
\end{theorem}

\begin{proof}
Throughout this proof we assume that $\Nscr(d)<d$ and $\o_{\max}^d\leq2
\o_{\min}^d$ for each $d\in(0,\Gamma]$ and that $\alpha$ and $\beta$ are
as in the theorem. First we claim that for any solution $(i_d,i_q,\omega,
\delta)$ of \rfb{eq:SG}, if the corresponding solution $\psi$ of
\rfb{SG_finalform} is such that $\gamma$ in \rfb{gamma_defn} satisfies
$\limsup|\gamma(s)|=0$, then $(i_d,i_q,\omega,\delta)$ converges to an
equilibrium point. Indeed, for each small $d$, $\Nscr(d)<d$ implies
that $|P_u^d-P_l^d|$ and hence (using the definitions of $g$ and $h$
and the fact that $\o_{\min}^d<\o_{\max}^d\leq(1+|\beta|+d)/\alpha+\rho
\o_g$) $\o_{\max}^d-\o_{\min}^d$ are both small, proportional to $d$.
This, and the observation using \rfb{velocity_bound} that if
\rfb{assumption} does not hold for a $d$, then  $\o_{\max}^d-
\o_{\min}^d>2/\alpha$, implies that \rfb{assumption} holds for all
sufficiently small $d$. Therefore, given a $\psi$ as above satisfying
$\limsup|\gamma(s)|=0$, we can apply Theorem \ref{main_result_sec5}
and Remark \ref{IC_infl} to \rfb{SG_finalform} and conclude that for
every $d$ sufficiently small there exists a $T_d>0$ such that
\vspace{-2mm}
\BEQ{Jer_conf}
  |\psi(s_1)-\psi(s_2)| \m<\m \psi_1-\psi_2+{4d}/{\alpha^2}
   \FORALL s_1,s_2>T_d  \m. \vspace{-2mm}
\end{equation}
Here $\psi_1$ and $\psi_2$ are as in \rfb{t1-t2}. Since $d>0$
can be arbitrarily small, \rfb{Jer_conf} implies that $\lim_{s
\to\infty}\psi(s)=\psi_l$ for some finite $\psi_l$. It follows
from \rfb{SG_finalform} that, since $|\psi'(s)|\leq|\psi'(0)|+
(|\beta|+\Gamma+1)/\alpha+\|f\|_{L^\infty}/(i_v\alpha)$ for all
$s\geq 0$ (by the argument in \rfb{JF_election}), $\psi''$ is a
continuous bounded function of time. We can therefore apply
Barb\u alat's lemma to $\psi'$ to conclude that $\lim_{s\to\infty}
\psi'(s)=0$. Since $\eta(\rho s)=\psi(s)$ for all $s\geq0$ by
definition, it follows using \rfb{psi_defn} that $\lim_{t\to
\infty}\delta(t)=\delta_l=\psi_l-\phi-3\pi/2$ and $\lim_{t\to\infty}
\dot\delta(t)=0$. This, along with \rfb{smaller_K}, gives that
$\lim_{t\to\infty}\o(t)=\o_g$. Note that both $i_d$ and $i_q$ are
bounded functions on $[0,\infty)$. This follows from the discussion
at the beginning of Section \ref{sec3} (if either $|i_d(t)|$, $|i_q
(t)|$ or $|\o(t)|$ is sufficiently large, then $\dot W(t)<0$ which
ensures that $i_d$, $i_q$ and $\o$ are bounded). Hence
\rfb{eq:currents} can be rewritten as a second-order exponentially
stable linear system driven by an input which is a sum of a constant
vector and a vanishing vector as follows:
\begin{align*}
  \bbm{\dot{i_d}\\ \dot{i_q}} \m=\m& \bbm{-p & \o_g\\ -\o_g & -p}
  \bbm{i_d\\ i_q} + \frac{1}{L_s}\bbm{V\sin\delta_l\\ V\cos\delta_l
  -mi_f\o_g} \\[3pt] &+\frac{1}{L_s}\bbm{L_s(\o-\o_g)i_q+V(\sin\delta
  -\sin\delta_l) \\ -L_s(\o-\o_g)i_d+V(\cos\delta-\cos\delta_l)-mi_f
  (\o-\o_g)} \m.
\end{align*}
The second term on the right side of the above equation is a
constant, while the third term decays to zero asymptotically
(because $\lim_{t\to\infty}\o(t)=\o_g$ and $\lim_{t\to\infty}\delta(t)
=\delta_l$). This means that $i_d$ and $i_q$ converge to some constants
$i_{d,l}$ and $i_{q,l}$, i.e. \vspace{0mm}
$$ \lim_{t\to\infty} (i_d(t),i_q(t),\o(t),\delta(t)) \m=\m (i_{d,l},
   i_{q,l},\omega_g,\delta_l) \m.\vspace{-1mm}$$
It is now easy to verify using \rfb{eq:SG} that $(\ddot i_d,\ddot i_q,
\ddot\o,\ddot\delta)$ are bounded continuous functions. Therefore, by
applying Barb\u alat's lemma to $(\dot i_d,\dot i_q,\dot\o,\dot\delta)$,
we can conclude that $\lim_{t\to\infty}(\dot i_d(t),\dot i_q(t),\dot\o
(t),\dot\delta(t))=0$ and so $(i_{d,l},i_{q,l},\o_g,\delta_l)$ is an
equilibrium point for \rfb{eq:SG}. This completes the proof of our
claim.

Next we show that for each solution $(i_d,i_q,\omega,\delta)$ of
\rfb{eq:SG}, the corresponding solution $\psi$ of \rfb{SG_finalform}
is such that $\gamma$ in \rfb{gamma_defn} satisfies $\limsup|\gamma(s)
|=0$. This and the claim established above imply that every trajectory
of \rfb{eq:SG} converges to an equilibrium point which in turn implies,
using Lemma \ref{stable_manifold}, that \rfb{eq:SG} is almost globally
asymptotically stable whenever all its equilibrium points are
hyperbolic. Below, we will use the fact that the nonlinear function
$\Nscr$ is right-continuous if $\o_{\max}^d\leq2\o_{\min}^d$ for all
$d\in(0,\Gamma]$. This follows from two (easily verifiable) facts:
(1) $P_u^d$ and $P_l^d$, which are determined by $\o_{\max}^d$ and
$\o_{\min}^d$ using the functions $g$ and $h$, depend continuously on
$\o_{\max}^d$ and $\o_{\min}^d$ and (2) $\o_p$ and $\o_n$ (defined in
\rfb{velocity_bound}), and consequently $\o_{\max}^d$ and $\o_{\min}^d$,
are right-continuous functions of $d$.

Consider a solution $(i_d,i_q,\omega,\delta)$ of \rfb{eq:SG} and the
corresponding solution $\psi$ of \rfb{SG_finalform}. We will show
that $d_0=\limsup|\gamma(s)|=0$. Suppose that $d_0\neq0$. It follows
from \rfb{gamma_defn} and \rfb{Pd_defn} that $d_0\leq\Gamma$. In fact,
$d_0<\Gamma$ since $\limsup|P(s)|<1$. The latter inequality is a
consequence of two simple facts: (i) the integrand in \rfb{Pd_newdefn}
is 0 when $\tau=0$ and (ii) $\psi'$ is a bounded function on $[0,\infty)$
(as stated below \rfb{Jer_conf}) and therefore so is $\o_\rho=\psi'+\rho
\o_g$. Fix $\e>0$ such that $d_0+\e<\Gamma$. Let $d=d_0+\e$ and define
$\o_{\max}=\o_{\max}^d$ and $\o_{\min}=\o_{\min}^d$. Lemma
\ref{velocity_trap} gives that there exists $T_1>0$ such that for all
$s\geq T_1$, $\o_{\min}<\o_\rho(s)<\o_{\max}$. From \rfb{Pd_newdefn}
we get that for all $s\geq T_1$, \vspace{-1mm}
\begin{align}
  P(s) &\m=\m p\rho \int_0^{\tilde s} e^{-p\rho\tau} \sin\left(
  \int_0^\tau\tilde\o_\rho(\tilde s-\sigma) \dd\sigma\right) \dd\tau
  \nonumber \\ &\qquad+ p\rho \int_{\tilde s}^s e^{-p\rho\tau} \sin\left(
  \int_0^\tau\o_\rho(s-\sigma) \dd\sigma\right) \dd\tau \m,
  \label{new_not}
\end{align}
where $\tilde s=s-T_1$ and $\tilde\o_\rho(\tilde s-\sigma)=\o_\rho(s-
\sigma)$ for all $0\leq\sigma\leq\tilde s$. The second term on the right
side of \rfb{new_not} decays exponentially to zero as $s\to\infty$.
Denote the first term on the right side of \rfb{new_not} by $\tilde
P(\tilde s)$. It is easy to see that $\o_{\min}<\tilde\o_\rho(\tau)<
\o_{\max}$ for all $0\leq\tau\leq\tilde s$. Using this and the assumption
$\o_{\max}\leq2\o_{\min}$, we can apply the bounds \rfb{UB_eqn} and
\rfb{LB_eqn} derived in Lemma \ref{bound} for $P(s)$ to $\tilde P
(\tilde s)$ to conclude that \vspace{-1mm}
$$ P_l^d-e^{-p\rho(\tilde s-T_{\min})} \m\leq\m \tilde P(\tilde s)
   \m\leq\m P_u^d + e^{-p\rho(\tilde s-T_{\min})} \m, \vspace{-1mm} $$
where $T_{\min}=2\pi/\o_{\min}$. This, together with \rfb{new_not},
implies that $P_l^d\leq\liminf P(s)\leq\limsup P(s)\leq P_u^d$. It now
follows from \rfb{gamma_defn} that $d_0\leq V_r\max\{P_u^d-P_\infty,
P_\infty-P_l^d\}=\Nscr(d_0+\e)$. Thus we have shown that $d_0\leq
\Nscr(d_0+\e)$ for all $\e>0$ satisfying $d_0+\e<\Gamma$ which, due to
the right-continuity of $\Nscr$, implies that $d_0\leq\Nscr(d_0)$. If
$d_0\neq0$, this contradicts our assumption that $\Nscr(d)<d$ for all
$d\in(0,\Gamma]$. Hence $d_0=0$.
\end{proof}


In general, the conditions in the above theorem are hard to verify
analytically, but it is straightforward to verify them numerically.
This will be demonstrated using an example in the next section.

\begin{remark} \label{compare_stab_cond}
In Section \ref{sec3} we showed that \rfb{eq:SG} has two sequences of
equilibrium points if and only if the right side of \rfb{fsquared2},
denoted as $\Lambda$, satisfies $|\Lambda|<1$. It is easy to check
that $\beta=-\Lambda$ ($\beta$ is defined in \rfb{alpha_beta}). The
condition $\Nscr(d)<d$ for all $d\in(0,\Gamma]$ in Theorem
\ref{stab_cond2} implies that $|\beta|<1$. Indeed, if $|\beta|\geq1$,
then irrespective of $d$, $\o_p-\o_n>2/\alpha$ and so $\o_{\max}^d-
\o_{\min}^d>2/\alpha$. This means that even when we let $d\to0$, $P_u^d-
P_l^d$ will remain bounded away from 0 which, along with \rfb{Defn_Phi},
implies that $\Nscr(d)>d$ for small $d$.
\end{remark}

\begin{remark} \label{CDC_thm}
In \cite[Theorem 5.1]{NaWe:14} we presented a simple set of
conditions, which can be easily verified analytically, under which
\rfb{eq:SG} is aGAS. These conditions were derived using the standard
form \rfb{SG_finalform}--\rfb{Pd_defn} of the ESE. They were stated in
\cite{NaWe:14} under the assumptions that $0<\beta<1$,
$\|\gamma\|_{L^\infty}<d<\beta$ for some $d>0$ and $\beta+d<1$,
because in that work the asymptotic bounds for the forced
pendulum equation were derived under these assumptions. In Section
\ref{sec5}, we have derived the same asymptotic bounds under the less
restrictive assumptions $|\beta|<1$, $\limsup|\gamma(s)|<d$ and
$|\beta|+d<1$ for some $d>0$. Hence the conclusions of \cite[Theorem
5.1]{NaWe:14} continue to hold under these less restrictive
assumptions as well. For the simple conditions of that theorem to hold
$V_r$ must be small (much less than 1), but for nominal SG parameters
typically $V_r>1$. Nevertheless, that theorem enabled us to identify a
large range of (not necessarily practical) SG parameters for which
\rfb{eq:SG} is aGAS. For instance, given a set of SG parameters, if we
increase $V$ by a factor of $n$ and decrease $J$ by the same factor,
then for all $n$ sufficiently large the simple stability conditions
will hold. 
\end{remark}

For the sufficient stability conditions in Theorem \ref{stab_cond2}
to hold, it is necessary that $|\beta|<1$ (see Remark
\ref{compare_stab_cond}) and it is desirable that the damping
coefficient $\alpha$ be large. Indeed, for any given $d>0$ it follows
from \rfb{velocity_bound} that $|\o_p|$ and $|\o_n|$ are inversely
proportional to $\alpha$ and it follows from the definitions of $g$,
$h$, $\Nscr$ in \rfb{g_defn}, \rfb{h_defn} and \rfb{Defn_Phi},
respectively, that $\Nscr(d)$ is proportional to $\max\{|\o_p|,|\o_n|
\}$. Hence for any given $\beta$ and $\Gamma$ with $|\beta|<1$, if
$\alpha$ is sufficiently large, then $\Nscr(d)<d$ and $\o_{\max}^d\leq2
\o_{\min}^d$ for all $d\in(0,\Gamma]$, i.e. the conditions of Theorem
\ref{stab_cond2} will hold. On an intuitive level, the need for large
$\alpha$ and $|\beta|<1$ for global asymptotic stability can be
anticipated from the ESE \rfb{SG_finalform}. It is easy to see from
\rfb{alpha_beta} that
$$ \alpha \m=\m \frac{D_p\sqrt{L_s^\m}\sqrt[4]{p^2\sbluff + \o_g^2}}
   {\sqrt{mi_f VJ}} \m,\qquad \beta \m=\m \frac{T_a L_s
   \sqrt{p^2\sbluff+\o_g^2}}{mi_f V} - \frac{mi_f p\m\o_g}{V\sqrt
   {p^2\sbluff+\o_g^2}} \m,\vspace{-2mm}$$
where, as usual, $p=R_s/L_s$ and $T_a=T_m-D_p\o_g$ is the actual
mechanical torque. The next two remarks contain suggestions for
choosing the parameters of the SG model to increase $\alpha$, so
that the sufficient conditions of Theorem \ref{stab_cond2} are
satisfied. These suggestions may be useful for designing a
synchronverter.

\begin{remark} \label{Aug12}
Let $P_n$ denote the nominal power of the SG. Then $T_a=P_n/\o_g$
and so it is independent of $D_p$. Clearly, by increasing $D_p$ or
decreasing $J$, $\alpha$ can be increased without changing $\beta$.
We can also increase $\alpha$ by increasing $L_s$, but this will
result in an increase in $\beta$ as well. The constraint $|\beta|<1$
gives an upper limit for $L_s$. Similarly, increasing $p$ also
increases $\alpha$ (to a smaller extent since typically $p<\o_g$).
Again the restriction $|\beta|<1$ imposes an upper bound on the
possible values for $p$.

Our numerical experiments with the stability conditions of Theorem
\ref{stab_cond2} suggest that a smaller value for $|\beta|$ is
preferable. Let $\beta_1$ and $\beta_2$ be the first and
second terms, respectively, in the above expression for $\beta$. For
typical SG parameters $p<<\o_g$. Hence increasing $p$ (in a certain
range) will scale up $\beta_2$ more than it does $\beta_1$. On the
other hand, when we increase $L_s$, $\beta_1$ increases while $\beta_2$
remains constant. Since $\beta=\beta_1-\beta_2$, it is possible to
increase $p$ and $L_s$ simultaneously such that $\beta$ remains small.
Increasing $p$ and $L_s$ increases $\alpha$, which is desirable.
\end{remark}

\begin{remark} \label{other_factors}
Modifying any SG parameter other than $D_p$ to increase $\alpha$,
while keeping $\beta$ small, will change at least one of $V_r$, $p$
and $\rho$. Also, when $L_s$ is increased we must increase $mi_f$,
for instance according to \rfb{mif_inc}. Thus when we increase
$\alpha$, we cause an unintentional change in the value of the
function $\Nscr$ corresponding to the variations that we induce in
the values of $V_r$, $p$, $\rho$ and $mi_f$. Our numerical studies
indicate that the desired changes in $\Nscr$ caused by increasing
$\alpha$ are often far more significant than these unintentional
changes. Thus increasing $\alpha$ according to Remark \ref{Aug12}
typically helps to satisfy the stability conditions of Theorem
\ref{stab_cond2}. This is demonstrated using an example in the next
section.
\end{remark}

\section{\secp Application and examples} \label{sec7} 

\ \ \ In this section, on the basis of the results of Section
\ref{sec6}, we propose a modification to the design of synchronverters
to enhance their global stability properties (by artificially
enlarging the filter inductors). Using this modification, we choose
the main parameters of a 500kW synchronverter so that they satisfy the
stability conditions of Theorem \ref{stab_cond2}, for a suitably
chosen constant field current. Thus for this set of parameters
\rfb{eq:SG} is almost globally asymptotically stable.

To select a set of nominal parameters for a synchronverter (without
using our modification), we follow the empirical guidelines used in
the design of SGs and commercial inverters. As usual, let $\o_g$ be
the grid frequency and let $V$ be the line voltage so that the rms
voltage on each phase is $V_{\rm rms}=V/\sqrt{3}$. Let $P_n$ denote
the nominal active power supplied by a SG. Then the nominal active
mechanical torque generated by its prime mover is $T_a=P_n/\o_g$.
Following empirical guidelines, the moment of inertia $J$ of the SG
rotor is chosen so that $(J\o_g^2/2)/P_n$ lies between 2 and 12
seconds. The frequency droop constant $D_p$ is selected such
that if the SG rotor frequency drops below the nominal grid frequency
by $d_p\%$ of $\o_g$, then the active power should increase by the
amount $P_n$. Thus, $D_p=100P_n/(d_p\o_g^2)$ (typically $d_p\approx
3$). By definition $T_m=T_a+D_p\o_g$. To compute $L_s$ and $R_s$ we
assume that in steady state the stator current $i$ and the grid
voltage $v$ are in phase. Then $P_n=3V_{\rm rms}I_{\rm rms}$, where
$I_{\rm rms}$ is the nominal rms value of the current on each
phase. In commercial inverters the inductance $L_s$ of the filter
inductor is chosen so that the voltage drop across $L_s$, given
by $L_s\o_g I_{\rm rms}$, is $3-5\%$ of $V_{\rm rms}$. The resistance
$R_s$ of the filter inductor is normally such that the rms voltage
drop across $R_s$ is below $0.5\%$ of $V_{\rm rms}$. The expression
$mi_f=\sqrt{\frac{3}{2}} M_fi_f$ is determined by \rfb{sync_burned}
(with $\o=\o_g$) and the condition \vspace{-2mm}
\BEQ{mif_inc}
   e_{\rm rms} \m\approx\m \sqrt{V_{\rm rms}^2+\o_g^2 L_s^2 I_{\rm
   rms}^2} \ , \vspace{-1mm}
\end{equation}
where $e_{\rm rms}$ is the rms of the electromotive force $e$
in each phase at steady state.

We have briefly introduced synchronverters in Section \ref{sec1},
but so far we have not described their structure. Without going
into too much detail, the synchronverter is based on an inverter
having three legs built from electronic switches which operate at a
high switching frequency, see \cite{Zh_etal:14,ZhWe:11} for details.
It has a DC side which is normally connected to a DC energy source (or
a storage device), three AC output terminals corresponding to the
three phases of the power grid and a neutral line (which serves as
reference for all the voltages).  We denote the vector of voltages on
the AC terminals, averaged over one switching period, by $g=[\m g_a\ \
g_b\ \ g_c]^\top$. These AC terminals are connected to passive
low-pass filters, each of which may be an inductor, or two inductors
and a capacitor (the so-called LCL filter) or they may have a more
complicated structure. The purpose of these filters is to transfer the
power from the inverter to the grid while eliminating the voltage and
current ripples at the switching frequency and its higher harmonics.
If there is an LCL filter, then for the purpose of modeling, we
neglect the capacitor and approximate the filter with a single
inductor whose inductance $L_s$ is the sum of the two inductances in
the circuit. (The same goes for the series resistances of these
inductors.) This is justified because, up to the grid frequency, the
impedance of the capacitor is much larger (in absolute value) than the
impedances of the inductors.

For a synchronverter designed as in \cite{ZhWe:11}, the voltages
$g_a$, $g_b$ and $g_c$ represent the synchronous internal voltages in
the stator windings of the virtual synchronous generator, while an LCL
filter represents the inductance $L_s$ and the series resistance $R_s$
of the stator windings (by ignoring the capacitor). As discussed in
Section \ref{sec1}, the rotor dynamics is implemented in software.
According to the design in \cite{Zh_etal:14,ZhWe:11}, $g=e$, where $e$
is computed by the synchronverter algorithm using the measured stator
currents and the equations \rfb{sync_burned} and
\rfb{Arye_Levinson_Karat}. This $e$ is then provided to the stator
coil, as depicted in Figure 6 with $n=1$. We now think that choosing
$g=e$ is not the best approach, because the inductance $L_s$ is far
too small.  Indeed, for reasons of size and cost, the inductance of
the filter inductor of a typical commercial inverter is usually much
smaller than the stator inductance of a SG of the same power rating
(about 50 times smaller), see for instance \cite[Example 3.1]{Kun:94}.
This fact alone justifies increasing $L_s$ artificially, in order to
make the synchronverter more similar to a SG. There is an additional
reason for increasing $L_s$ artificially, and this is to improve the
stability of the system, as explained below.

\begin{center} \vspace{-5mm}
\includegraphics[scale=0.12]{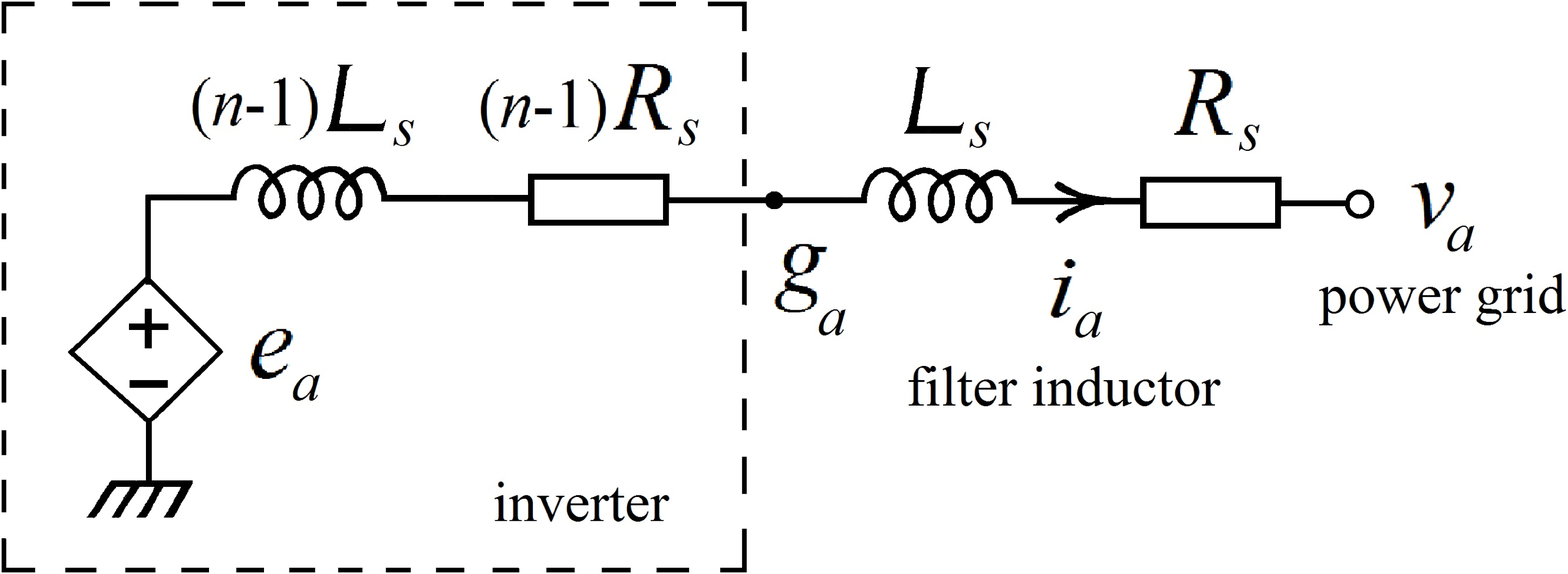} \vspace{-5mm}
\end{center}
\centerline{ \parbox{5.3in}{\vspace{3mm}
Figure 6. An inverter operated as a synchronverter, with filter
inductor $L_s$ and its series resistance $R_s$, connected to the
utility grid. Only phase $a$ is shown. The synchronverter algorithm
provides the synchronous internal voltage $e_a$ according to
\rfb{sync_burned}. The inductor and resistor multiplied with $(n-1)$
are virtual. For $n=1$ we get a usual synchronverter as in
\cite{ZhWe:11}. The actual (short time average) voltage generated by
the inverter is $g_a$.\vspace{3mm}}}

The parameters of synchronverters selected according to the empirical
guidelines described earlier typically do not satisfy the stability
conditions of Theorem \ref{stab_cond2}. But as discussed in Section
\ref{sec6} (above Remark \ref{Aug12}), if we increase the damping
factor $\alpha$ to a sufficiently large value, then the stability
conditions will hold. We see from Remark \ref{Aug12} that $\alpha$ can
be increased primarily by increasing $D_p$, decreasing $J$ or
increasing $L_s$. In a synchronverter (in normal operation), the
parameters $D_p$ and $J$ are chosen based on grid requirements
(standard droop behavior and inertia) as discussed earlier and we
cannot change them to increase $\alpha$. Thus, the only way to
increase $\alpha$ is via increasing $L_s$. Replacing the existing
inductor with a much larger one (designed for the same nominal current,
of course) would be very expensive and the larger inductor would be very
bulky. We propose a method to virtually increase the inductance $L_s$
of the inductor (and also its series resistance $R_s$) by a factor
$n$ (for instance, $n=30$), by only changing the synchronverter
control algorithm. The idea is to create a virtual inductor of value
$(n-1)L_s$ and with series resistance $(n-1)R_s$ in series with the
real inductor, as shown in Figure 6 (which shows only one phase out
of three, phase $a$). We see from the figure (and a trivial
computation) that
\BEQ{gagbgc}
   g_a \m=\m \frac{(n-1) v_a + e_a}{n} \m.
\end{equation}
Here $e_a$ is the synchronous internal voltage given by
\rfb{sync_burned} while $g_a$ is the average voltage (over one
switching period) at the output of the switches in the inverter.
Thus, by enforcing $g$ computed as in \rfb{gagbgc} we create the
effect of providing the synchronous internal voltage $e=[\m e_a\ \
e_b\ \ e_c]^\top$ to stator coils with inductance $nL_s$ and
resistance $nR_s$. By this method we increase the effective
inductance and resistance of the stator coils by a factor of $n$
and $\alpha$ by a factor of $\sqrt{n}$. For $n=1$ we recover the
structure of the synchronverter in \cite{ZhWe:11}.

In the sequel, we present a choice for the main parameters of a 500kW
synchronverter which (after the modification in Figure 6) satisfy the
conditions of Theorem \ref{stab_cond2}. Thus for this set of
parameters the grid-connected SG model \rfb{eq:SG} is almost globally
asymptotically stable (aGAS). We also consider other values for the SG
parameters to illustrate the different types of global dynamic behavior
that the system \rfb{eq:SG} can exhibit.

\begin{example} \label{500kW}
Consider a synchronverter designed for the grid frequency
$\o_g=100\pi\,$rad/sec and line voltage $V=6000\sqrt{3}$ Volts. The
synchronverter supplies a nominal active power $P_n=500\,$kW and
operates with a $3\%$ frequency droop coefficient, i.e. $d_p=3$.
Following the empirical guidelines discussed earlier, we choose
$D_p=168.87\,$N$\cdot$m/(rad/sec), $T_m=54.64\,$kN$\cdot$m,
$J=20.26\,$Kg$\cdot$m$^2$/rad, $L_s=27.5\,$mH, $R_s=1.08\,\Om$ and
$mi_f=33.11\,$Volt$\cdot$sec. For this set of parameters the sufficient
stability conditions in Theorem \ref{stab_cond2} do not hold. Although
the stability conditions can be satisfied by increasing $D_p$ or
decreasing $J$, this is not desirable from an operational
standpoint. So following the discussion earlier in this section we
increase the effective inductance and resistance by a factor of $30$,
i.e. $n=30$ in Figure 6. Then (the effective) $L_s=825.06\,$mH,
$R_s=32.4\,\Om$ and $mi_f=51.67\,$Volt$\cdot$sec.

\vspace{-3mm} \hspace{-10mm}
\centerline{
 \parbox{73mm}{$$\includegraphics[scale=.58]{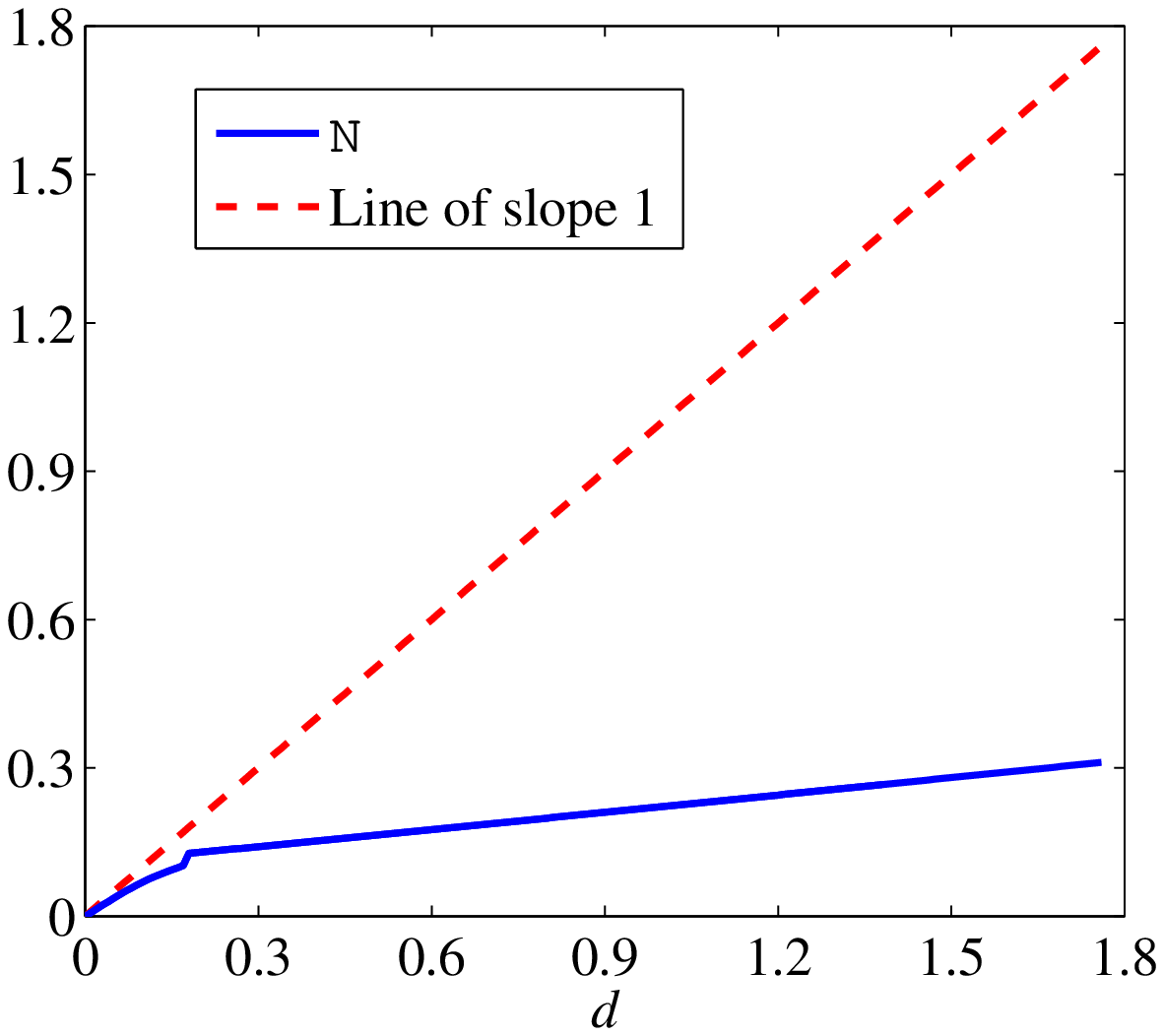}$$
 \hspace{3mm}
 \centerline{\parbox{70mm}{Figure 7. Plot of $\Nscr$ for the SG
 parameters in Example \ref{500kW}.}}}
 \hspace{2mm}
 \parbox{73mm}{$$\includegraphics[scale=.58]{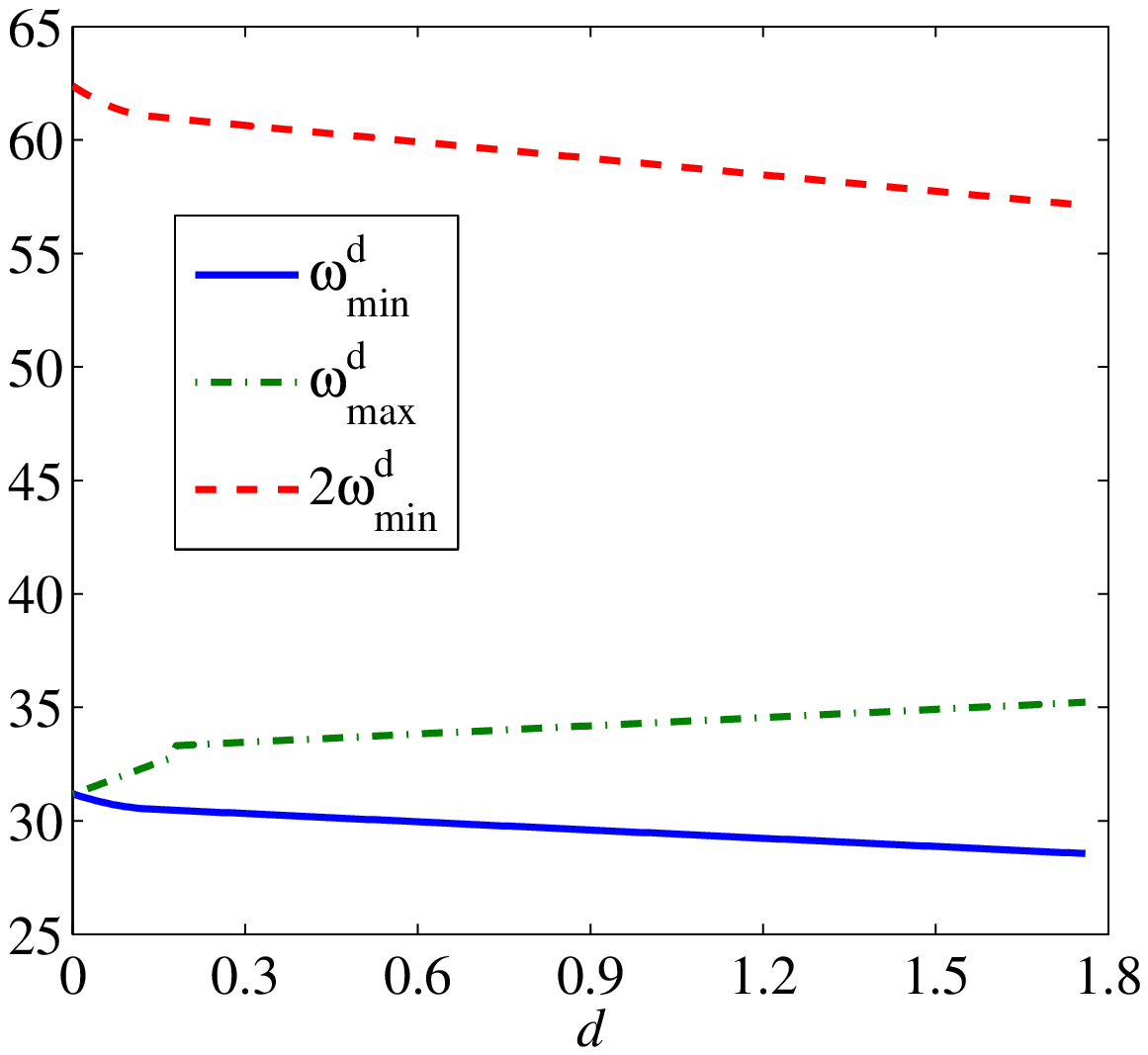}$$
 \hspace{2mm}
 \centerline{\parbox{70mm}{Figure 8. Plots of $\o_{\max}^d$, $\o_{\min}^d$
 and $2\o_{\min}^d$ for the SG parameters in Example \ref{500kW}.}}}}


\medskip
Consider the SG model \rfb{eq:SG} and the corresponding pendulum
system \rfb{SG_finalform} with the above parameter values and
$n=30$. This yields $p=39.27\,$rad/sec, $i_v=39.78\,$A, $\alpha=
0.83\,$sec/rad, $\beta=0.58$, $\rho=0.1\,$s/$\sqrt{\text{rad}}$,
$V_r=1.57$ and $P_\infty=0.12$. For the chosen parameter values,
all the equilibrium points of \rfb{eq:SG}
are hyperbolic. Figure 7 is the plot of the function
$\Nscr$ defined in \rfb{Defn_Phi} on the interval $(0,(1+P_\infty)
V_r]$. Figure 8 is the plot of $\o_{\max}^d$, $\o_{\min}^d$ and
$2\o_{\min}^d$ (as defined in Theorem \ref{stab_cond2}) on the same
interval. It is clear from these figures that the conditions in
Theorem \ref{stab_cond2} are satisfied. Hence the SG model \rfb{eq:SG}
with the virtual inductor is aGAS. \hfill$\square$
\end{example}

Numerical simulations suggest that the SG model \rfb{eq:SG} is aGAS
for the parameter values in the above example even when we take $n=1$
(no virtual inductance). But we cannot prove this since the conditions
of Theorem \ref{stab_cond2} do not hold. In fact it may be hard to
prove this analytically because if we make small changes in the value
of $R_s$ (while keeping all other parameter values same), then the SG
model loses the aGAS property. Indeed, if we increase $R_s$ so that
the voltage drop across the resistor is $1\%$ (instead of $0.5\%$) of
$V_{\rm rms}$, then the SG model is not aGAS. In this case, the SG
model has a sequence of stable and unstable equilibrium points.
It also has a sequence of periodic solutions (two periodic solutions
in this sequence differ only in their value of $\delta$ and the
difference is a multiple of $2\pi$). Along each periodic solution,
$\o<\o_g$ and $\o$, $i_d$ and $i_q$ oscillate with a time period of
about 0.16 seconds while $\delta$ decreases monotonically ($\delta$
is periodic modulo $2\pi$).

Suppose that we choose $D_p=15\,$N\m m/(rad/sec) and let all the other
SG parameters be as in Example \ref{500kW} with $n=1$. Then the SG
model \rfb{eq:SG} has two sequences of equilibrium points, both of
them unstable (this cannot happen in the case of a pendulum
equation with constant forcing). The SG model also has a sequence of
periodic solutions.


\end{document}